\documentclass[a4paper,11pt]{article}
\usepackage{amssymb,amsmath,amsthm}
\usepackage[utf8]{inputenc}
\usepackage[english]{babel}
\usepackage{hyperref}
\usepackage{natbib}
\usepackage{authblk}

\usepackage{graphicx}

\everymath{\displaystyle}

\newcommand{\E}{\mathbb{E}}

\newcommand{\R}{\mathbb{R}}
\newcommand{\N}{\mathbb{N}}

\newcommand{\X}{\mathcal{X}}
\newcommand{\telque}{\mbox{ s.t. }} 
\DeclareMathOperator{\card}{Card} 

\newcommand{\egalloi}{\stackrel{(d)}{=}} 
\newcommand{\loi}{\mathcal{L}} 

\newcommand{\un}{\mathbf{1}}

\newcommand{\paren}[1]{\left( \left. #1 \right. \right)} 
\newcommand{\sparen}[1]{( \left. #1 \right. )} 
\newcommand{\croch}[1]{\left[ \left. #1 \right. \right]} 
\newcommand{\scroch}[1]{[ \left. #1 \right. ]} 
\newcommand{\set}[1]{\left\{ \left. #1 \right. \right\}}

\newcommand{\absj}[1]{\left\lvert #1 \right\rvert} 
\newcommand{\norm}[1]{\left \lVert #1 \right\rVert}
\newcommand{\snorm}[1]{\lVert #1 \rVert}

\newcommand{\egaldef}{:=} 

\newcommand{\flens}{\mapsto} 

\newcommand{\betl}{\beta_{\lambda}}

\theoremstyle{plain}

\newtheorem{proposition}{Proposition}
\newtheorem{prop}[proposition]{Proposition}
\newtheorem{corollary}[proposition]{Corollary}

\newtheorem{lemma}[proposition]{Lemma}
\newtheorem{lem}[proposition]{Lemma}
\newtheorem{algo}{Model}
\theoremstyle{definition}

\theoremstyle{remark}
\newtheorem{remark}{Remark}

\newtheorem{postita}{Post-it}

\newcommand{\Prob}{\mathbb{P}} 

\renewcommand{\P}{\Prob}

\DeclareMathOperator{\var}{var} 
\DeclareMathOperator{\cov}{cov} 

\newcommand{\sachant}{\right| \left.} 

\newcommand{\bU}{\mathbb{U}}
\newcommand{\cU}{\mathcal{U}}
\newcommand{\cUtoy}[1]{\cU_{#1}^{\mathtt{toy}}}
\newcommand{\cUpurf}[1]{\cU_{#1}^{\mathtt{purf}}}
\newcommand{\cUurt}[1]{\cU_{#1}^{\mathtt{bprf}}}

\newcommand{\bV}{\mathbb{V}}

\newcommand{\Bias}{\mathcal{B}}
\newcommand{\Biasinfty}[1]{\Bias_{#1,\infty}}
\newcommand{\Vararbre}[1]{\mathcal{V}_{#1}}

\newcommand{\epstoy}{\varepsilon^{\mathtt{toy}}}
\newcommand{\epspurf}{\varepsilon^{\mathtt{purf}}_{\nfeu}}

\newcommand{\CHdeuxa}{C_2}
\newcommand{\CHtroisa}{C_3}

\newcommand{\PolMun}{P_0}
\newcommand{\PolNdeu}{P_1}
\newcommand{\PolRdeu}{P_2}
\newcommand{\PolRqua}{P_3}

\newcommand{\PolRtro}{P_5}
\newcommand{\PolVarU}{Q}

\newcommand{\RestePURF}{R^{\mathtt{purf}}}
\newcommand{\PolPURFdiff}[1]{\RestePURF_{1,#1}}
\newcommand{\PolPURFprod}[1]{\RestePURF_{2,#1}}
\newcommand{\PolPURFdeu}[1]{\RestePURF_{3,#1}}
\newcommand{\PolPURFtro}[1]{\RestePURF_{4,#1}}
\newcommand{\PolPURFqua}[1]{\RestePURF_{5,#1}}

\newcommand{\narb}{q} 
\newcommand{\dimX}{d} 
\newcommand{\nfeu}{k} 
\newcommand{\prof}{p} 
\newcommand{\nobs}{n} 

\newcommand{\alphaurt}{\alpha} 
\newcommand{\betaurt}{\beta} 

\newcommand{\ERM}{\widehat{s}}
\newcommand{\pl}{p_{\lambda}}
\DeclareMathOperator{\diam}{diam}

\allowdisplaybreaks

\title{Analysis of purely random forests bias}

\author[1]{Sylvain Arlot\thanks{sylvain.arlot@ens.fr}}
\author[2,3]{Robin Genuer\thanks{robin.genuer@isped.u-bordeaux2.fr}}

\affil[1]{CNRS; Sierra Project-Team; 
Departement d'Informatique de l'Ecole Normale Superieure (DI/ENS) 
(CNRS/ENS/INRIA UMR 8548); 
23 avenue d'Italie, CS 81321, 75214 PARIS Cedex 13 - France}
\affil[2]{Univ. Bordeaux, ISPED, Centre INSERM U-897,
146 rue L\'eo Saignat,
F-33076 Bordeaux Cedex,
France}
\affil[3]{INRIA Bordeaux Sud-Ouest, Equipe SISTM}

\begin{document}

\maketitle

\begin{abstract}
  Random forests are a very effective and commonly used statistical
  method, but their full theoretical analysis is still an open
  problem. As a first step, simplified models such as purely random
  forests have been introduced, in order to shed light on the good
  performance of random forests. In this paper, we study the
  approximation error (the bias) of some purely random forest models
  in a regression framework, focusing in particular on the influence
  of the number of trees in the forest. Under some regularity
  assumptions on the regression function, we show that the bias of an
  infinite forest decreases at a faster rate (with respect to the size
  of each tree) than a single tree. As a consequence, infinite forests
  attain a strictly better risk rate (with respect to the sample size)
  than single trees. Furthermore, our results allow to derive a
  minimum number of trees sufficient to reach the same rate as an
  infinite forest. As a by-product of our analysis, we also show a
  link between the bias of purely random forests and the bias of some
  kernel estimators.
\end{abstract}

\section{Introduction} \label{sec.intro}

Random Forests (RF henceforth) are a very effective and increasingly
used statistical machine learning method. They give outstanding
performances in lots of applied situations for both classification and
regression problems. However, their theoretical analysis remains a
difficult and open problem, especially when dealing with the original
RF algorithm, introduced by \cite{Bre:2001}.

Few theoretical results exist on RF, mainly the analysis of Bagging of
\cite{Buh_Yu:2002}---bagging, introduced by \cite{Bre:1996}, can be
seen a posteriori as a particular case of RF--- and the link between
RF and nearest neighbors \citep{Lin_Jeo_2006,Bia_Dev:2010}.

As a first step towards theoretical comprehension of RF, simplified
models such as purely random forests (PRF henceforth) have been
introduced.  Breiman first began to study such simplified RF
\citep{Bre:2000}, and then well-established results were obtained by
\cite{Bia_Dev_Lug:2008}.  The main difference between PRF and RF is
that, in PRF, partitioning of the input space is performed
independently from the dataset, using random variables independent
from the data.  The first reason why it is easier to handle
theoretically PRF is that the random partitioning (associated to a
tree) is thus independent of the prediction made within a given element of the partition. 
Secondly, random
mechanisms used to obtain the partitioning of PRF are usually simple
enough to allow an exact calculation of several quantities of
interest.

In addition to theoretical analysis of PRF models described
below, some empirical studies tried these methods. \cite{Cut_Zha:2001}
compared performances of PERT (PErfect Random Tree ensemble) with
original RF. 
\cite{Geu_Ern_Weh:2006} studied ``Extremely Randomized
Trees'', which are not exactly PRF but lay between standard
RF and PRF. 
These results are encouraging since PRF or ``Extremely
Randomized Trees'' reach very good performances on real datasets. Thus,
understanding such PRF models could give birth to simple but performing RF variants, 
in addition to the original goal of understanding the original RF model.

\subsection{RF and PRF partitioning schemes}
\label{sec:rf-prf-partitioning}
We now precisely define some RF and PRF models, focusing on the regression setting
that we consider in the paper.

Following the usual terminology of RF, in this paper, any partitioning
of the input space $\X \subset \R^d$ is called a tree. Classical
tree-based estimators are related to trees because of the recursive
aspect of the partitioning mechanism. In order to simplify further
discussion, we make a slight language abuse by also calling a tree a
partitioning obtained in a non-recursive way. The leaves of the tree
(its terminal nodes) are the elements of the final partition. Inner
nodes of the tree are also useful for determining (recursively) to
which element of the partition belongs some $x \in \X$, as usual with
decision trees.

Furthermore, as in classical tree-based estimators we focus on
partitions of $\X$ made of hyperrectangles and we denote an
hyperrectangle by $\lambda = \prod_{j=1}^d \lambda_j$ where
$\lambda_1, \ldots, \lambda_d$ are intervals of $\R$.

To each tree corresponds a tree estimator, obtained by assigning a
real number to each leaf of the tree, which is the (constant) value of
the estimator on the corresponding element of the partition.
Throughout the paper, we always consider regressograms, that is, the
value assigned to each leaf is the average of the response variable
values among observations falling into this leaf.

Finally, a forest is a sequence of trees, and the corresponding forest
estimator is obtained by aggregating the corresponding tree
estimators, that is, averaging them.

\medskip

We can now describe precisely some important RF and PRF models. 
Original RF \citep{Bre:2001} are defined as follows.  Each randomized
tree is obtained from independent bootstrap samples of the original
data set, by the following recursive partitioning of the input space,
which is a variant of the CART algorithm of \cite{Bre_etal:1984}:
\begin{algo}[Original RF model] \label{model.original}
\begin{itemize}\hfill\\ \vspace*{-3ex}
\item Put $\X$ at the root of the tree. 
\item Repeat (until a stopping criterion is met), for each leaf
  $\lambda$ of the current tree:
  \begin{itemize}
  \item choose \texttt{mtry} variables (uniformly, and without
    replacement, among all variables),
  \item find the \emph{best} split (i.e., the best couple \{ split
    variable $j$, split point $t$ \}, among all possible ones
    involving the \texttt{mtry} selected variables) and perform the
    split, that is, put $\{ x \in \lambda \, / \, x_j < t\}$ and $\{ x
    \in \lambda \, / \, x_j \geq t\}$ at the two children nodes below
    $\lambda$.
  \end{itemize}
\end{itemize}
\end{algo}
The parameter \texttt{mtry} $\leq \dimX$ is crucial for the method and is fixed for
all nodes of all trees of the forest. The \emph{best} split is found
by minimizing an heterogeneity measure, which is related to some
quadratic risk \citep[see][]{Bre_etal:1984}.

One of the main reasons of the difficulty to theoretically analyze this
algorithm comes from the fact that the partitioning is data-dependent,
and that the same data are used to optimize the partition and
to allocate values to tree leaves.

\bigskip

The first PRF model was introduced in \cite{Bre:2000}. 
In comparison to another model introduced in Section~\ref{sec.multidim.BPRF} (Balanced PRF, BPRF), 
we name it UBPRF (UnBalanced PRF).
The input space is set to $\X = [0, 1)^\dimX$, and the random partitioning
mechanism is the following:
\begin{algo}[UBPRF model]\label{algo.ubprf}\hfill\\ \vspace*{-3ex}
\begin{itemize}
\item Put $[0,1)^{\dimX}$ at the root of the tree.
\item Repeat $\nfeu$ times:
  \begin{itemize}
  \item randomly choose a node $\lambda$, to be splitted, uniformly among all terminal nodes,
  \item randomly choose a split variable $j$ (uniformly among the
    $\dimX$ coordinates),
  \item randomly choose a split point $t$ uniformly over $\lambda_j$
    and perform the split, that is, put $\{ x \in \lambda \, / \, x_j
    < t\}$ and $\{ x \in \lambda \, / \, x_j \geq t\}$ at the two children nodes
    below $\lambda$.
  \end{itemize}
\end{itemize}
\end{algo}
\cite{Bia_Dev_Lug:2008} established a universal consistency result in
a classification framework, for trees and forests associated to this
PRF model, provided that input variables have a uniform distribution
on $[0,1)^\dimX$.

In this paper, we do not study the UBPRF model but we consider a very 
close one in Section~\ref{sec.multidim.BPRF} (BPRF). The only
difference is that at each step, all nodes are split, resulting in
balanced trees.

\bigskip

Assuming $\X=[0,1)$, another PRF model, introduced in \cite{Gen:2012}
and called PURF (Purely Uniformly Random Forests), is obtained by
drawing $\nfeu$ points independently with a uniform distribution
on $[0,1)$, and by taking them as split points for the partitioning. An
equivalent recursive definition of the PURF model is the following:
\begin{algo}[PURF model]\label{algo.purf}\hfill\\ \vspace*{-3ex}
\begin{itemize}
\item Put $[0,1)$ at the root of the tree. 
\item Repeat $\nfeu$ times:
  \begin{itemize}
  \item choose a terminal node $\lambda$, to be splitted, each with a
    probability equal to its length,
  \item choose a split point $t$ uniformly over $\lambda$ and perform
    the split, that is, put $\{ x \in \lambda \, / \, x < t\}$ and $\{
    x \in \lambda \, / \, x \geq t\}$ at the two children nodes below $\lambda$.
  \end{itemize}
\end{itemize}
\end{algo}
Compared to UBPRF, $\dimX=1$ and the probability to choose a terminal
node for being splitted is not uniform but equal to its length.
\cite{Gen:2012} proved for the PURF model the estimation error is
strictly smaller for an infinite forest than for a single tree, and that 
when $\nfeu$ is well chosen, both trees and forests of the PURF
model reach the minimax rate of convergence when the regression
function is Lipschitz.

\subsection{Contributions}
\label{sec:objectives}
This paper compares the performances of a forest estimator and
a single tree, for three PRF models, in the regression framework with
an input space $\X \subset \R^{\dimX}$.  Section~\ref{sec.general}
presents a general decomposition of the quadratic risk of a general
PRF estimator into three terms, which can be interpreted as a
decomposition into approximation error and estimation error.  The rest
of the paper focuses mostly on the approximation error terms.
Section~\ref{sec.multidim} shows general bounds on the approximation
error under smoothness conditions on the regression function.  These
bounds allow us to compare precisely the rates of convergence of the
approximation error and of the quadratic risk of trees and forests for
three PRF models: a toy model (Section~\ref{sec_toy_model}), the PURF
model (Section~\ref{sec.purf}) and the BPRF model
(Section~\ref{sec.multidim.BPRF}).
For all three models, the approximation error decreases to zero as the
number of leaves of each tree tends to infinity, with a faster rate
for an infinite forest than for a single tree. As a consequence, when
the sample size tends to infinity and assuming the number of leaves is
well-chosen, the quadratic risk decreases to zero faster for an
infinite forest estimator than for a single tree estimator.
As a by-product, our analysis provides theoretical grounds for choosing
the number of trees in a forest in order to perform almost as well as
an infinite forest, contrary to previous results on this question that were only
empirical \citep[for instance][]{Lat_Dev_Dec:2001}.
Furthermore, we show a link between the bias of the infinite forest
and the bias of some kernel estimator, which enlightens the different
rates obtained in
Sections~\ref{sec_toy_model}--\ref{sec.multidim.BPRF}.
Finally, our theoretical analysis is illustrated by some simulation
experiments in Section~\ref{sec.simu}, for the three models of
Sections~\ref{sec_toy_model}--\ref{sec.multidim.BPRF} and for another
PRF model closer to original RF.

\subsection*{Notation}
Throughout the paper, $L(*)$ denotes a constant 
depending only on quantities appearing in $*$, 
that can vary from one line to another or even within the same line. 

\section{Decomposition of the risk of purely random forests} 
\label{sec.general}

This section introduces the framework of the paper and provides a
general decomposition of the risk of purely random forests, on which
the rest of the paper is built.

\subsection{Framework} \label{sec.general.framework}

Let $\X$ be some measurable set and $ s : \X \flens \R $ some
measurable function.
Let us assume we observe a learning sample $D_{\nobs} = (X_i,Y_i)_{1
  \leq i \leq \nobs} \in \paren{\X \times \R}^{\nobs}$ of $\nobs$
independent observations with common distribution $P$ such that
\[ \forall i \in \set{1, \ldots, \nobs}\, , \quad \E\croch{Y_i \sachant X_i} = s(X_i) \quad \mbox{and} \quad \E\croch{ \paren{ Y_i - s(X_i)}^2 \sachant X_i} = \sigma^2 > 0 \enspace . \]

The goal is to estimate the function $s$ in terms of quadratic risk.
Let $(X,Y) \sim P$ be independent from $D_{\nobs}$.  Then, the
quadratic risk of some (possibly data-dependent) estimator $t: \X
\mapsto \R$ of $s$ is defined by
\[ \E\croch{ \paren{ s(X) - t(X) }^2 } \enspace . \]

\subsection{Purely random forests} \label{sec.general.forest}
In this paper, we consider random forest estimators which are the aggregation of several tree estimators, that is, several regressograms.

For every finite partition $\bU$ of $\X$, the tree (regressogram) estimator
on $\bU$ is defined by 
\[ 
\forall x \in \X \, , \quad 
\ERM (x;\bU;D_{\nobs}) 
= \ERM (x;\bU)
\egaldef \sum_{\lambda \in \bU} \frac{ \sum_{i=1}^{\nobs} Y_i \un_{X_i
    \in \lambda} } { \card\set{ 1 \leq i \leq \nobs \, / \, X_i \in
    \lambda } }\un_{x \in \lambda} 
    \enspace , \]
with the convention $0/0=0$ for dealing with the case where no $X_i$ belong to some $\lambda\in \bU$. 
Note that $\bU$ can be any partition of $\X$, not necessarily obtained from a decision tree, even if we always call it a tree. 

Let $\narb \geq 1$ be some integer.  Given a sequence $\bV_{\narb}=(\bU^j)_{1
  \leq j \leq \narb}$ of finite partitions of $\X$, the associated
forest estimator is defined by
\[ 
\forall x \in \X \, , \quad 
\ERM (x ; \bV_{\narb};D_{\nobs}) 
= \ERM (x ; \bV_{\narb}) 
= \ERM_{\bV_{\narb}} (x) \egaldef \frac{1}{\narb} \sum_{j=1}^{\narb} \ERM
(x;\bU^j;D_{\nobs}) 
\enspace . \]

This paper considers random forests, that is, for which $\bU^1,
\ldots, \bU^{\narb}$ are independent finite partitions of $\X$ with
common distribution $\cU$.
More precisely, we focus on \emph{purely random forest}, that is, we
assume
\begin{equation}
  \label{hyp.purely-random} \tag{\ensuremath{\mathbf{PR}}}
  \bV_{\narb} = (\bU^1, \ldots, \bU^{\narb}) \mbox{ is independent from the data } D_{\nobs} = (X_i,Y_i)_{ 1 \leq i \leq \nobs} 
  \enspace . 
\end{equation}

\subsection{Decomposition of the risk}
\label{sec.general.risk} 
For any finite partition $\bU$ of $\X$, we define
\[ 
\tilde{s}_{\bU} \egaldef \sum_{\lambda\in\bU} \betl \un_{\lambda}
\quad \mbox{where} \quad
\betl \egaldef \E\croch{ s(X) \sachant X \in \lambda } 
\]
is well-defined for every $\lambda \subset \X$ such that
$\P(X\in\lambda)>0$. So, $\tilde{s}_{\bU}(X)$ is a.s. well-defined.
The function $\tilde{s}_{\bU}$ minimizes the least-squares risk among
functions that are constant on every $\lambda \in \bU$.
For any finite sequence $\bV_{\narb}=(\bU^j)_{1 \leq j \leq \narb}$, we define
\[ 
\tilde{s}_{\bV_{\narb}} \egaldef
\frac{1}{\narb} \sum_{j=1}^{\narb} \tilde{s}_{\bU^j} 
\enspace .\]

\medskip

%
Then, as noticed in \cite{Gen:2012}, assuming
\eqref{hyp.purely-random}, the (point-wise) quadratic risk of $\ERM
(\cdot ; \bV_{\narb})$ can be decomposed as the sum of two terms: for
every $x \in \X$,
\begin{equation} \label{eq.risk-cond.bias+var}
  \E\croch{ \paren{ s(x) - \ERM (x ; \bV_{\narb}; D_{\nobs})}^2 } =
  \E \croch{ \paren{s(x) - \tilde{s}_{\bV_{\narb}}(x)}^2 } +
  \E\croch{ \paren{ \tilde{s}_{\bV_{\narb}}(x) - \ERM (x ; \bV_{\narb}; D_{\nobs})}^2 }
\end{equation}
since
\[ 
\E\croch{ \ERM (x ; \bV_{\narb}; D_{\nobs}) \sachant \bV_{\narb} }
= \tilde{s}_{\bV_{\narb}}(x) 
\enspace . \]

%
Furthermore, the first term in the right-hand side of Eq.~\eqref{eq.risk-cond.bias+var} can be decomposed as follows. 
\begin{prop}\label{pro.bias_decomposition}
Let $\cU$ be some distribution over the set of finite partitions of $\X$, 
$\narb \geq 1$ some integer and $x\in \X$. 
Then,
\begin{equation} \label{eq.pro.bias_decomposition}
\E \croch{ \paren{s(x) - \tilde{s}_{\bV_{\narb}}(x)}^2 } 
= \paren{ s(x)-\E_{\bU \sim \cU} \croch{ \tilde{s}_{\bU}(x) } }^2 +
    \frac{\var_{\bU \sim \cU} \paren{ \tilde{s}_{\bU}(x) }}{\narb}
  \end{equation}
\end{prop}
\begin{proof}[Proof of Proposition~\ref{pro.bias_decomposition}]
  Remark that
  \[ 
  s(x) - \tilde{s}_{\bV_{\narb}}(x) 
  = \frac{1}{\narb}
  \sum_{j=1}^{\narb} \paren{ s(x) - \tilde{s}_{\bU^j}(x)} 
  \]
  is the average of $\narb$ independent random variables, with the
  same mean
  \[ s(x)-\E_{\bU \sim \cU} \croch{ \tilde{s}_{\bU}(x)} \]
  and variance
  \[ \var_{\bU \sim \cU} \paren{ \tilde{s}_{\bU}(x) } \enspace , \]
  which directly leads to Eq.~\eqref{eq.pro.bias_decomposition}.
\end{proof}
Hence, for every $x\in \X$, we get a decomposition of the (point-wise) quadratic risk of $\ERM (\cdot ; \bV_{\narb})$ into three terms: for every $x\in \X$, if \eqref{hyp.purely-random} holds true, 
\begin{align} 
\notag
& \qquad   \E\croch{ \paren{ s(x) - \ERM (x ; \bV_{\narb}; D_{\nobs})}^2 } 
\\
\label{eq.risk-cond.three.terms} 
& = \underbrace{\paren{ s(x)-\E_{\bU \sim \cU} \croch{ \tilde{s}_{\bU}(x)} }^2 +
  \frac{\var_{\bU \sim \cU} \paren{ \tilde{s}_{\bU}(x) }}{\narb}}_{\mbox{approximation error or bias}}
 + \underbrace{\E\croch{ \paren{ \tilde{s}_{\bV_{\narb}}(x) - \ERM (x ; \bV_{\narb}; D_{\nobs})}^2 }}_{\mbox{estimation error or variance}}
\end{align}
which can be rewritten as 
\begin{gather} \label{eq.risk-cond.three.terms.avec-notation} 
\E\croch{ \paren{ s(x) - \ERM (x ; \bV_{\narb}; D_{\nobs})}^2 } 
= \underbrace{
\paren{\Biasinfty{\cU}(x)}^2 +
  \frac{\Vararbre{\cU}(x)}{\narb}}_{\mbox{approximation error}}
 + \underbrace{
 \E\croch{ \paren{ \tilde{s}_{\bV_{\narb}}(x) - \ERM (x ; \bV_{\narb}; D_{\nobs})}^2 }}_{\mbox{estimation error}}
\\
\mbox{where} \quad 
\notag 
\Biasinfty{\cU}(x) \egaldef \paren{ s(x)-\E_{\bU \sim \cU} \croch{ \tilde{s}_{\bU}(x)} }^2 
\quad \mbox{and} \quad 
\Vararbre{\cU}(x) \egaldef \var_{\bU \sim \cU} \paren{ \tilde{s}_{\bU}(x) } \enspace .
\end{gather}
We choose to name (point-wise) approximation error (or bias) 
\begin{equation} 
\label{eq.bias-cond.avec-notation}
\Bias_{\cU,\narb}(x) \egaldef
\E_{\bV_{\narb} \sim \cU^{\otimes \narb}} \croch{ \paren{s(x) - \tilde{s}_{\bV_{\narb}}(x)}^2 }
= \Biasinfty{\cU}(x) +
  \frac{\Vararbre{\cU}(x)}{\narb}
\end{equation}
for consistency with the case of a single tree, where
$\ERM(\cdot ;\bU;D_{\nobs})$ is a regressogram (conditionally to $\bU$) and the (point-wise)
approximation error is
\[ 
\E\croch{ \paren{s(x) - \tilde{s}_{\bU}(x)}^2 } = \Bias_{\cU,1}(x)
\, . 
\]
Note that in all examples we consider in the following,
$\Biasinfty{\cU}(x)$ and $\Vararbre{\cU}(x)$ are asymptotically
decreasing functions of the number of leaves in the tree, as expected
for an approximation error.
Remark also that by Eq.~\eqref{eq.bias-cond.avec-notation},
$\lim_{\narb \to +\infty} \Bias_{\cU,\narb}(x) = \Biasinfty{\cU}(x)$,
which justifies the notation $\Biasinfty{\cU}(x)$.
Let us emphasize other authors such as \cite{Geu_Ern_Weh:2006} call
bias (of any tree or forest) the quantity
\[  
\Biasinfty{\cU} \egaldef \E\croch{ \Biasinfty{\cU}(X) } =
\E\croch{ \paren{s(X) - \E_{\bU \sim \cU} \croch{ \tilde{s}_{\bU}(X)} }^2 } \, ,
\]
that we call (integrated) bias of the infinite forest, 
so their simulation results must be compared with our theoretical statements with caution.

\medskip

%
The main goal of the paper is to study the (integrated) bias  
\[  
\Bias_{\cU,\narb} \egaldef \E\croch{ \Bias_{\cU,\narb}(X) } = \E\croch{ \paren{s(X) - \tilde{s}_{\bV_{\narb}}(X)}^2 } 
\]
of a forest of $\narb$ trees, in particular how it depends on $q$. 
By Eq.~\eqref{eq.bias-cond.avec-notation}, $\Bias_{\cU,\narb}$ is a non-increasing function of the number $\narb$ of trees in the forest. 
Furthermore, we can write the ratio between the (integrated)
approximation errors of a single tree and of a forest as
%
\begin{equation}
  \label{eq.comp-biais-arbre-foret.gal}
  \frac{\Bias_{\cU,1}}{\Biasinfty{\cU}}
  = 1 + \frac{\Vararbre{\cU}}{\Biasinfty{\cU}} 
\qquad \mbox{where} \quad 
\Vararbre{\cU} \egaldef \E\croch{ \Vararbre{\cU}(X)} 
\enspace .
\end{equation}
So, taking an infinite forest instead of a single tree decreases the
bias by the factor given by
Eq.~\eqref{eq.comp-biais-arbre-foret.gal}, which is larger or equal to one. 

The following sections compute in several cases the two key quantities
$\Biasinfty{\cU}$ and $\Vararbre{\cU}$, showing the ratio
$\frac{\Bias_{\cU,1}}{\Biasinfty{\cU}}$ can be much larger than one.

\subsection{General bounds on the variance term}
%
%
The variance term in Eq.~\eqref{eq.risk-cond.three.terms}, also called estimation error, is not the primary focus of the paper, but we need some bounds on its integrated version for comparing the risks of tree and forest estimators. 
The following proposition provides the bounds we use throughout the paper. 

\begin{proposition} \label{pro.maj-variance}
  Let $\narb \in [ 1 , +\infty] $ and $\bV_{\narb} = (\bU^1, \ldots, \bU^{\narb})$ be a sequence of independent random partitions of $\X$ with common distribution $\cU$, such that $\card(\bU^1) = \nfeu \in [1, + \infty)$ almost surely. 
  If assumption \eqref{hyp.purely-random} holds true, then, 
  \begin{align} 
    \label{eq.pro.maj-variance.foret-vs-arbre}
    \E\croch{ \paren{ \tilde{s}_{\bV_{\narb}}(X) - \ERM (X ; \bV_{\narb})}^2 } &\leq \E\croch{ \paren{ \tilde{s}_{\bU^1}(X) - \ERM (X ; \bU^1)}^2 }
    \\
    \label{eq.pro.maj-variance.arbre-gal}
    \E\croch{ \paren{ \tilde{s}_{\bU^1}(X) - \ERM (X ; \bU^1)}^2 } &\leq 
    \frac{\nfeu}{\nobs} \paren{ 2 \sigma^2  + 9 \norm{s}_{\infty}^2 }
    \\
    \mbox{and} \qquad 
    \E\croch{ \paren{ \tilde{s}_{\bU^1}(X) - \ERM (X ; \bU^1)}^2 } 
    &\geq 
    \frac{\sigma^2}{\nobs} \paren{ \nfeu - 2 \E_{\bU^1 \sim \cU} 
    \croch{ \sum_{\lambda \in \bU^1} \exp\paren{ - \nobs \P\paren{ X \in \lambda } } } }
    \label{eq.pro.maj-variance.arbre-gal-min}
    \enspace .
  \end{align}
If we assume in addition that $s$ is $K$-Lipschitz with respect to
some distance $\delta$ on $\X$, then
\begin{align}
    \E\croch{ \paren{ \tilde{s}_{\bU^1}(X) - \ERM (X ; \bU^1)}^2 } 
    &\leq \frac{2}{\nobs} 
    \paren{\sigma^2 \nfeu 
    + K^2 \E_{\bU^1 \sim \cU} \croch{ \sum_{\lambda \in \bU^1} \paren{ \diam_{\delta} (\lambda) }^2 } }
\notag 
\\
    & \qquad 
    + \frac{\norm{s}_{\infty}^2}{\nobs} \E_{\bU^1 \sim \cU} 
    	\croch{ \sum_{\lambda \in \bU^1} \exp\paren{ - \nobs \P\paren{ X \in \lambda } } }
    \label{eq.pro.maj-variance.arbre-Lipschitz}
    \enspace . 
\end{align}
\end{proposition}
Proposition~\ref{pro.maj-variance} is proved in
Section~\ref{sec.pr.gen_res.variance}. 
It shows the variance of a forest is upper bounded by the variance of
a single tree---with
Eq.~\eqref{eq.pro.maj-variance.foret-vs-arbre}---, and that the
variance of a single tree is roughly proportional to $\nfeu / \nobs$
where $\nfeu$ is the (deterministic) number of leaves of the
tree---with
Eq.~\eqref{eq.pro.maj-variance.arbre-gal}--\eqref{eq.pro.maj-variance.arbre-Lipschitz}.

\section{Approximation of the bias under smoothness conditions} \label{sec.multidim}
We now focus on the multidimensional case, say $\X=[0,1)^{\dimX}$ for some integer $\dimX \geq 1$, and we only consider partitions $\bU$ of $\X$ into hyper-rectangles, such that each $\lambda \in \bU$ has the form
\begin{equation} \label{eq.partition.multidim} \lambda = \prod_{i=1}^{\dimX} [A_i, B_i) \enspace , \end{equation} 
with $0 \leq A_i < B_i \leq 1$, for all $i=1, \ldots, \dimX$.
All RF models lead to such partitions. 
For the sake of simplicity, we assume from now on that 
\begin{equation} 
  \label{hyp.unif} \tag{{\bf Unif}} 
  X \mbox{ has a uniform distribution over } [0,1)^{\dimX}
  \enspace , 
\end{equation}
so that for each $\lambda = \prod_{i=1}^{\dimX} [A_i, B_i) \in \bU$,
\begin{equation} \label{eq.betl.multidim} \betl = \E\croch{ s(X) \sachant X \in \lambda } = \frac{1}{\absj{\lambda}} \int_{\lambda} s(t) dt \end{equation}
where $\absj{\lambda} = \prod_{i=1}^{\dimX} (B_i - A_i)$ denotes the volume of $\lambda$.

Let us now fix some $x \in \X$. 
For every partition $\bU$ of $\X$, $I_{\bU}(x) \egaldef \prod_{i=1}^{\dimX} [ A_{i,\bU}(x) \, , \, B_{i,\bU}(x) )$ denotes the unique element of $\bU$ to which $x$ belongs. Then, Eq.~\eqref{eq.betl.multidim} implies 
\begin{equation} \label{eq.valeur-foret-x.multidim}
  \tilde{s}_{\bU}(x) = \frac{1}{\absj{I_{\bU}(x)}} \int_{I_{\bU}(x)}
  s(t)  dt \enspace .
\end{equation}

\medskip

In order to compute $\Biasinfty{\cU}(x)$ and $\Vararbre{\cU}(x)$ we need some smoothness assumption about $s$ among the following: 
\begin{gather}
  \tag{{\bf H2a}} \label{hyp.s-2-fois-derivable.alt}
  \left.
    \begin{split}
      s \mbox{ is differentiable  on } \X \mbox{ and } 
      \exists \CHdeuxa > 0 \, , \quad \forall t,x \in \X \, , \qquad  \\ 
      \absj{ s(t) - s(x) - \nabla s(x) \cdot (t-x)} \leq \CHdeuxa \norm{t-x}_2^2 \quad 
    \end{split}
  \right\}
  \\
  \tag{{\bf H2}} \label{hyp.s-2-fois-derivable}
  \hspace{-0.95cm} s \mbox{ is twice differentiable  on } \X \mbox{ and } \nabla^{(2)} s \mbox{ is bounded}
  \\
  \tag{{\bf H3a}} \label{hyp.s-3-fois-derivable.alt}
  \left.
    \begin{split}
      s \mbox{ is twice differentiable  on } \X \mbox{ and } 
      \exists \CHtroisa > 0 \, , \quad \forall t,x \in \X \, ,  \qquad  \\
      \absj{ s(t) - s(x) - \nabla s(x) \cdot (t-x) - \frac{1}{2} (t-x)^{\top} \nabla^{(2)} s(x) (t-x) } \leq \CHtroisa \norm{t-x}_3^3 \enspace 
    \end{split}
  \right\}
  \\
  \tag{{\bf H3}} \label{hyp.s-3-fois-derivable}
  s \mbox{ is three times differentiable  on } \X \mbox{ and } \nabla^{(3)} s \mbox{ is bounded}
\end{gather}
where for any $v \in \R^{\dimX}$, $\snorm{v}_2^2 \egaldef \sum_{i=1}^{\dimX} v_i^2$ and $\snorm{v}_3^3 \egaldef \sum_{i=1}^{\dimX} \absj{v_i}^3$.
We denote by $\snorm{\nabla^{(2)} s}_{\infty,2}$ (resp. $\snorm{\nabla^{(3)} s}_{\infty,3}$) the sup-norm of $\nabla^{(2)} s$ (resp. $\nabla^{(3)} s$):
\begin{align*}
  & \norm{\nabla^{(2)} s}_{\infty,2} \egaldef \sup_{\norm{v}_2 \leq 1, \, y \in [0,1)^{\dimX} } \absj{ \sum_{1\leq i,j \leq \dimX} \frac{\partial^2 s(y)}{\partial x_i \partial x_j} v_i v_j }  \\
  & \norm{\nabla^{(3)} s}_{\infty,3} \egaldef \sup_{\norm{v}_3 \leq 1, \, y \in [0,1)^{\dimX} } \absj{ \sum_{1\leq i,j,k \leq \dimX} \frac{\partial^3 s(y)}{\partial x_i \partial x_j \partial x_k} v_i v_j v_k }  \, .
\end{align*}
By Taylor-Lagrange inequality, \eqref{hyp.s-2-fois-derivable} implies \eqref{hyp.s-2-fois-derivable.alt} with $ \CHdeuxa = \snorm{\nabla^{(2)} s}_{\infty,2} / 2$. 
Similarly, \eqref{hyp.s-3-fois-derivable} implies \eqref{hyp.s-3-fois-derivable.alt} with $ \CHtroisa = \snorm{\nabla^{(3)} s}_{\infty,3} / 6$. 

\medskip

Let us define, for every $i,j \in \set{1, \ldots, \dimX}$ and $x \in \X$, 
\begin{gather*}
  m_{A,i,\cU,x} \egaldef \E\croch{ x_i -A_{i,\bU}(x)} \qquad m_{B,i,x,\cU} \egaldef \E\croch{ B_{i,\bU}(x) -x_i} \\
  m_{AA,i,x,\cU} \egaldef \E\croch{ \paren{x_i -A_{i,\bU}(x)}^2 } \qquad m_{BB,i,x,\cU} \egaldef \E\croch{ \paren{B_{i,\bU}(x) -x_i}^2} \\
  m_{AAA,i,x,\cU} \egaldef \E\croch{ \paren{x_i -A_{i,\bU}(x)}^3 } \qquad m_{BBB,i,x,\cU} \egaldef \E\croch{ \paren{B_{i,\bU}(x) -x_i}^3} \\
  m_{AAAA,i,x,\cU} \egaldef \E\croch{ \paren{x_i -A_{i,\bU}(x)}^4 } \qquad m_{BBBB,i,x,\cU} \egaldef \E\croch{ \paren{B_{i,\bU}(x) -x_i}^4} \\
  m_{AB,i,x,\cU} \egaldef \E\croch{ \paren{x_i -A_{i,\bU}(x)}\paren{B_{i,\bU}(x) -x_i}} \\
  m_{B-A,i,j,x,\cU} \egaldef \E\croch{ \paren{B_{i,\bU}(x) -x_i - \paren{x_i -A_{i,\bU}(x)} } \paren{B_{j,\bU}(x) -x_j - \paren{x_j -A_{j,\bU}(x)} } } 
\end{gather*}
and for any $x \in \X \,$, assuming either \eqref{hyp.s-2-fois-derivable.alt} or \eqref{hyp.s-3-fois-derivable.alt}, 
\begin{align*}
  \mathcal{M}_{1,\cU,x} &\egaldef \frac{1}{2} \sum_{i=1}^{\dimX} \croch{ \frac{\partial s}{\partial x_i} (x) \paren{ m_{B,i,x,\cU} - m_{A,i,x,\cU} } } \\
  \mathcal{M}_{2,\cU,x} &\egaldef \frac{1}{6} \sum_{i=1}^{\dimX}
  \croch{ \frac{\partial^2 s}{\partial x_i^2}(x) \paren{
      m_{AA,i,x,\cU} + m_{BB,i,x,\cU} - m_{AB,i,x,\cU} } } \\ & \qquad
  + \frac{1}{8} \sum_{1\leq i\neq j \leq \dimX} \croch{ \frac{\partial^2 s}{\partial x_i \partial x_j}(x) m_{B-A,i,j,x,\cU} } \\
  \mathcal{N}_{2,\cU,x} &\egaldef \frac{1}{4} \sum_{i=1}^{\dimX}
  \croch{ \paren{ \frac{\partial s}{\partial x_i} (x) }^2 \paren{
      m_{AA,i,x,\cU} + m_{BB,i,x,\cU} - 2 \, m_{AB,i,x,\cU} } } \\ &
  \qquad
  + \frac{1}{4} \sum_{1\leq i\neq j \leq \dimX} \croch{ \frac{\partial s}{\partial x_i} (x) \frac{\partial s}{\partial x_j} (x) m_{B-A,i,j,x,\cU} } \\
  \mathcal{R}_{2,\cU,x} &\egaldef \frac{ \CHdeuxa }{3} \sum_{i=1}^{\dimX} \paren{ m_{AA,i,x,\cU} + m_{BB,i,x,\cU} - m_{AB,i,x,\cU} } \\
  \mathcal{R}_{3,\cU,x} &\egaldef \frac{ \CHtroisa }{4} \sum_{i=1}^{\dimX} \paren{ m_{AAA,i,x,\cU} + m_{BBB,i,x,\cU} } \\
  \mathcal{R}_{4,\cU,x} &\egaldef \frac{ 2 \dimX \CHdeuxa^2 }{9}
  \sum_{i=1}^{\dimX} \paren{ m_{AAAA,i,x,\cU} + m_{BBBB,i,x,\cU} }
  \enspace .
\end{align*}
We can now state a general result on the two terms appearing in decomposition~\eqref{eq.bias-cond.avec-notation} of the approximation error of a forest of size $\narb$ when assumption~\eqref{hyp.purely-random} holds true. 
\begin{prop} \label{pro.bias.multidim.H3} Let $\X=[0,1)^{\dimX}$ and
  assume \eqref{hyp.s-2-fois-derivable.alt} and~\eqref{hyp.unif} hold
  true.  Then, for every $x \in \X \,$,
  \begin{gather} 
    \paren{\mathcal{M}_{1,\cU,x} }^2 - 2 \mathcal{M}_{1,\cU,x}
    \mathcal{R}_{2,\cU,x} \leq \Biasinfty{\cU}(x) \leq \paren{
      \mathcal{M}_{1,\cU,x} + \mathcal{R}_{2,\cU,x}}^2
    \label{eq.pro.bias.multidim.Bias-H2}
    \\
\mbox{and} \quad 
    \absj{ \Vararbre{\cU}(x) - \paren{ \mathcal{N}_{2,\cU,x} - \paren{
          \mathcal{M}_{1,\cU,x} }^2 } } \leq 2 \sqrt{
      \mathcal{R}_{4,\cU,x} \: \paren{ \mathcal{N}_{2,\cU,x} - \paren{
          \mathcal{M}_{1,\cU,x} }^2 } } + \mathcal{R}_{4,\cU,x}
\enspace . 
    \label{eq.pro.bias.multidim.Var-H2}
  \end{gather}
  Furthermore, if \eqref{hyp.s-3-fois-derivable.alt} also holds true
  \textup{(}which implies \eqref{hyp.s-2-fois-derivable.alt} holds with
  $\CHdeuxa = \snorm{\nabla^{(2)} s}_{\infty,2} / 2$\textup{)}, for
  every $x \in \X\,$,
  \begin{gather} 
    \absj{ \Biasinfty{\cU}(x) - \paren{ \mathcal{M}_{1,\cU,x} +
        \mathcal{M}_{2,\cU,x} }^2 } \leq 2 \absj{
      \mathcal{R}_{3,\cU,x} \paren{ \mathcal{M}_{1,\cU,x} +
        \mathcal{M}_{2,\cU,x} } } + \paren{ \mathcal{R}_{3,\cU,x} }^2
    \label{eq.pro.bias.multidim.Bias-H3}
  \end{gather}
\end{prop}
Proposition~\ref{pro.bias.multidim.H3} is proved in
Section~\ref{sec.pr.gen_res.bias}.
As we will see in the following, Proposition~\ref{pro.bias.multidim.H3} is precise. 
Indeed, under \eqref{hyp.s-2-fois-derivable.alt}, 
we get a gap of an order of magnitude between the upper bound on $\Biasinfty{\cU}(x)$ 
in Eq.~\eqref{eq.pro.bias.multidim.Bias-H2} and the lower bound on $\Vararbre{\cU}(x)$ 
in Eq.~\eqref{eq.pro.bias.multidim.Var-H2}. 
Thus from Eq.~\eqref{eq.comp-biais-arbre-foret.gal}, it comes that the
bias of infinite forests 
is much smaller than the bias of single trees.
Furthermore, under \eqref{hyp.s-3-fois-derivable.alt}, we get a tight
lower bound for $\Biasinfty{\cU}(x)$, which shows the upper bound
in Eq.~\eqref{eq.pro.bias.multidim.Bias-H2} gives the actual rate of
convergence, at least when $s$ is smooth enough.
 
\section{Toy model}\label{sec_toy_model}
A toy model of PRF is when the random partition is obtained by
translation of a regular partition of $\X = [0,1)$ into $\nfeu +1 \geq 2$
pieces.  
Formally, $\bU \sim \cUtoy{\nfeu}$ is defined by
\[ 
\bU = \set{ \left[ 0 , \frac{1-T}{\nfeu} \right) \, , \, \left[ \frac{1-T}{\nfeu} , \frac{2-T}{\nfeu}\right) \, , \ldots , \left[ \frac{\nfeu-T}{\nfeu} , 1 \right) } 
\]
where $T$ is a random variable with uniform distribution over $[0,1)$.

This random partition scheme is very close to the example of random
binning features in Section~4 of \cite{Rah_Rec:2007}, the main
difference being that here $\X=[0,1)$ instead of $\R$. 

\subsection{Link between the bias of the infinite forest and the bias of some kernel estimator} \label{sec_toy_model.link-kernel}
First, we show that the expected infinite forest, defined by
$\tilde{s}_{\infty}(x) \egaldef \E_{\bU \sim \cUtoy{\nfeu}} \croch{
  \tilde{s}_{\bU}(x) } $, can be expressed as a convolution between
$s$ and some kernel function.
\begin{prop}\label{toy_model_kernel}
  Assume that $\nfeu \geq 2$, \eqref{hyp.unif} holds true and consider
  the purely random forest model $\cUtoy{\nfeu}$. For any $ x\in
  \croch{ \frac{1}{\nfeu}, 1-\frac{1}{\nfeu} } \, $, the expected
  infinite forest at point $x$ satisfies:
  \begin{align}
    \tilde{s}_{\infty}(x) &= \E_{\bU \sim \cUtoy{\nfeu}} \croch{
      \tilde{s}_{\bU}(x) } = \int_0^1 s(t) h^{\cUtoy{\nfeu}} (t-x) \,
    dt
  \\
   \mbox{where} \quad h^{\cUtoy{\nfeu}}(u) &\egaldef
  \begin{cases}
    \nfeu(1 - \nfeu u)  \mbox{ if } 0 \leq u \leq \frac{1}{\nfeu} \\
    \nfeu(1 + \nfeu u)  \mbox{ if } -\frac{1}{\nfeu} \leq u \leq 0 \\
    0  \mbox{ if } |u| \geq \frac{1}{\nfeu} 
    \enspace . 
  \end{cases}
  \end{align}
\end{prop}
Proposition~\ref{toy_model_kernel} is proved in
Section~\ref{sec.pr.toy_model.link-kernel}.
One key quantity for our bias analysis is
$\Biasinfty{\cUtoy{\nfeu}}(x)=(\tilde{s}_{\infty}(x) - s(x))^2$, which
is, according to Proposition~\ref{toy_model_kernel}, close to the bias
of the kernel estimator associated to $h^{\cUtoy{\nfeu}}$ \citep[see
e.g. Chapter~4 of][]{Gyo_Las_Krz_Wal:2002}.
This point will enlighten the bias decreasing rates of the next
section.
See also Figure~\ref{fig.urt+all.inf-forest} for a plot of
$h^{\cUtoy{\nfeu}}$, and a comparison with other PRF models.

Finally, we point out that a similar remark has been made by
\cite{Rah_Rec:2007} with a different goal, where $h^{\cUtoy{\nfeu}}$ is called the ``hat
kernel''.

\subsection{Bias for twice differentiable functions}
As a corollary of Proposition~\ref{pro.bias.multidim.H3}, we get the
following estimates of the terms appearing in
decomposition~\eqref{eq.bias-cond.avec-notation} of the bias for the
toy model.
\begin{corollary}\label{cor.bias.toy.H2}
  Let $\nfeu\geq 2$, $ \epstoy_{\nfeu}=1/\nfeu $ and assume \eqref{hyp.s-2-fois-derivable.alt} and~\eqref{hyp.unif} hold true. 
  Then, for every $ x \in \croch{ \epstoy_{\nfeu} , 1 - \epstoy_{\nfeu} }$, 
  \begin{align} \label{eq.bias.toy.H2.Biasinfty}
    \Biasinfty{\cUtoy{\nfeu}}(x) &\leq \frac{\CHdeuxa^2}{36 \nfeu^4} \\ 
    \label{eq.bias.toy.H2.Vararbre}
    \absj{\Vararbre{\cUtoy{\nfeu}}(x) - \frac{\paren{s^{\prime}(x)}^2}{12 \nfeu^2}} &\leq \frac{ 2 \norm{s^{\prime}}_{\infty} \CHdeuxa + \CHdeuxa^2}{\nfeu^3} 
  \end{align}
  and for every $ x \in(0,1) \backslash \croch{ \epstoy_{\nfeu} , 1 - \epstoy_{\nfeu} }$, 
  \begin{align} \label{eq.bias.toy.H2.Biasinfty.border}
    \Biasinfty{\cUtoy{\nfeu}}(x) &
    \leq  \frac{ \paren{s^{\prime}(x)}^2 }{16 \nfeu^2} + \frac{\CHdeuxa \norm{s^{\prime}}_{\infty}  + \CHdeuxa^2}{2 \nfeu^3} 
    \\ 
    \label{eq.bias.toy.H2.Vararbre.border}
    \absj{\Vararbre{\cUtoy{\nfeu}}(x) - \frac{\paren{s^{\prime}(x)}^2 \PolVarU( \nfeu \min\set{x , 1-x })}{\nfeu^2}} &\leq \frac{ 2 \norm{s^{\prime}}_{\infty} \CHdeuxa + \CHdeuxa^2}{\nfeu^3} 
  \end{align}
  for some polynomial $\PolVarU$ such that $\sup_{t \in [0,1]} \absj{ \PolVarU(t) } \leq 1$.
  As a consequence,
  \begin{align} \label{eq.bias.toy.H2.Biasinfty.integrated}
    \int_0^1 \Biasinfty{\cUtoy{\nfeu}}(x) dx &\leq 
    \frac{\norm{s^{\prime}}_{\infty}^2}{8 \nfeu^3 } + \frac{ \CHdeuxa \norm{s^{\prime}}_{\infty} + 2 \CHdeuxa^2}{\nfeu^4} 
    \\ 
    \label{eq.bias.toy.H2.Vararbre.integrated}
    \absj{ \int_0^1 \Vararbre{\cUtoy{\nfeu}}(x) dx - \frac{1}{12  \nfeu^2} \int_0^1 \paren{s^{\prime}(x)}^2 \, dx } &\leq \frac{2 \norm{s^{\prime}}_{\infty} \CHdeuxa + \CHdeuxa^2 + 3 \norm{s^{\prime}}_{\infty}^2 }{ \nfeu^3 }  
    \\ 
\mbox{and} \quad 
  \label{eq.bias.toy.H2.Biasinfty.integrated-noborder}
    \int_{\epstoy_{\nfeu}}^{1-\epstoy_{\nfeu}} \Biasinfty{\cUtoy{\nfeu}}(x) dx &\leq \frac{\CHdeuxa^2}{36 \nfeu^4} 
    \enspace .
  \end{align}
\end{corollary}
Corollary~\ref{cor.bias.toy.H2} is proved in
Section~\ref{sec.pr.toy_model.cor-H2}.
The order of magnitude of the bounds on $\Biasinfty{\cUtoy{\nfeu}}$
are the correct ones (up to constants) when $s$ is smooth enough, as
shown by Corollary~\ref{cor.bias.toy.H3} in Section~\ref{sec.toy.H3}.

Inequalities \eqref{eq.bias.toy.H2.Biasinfty.integrated} and
\eqref{eq.bias.toy.H2.Vararbre.integrated} give the first order of the
bias of a tree: $\frac{1}{12 \nfeu^2} \int_0^1 \paren{s^{\prime}(x)}^2
\, dx $, which is the classical bias term of a regular regressogram
\citep[see e.g. Chapter~6.2 of][for a regular histogram, in a density
estimation framework]{Was:2006}.
This is not surprising because a random tree in the toy model is very
close to a regular regressogram, the only difference being that the
regular partition of $[0,1)$ is randomly translated.  So at first
order, the bias of the histogram built on the regular subdivision of
$[0,1)$ and of the one built on the randomly translated one are equal.

\begin{remark}[Border effects]
  The border effects, also known as boundary bias, highlighted by
  Corollary~\ref{cor.bias.toy.H2} is a well-known phenomenon for
  kernel estimators \citep[see e.g. Chapter~5.4 in][]{Was:2006}.
  Since the infinite forest is equivalent to a kernel estimator in
  terms of bias, it suffers from the same phenomenon. We could use
  standard techniques to suppress these border effects (e.g. by
  working on the torus instead of interval $[0,1)$), but this is out
  of the scope of this paper.
\end{remark}

\subsection{Discussion: single tree vs. infinite forest}\label{subsec_discuss_toy_model}
We can now compare a single tree and an infinite forest for the toy
model $\cUtoy{\nfeu}$, first in terms of approximation error for a
given $\nfeu$, then in terms of risk for a well-chosen $\nfeu$.  In
this section, we assume \eqref{hyp.s-2-fois-derivable.alt}
and~\eqref{hyp.unif} hold true.

\subsubsection*{Approximation error}
Corollary~\ref{cor.bias.toy.H2} and
Eq.~\eqref{eq.comp-biais-arbre-foret.gal} allow to compare the
approximation errors of a single tree and of an infinite forest: for
all $x \in [\nfeu^{-1} , 1 - \nfeu^{-1} ]$,
\[ \Biasinfty{\cUtoy{\nfeu}}(x) \leq \frac{L(\CHdeuxa)}{\nfeu^4} \quad
\mbox{whereas} \quad \Bias_{\cUtoy{\nfeu},1}(x) \geq \frac{1}{12 \nfeu^2}
\paren{s^{\prime}(x)}^2  -
\frac{L(\norm{s^{\prime}}_{\infty} , \CHdeuxa)}{\nfeu^3}  \]
is much larger, 
where we recall that notation $L(\cdot)$ is defined at the end of
Section~\ref{sec.intro}. 
The same comparison occurs when
integrating over $x \in [\nfeu^{-1} , 1 - \nfeu^{-1} ]$.
Therefore, considering an infinite forest instead of a single tree
decreases the approximation error from an order of magnitude, and not
only from a constant factor, when the number $\nfeu + 1$ of leaves of
each tree tends to infinity. More precisely, the bias decreasing rate
of an infinite forest is smaller or equal to the square of the bias
rate of a single tree.

\subsubsection*{Risk bounds for a well-chosen $\nfeu$}
Combining Eq.~\eqref{eq.risk-cond.three.terms.avec-notation}  with approximation
error controls (Corollary~\ref{cor.bias.toy.H2}) and the
general bounds on the estimation error
(Proposition~\ref{pro.maj-variance}), we can compare the statistical
risks of estimators built on a single tree and on an infinite forest,
respectively.
For all $\narb \in [1,+\infty]$ and $\nfeu \geq 1$, suppose $\bV_{\narb}
\sim (\cUtoy{\nfeu})^{\otimes \narb}$ and $\nobs \geq 1$ data points
are available. 
Let $\varepsilon \in ]0, 1/2[$ and consider only trees with $\nfeu
\geq 1/\varepsilon$ leaves and points $x \in [\varepsilon,1-\varepsilon]$, in order to avoid border effects. 
Then, 
\begin{align*}
&\quad \int_{\varepsilon}^{1-\varepsilon} \E\croch{\paren{ \ERM\paren{x;
      \bV_{\narb}; D_{\nobs}} - s(x) }^2} dx 
\\
&= 
\int_{\varepsilon}^{1-\varepsilon} \Bias_{\cUtoy{\nfeu},\narb}(x) dx +
\int_{\varepsilon}^{1-\varepsilon} \E\croch{ \paren{\ERM\paren{x; \bV_{\narb};
  D_{\nobs}} - \tilde{s}_{\bV_{\narb}}(x)}^2 } dx
\\
&\leq \int_{\varepsilon}^{1-\varepsilon}
\Bias_{\cUtoy{\nfeu},\narb}(x) dx + \int_0^1 \E\croch{ \paren{\ERM\paren{x; \bV_{\narb};
  D_{\nobs}} - \tilde{s}_{\bV_{\narb}}(x)}^2 } dx
\\
&\leq \int_{\varepsilon}^{1-\varepsilon}
\Bias_{\cUtoy{\nfeu},\narb}(x) dx 
+ \frac{2 \sigma^2 (\nfeu + 1)}{\nobs} 
+ \frac{2 \norm{s^{\prime}}_{\infty}^2 (\nfeu+1)}{\nobs \nfeu^2}
+ \frac{\norm{s}_{\infty}^2 \croch{ (\nfeu-1) e^{-\nobs / \nfeu} + 2}}{\nobs} 
\end{align*}
using Eq.~\eqref{eq.pro.maj-variance.arbre-Lipschitz} and that if $\bU \sim
\cUtoy{\nfeu}$, $\sup_{\lambda \in \bU}
\diam_{L^2}(\lambda)\leq \nfeu^{-1}$ and 
$\P(X \in \lambda) = \nfeu^{-1}$ for every $\lambda \in \bU \setminus
\set{ \left[ 0 , \frac{1-T}{\nfeu} \right) \, , \, \left[
    \frac{\nfeu-T}{\nfeu} , 1 \right) } $ a.s.

So, if we are able to choose the number of leaves $\nfeu+1$
optimally---for instance by cross-validation
\citep[see e.g.][]{Arl_Cel:2010:surveyCV}---, the risk of an infinite forest
estimator, defined by:
\[\forall x\in\X \quad \ERM_{\infty}\paren{x, D_{\nobs}} \egaldef
\E_{\bU \sim \cU} \croch{ \ERM\paren{x, \bU, D_{\nobs}} \sachant
  D_{\nobs} } \enspace ,\]
is upper bounded as follows: if $\nobs \geq 1/\varepsilon$ and if \eqref{hyp.s-2-fois-derivable.alt} holds true, 
\begin{align*}
  &\qquad \inf_{1/\varepsilon \leq \nfeu \leq \nobs}
  \int_{\varepsilon}^{1-\varepsilon} \E\croch{\paren{
      \ERM_{\infty}\paren{x, D_{\nobs}} - s(x) }^2} dx
  \\
  &\leq \inf_{1/\varepsilon \leq \nfeu \leq \nobs} \set{
\frac{L(\CHdeuxa)}{\nfeu^4}
+ \frac{2 \sigma^2 (\nfeu + 1)}{\nobs} 
+ \frac{2 \norm{s^{\prime}}_{\infty}^2 (\nfeu+1)}{\nobs \nfeu^2}
+ \frac{\norm{s}_{\infty}^2 \croch{ (\nfeu-1) e^{-\nobs / \nfeu} + 2}}{\nobs} 
 }
  \\
  &\leq L(\CHdeuxa) \inf_{1/\varepsilon \leq \nfeu \leq \nobs} \set{
    \frac{1}{\nfeu^4} 
+ \frac{\sigma^2 \nfeu }{\nobs} 
+ \norm{s}_{\infty}^2 \frac{\nfeu}{\nobs} e^{-\nobs / \nfeu}
 }
+ \frac{\norm{s}_{\infty}^2 + \norm{s^{\prime}}_{\infty}^2 }{\nobs}
  \\
  &\leq L(\CHdeuxa) 
  \paren{ \frac{\sigma^2}{\nobs} }^{4/5} \paren{ 1 +
    \frac{ \norm{s}_{\infty}^2 }{\nobs^{4/5} \sigma^{12/5}} } 
    + \frac{ \norm{s}_{\infty}^2 + \norm{s^{\prime}}_{\infty}^2 }{\nobs}
\\
&\leq L(\CHdeuxa) 
  \paren{ \frac{\sigma^2}{\nobs} }^{4/5}
\end{align*}
by Lemma~\ref{le.opt-risk-bis} in Section~\ref{sec.pr.technical}, assuming in addition $\nobs \geq L(\sigma^2, \varepsilon, \norm{s}_{\infty},  \norm{s^{\prime}}_{\infty})$. 
Thus, we recover the classical convergence risk
rate of a kernel estimator when the regression function satisfies \eqref{hyp.s-2-fois-derivable.alt} \citep[see e.g. Chapter~5.4 in][]{Was:2006}.

\bigskip

For $\narb=1$, the risk of a tree estimator is lower bounded by the
following. We again suppose that $\varepsilon \in ]0, 1/2[$, and in
addition, we assume that $\int_{\varepsilon}^{1-\varepsilon} \paren{
  s^{\prime}(x) }^2 dx >0$ and fix $\nfeu_0 = \left\lfloor \max \set{
    \frac{24 L(\norm{s^{\prime}}_{\infty} ,
      \CHdeuxa)}{\int_{\varepsilon}^{1-\varepsilon} \paren{
        s^{\prime}(x) }^2 dx} , 1/\varepsilon} \right\rfloor + 1$.
From a slight adaptation of Eq.~\eqref{eq.pro.maj-variance.arbre-gal-min} in Proposition~\ref{pro.maj-variance} (by
integrating only over leaves $\lambda \in \bU$ such that $\lambda \subset [\varepsilon , 1-\varepsilon]$), if $\bU
\sim \cUtoy{\nfeu}$ we have:

\begin{align*}
  \int_{\varepsilon}^{1-\varepsilon} \E\croch{\paren{\ERM\paren{x; \bU;
      D_{\nobs}} - \tilde{s}_{\bU}(x)}^2} dx
  &\geq
  \frac{\sigma^2}{\nobs} \croch{1-2 \exp(-\nobs / \nfeu)} \times \card\set{ \lambda \in
    \bU \telque \lambda \subset [\varepsilon, 1-\varepsilon] }
  \\
  &\geq \frac{\sigma^2}{\nobs} \croch{1- 2 \exp(-\nobs / \nfeu)} \nfeu (1-2\varepsilon)
\end{align*}

so that if $\nobs \geq \nfeu_0$, 
\begin{align*}
  &\qquad \inf_{\nfeu_0 \leq \nfeu \leq n}
  \int_{\varepsilon}^{1-\varepsilon} \E\croch{\paren{ \ERM \paren{x,
        \bU, D_{\nobs}} - s(x) }^2} dx
  \\
  &= \inf_{\nfeu_0 \leq \nfeu \leq n}
  \int_{\varepsilon}^{1-\varepsilon} \paren{
    \E\croch{\paren{\tilde{s}_{\bU}(x) - s(x) }^2} +
    \E\croch{\paren{ \ERM \paren{x, \bU, D_{\nobs}} - \tilde{s}_{\bU}(x) }^2} } dx
  \\
  &\geq \inf_{\nfeu_0 \leq \nfeu \leq n} \set{
    \int_{\varepsilon}^{1-\varepsilon} \Bias_{\cUtoy{\nfeu},1}(x) dx +
    \frac{\sigma^2 \nfeu(1-2\varepsilon) \croch{ 1 - 2 \exp\paren{ - \nobs / \nfeu }
      }}{\nobs} }
  \\
  &\geq \inf_{\nfeu_0 \leq \nfeu \leq n} \set{ \nfeu^{-2} \frac{1}{12}
    \int_{\varepsilon}^{1-\varepsilon} \paren{s^{\prime}(x)}^2 \, dx -
    L(\norm{s^{\prime}}_{\infty} , \CHdeuxa) \nfeu^{-3} +
    \frac{\sigma^2 \nfeu (1-2\varepsilon) \croch{1 - 2\exp\paren{-1}} }{\nobs}
  }
  \\
  &\geq \inf_{\nfeu_0 \leq \nfeu \leq n} \set{ \nfeu^{-2} \frac{1}{24}
    \int_{\varepsilon}^{1-\varepsilon} \paren{s^{\prime}(x)}^2 \, dx +
    \frac{\sigma^2 \nfeu (1-2\varepsilon) \croch{ 1 - 2\exp\paren{-1} } }{\nobs} }
  \\
  &\geq L(s, \varepsilon) \paren{ \frac{\sigma^2}{\nobs} }^{2/3} 
\end{align*}
by Lemma~\ref{le.opt-risk}, assuming in addition $\nobs \geq L(\sigma^2,\varepsilon)$. 

Here, we recover the classical risk rate of a regular histogram
estimator \citep[see e.g. Chapter~4 in][]{Gyo_Las_Krz_Wal:2002}.
Therefore, an infinite forest estimator attains (up to some constant)
the minimax rate of convergence \citep[see e.g. Chapter~3
in][]{Gyo_Las_Krz_Wal:2002} over the set of $C^2$ functions---all $C^2$
functions satisfy \eqref{hyp.s-2-fois-derivable.alt}---, whereas a
single tree estimator does not (except maybe for constant functions
$s$).

Note finally that when taking care of the borders, even an infinite
forest estimator is not sufficient for attaining the minimax rate of
convergence (at least, with our upper bounds, but they are tight under
additional assumptions according to Corollary~\ref{cor.bias.toy.H3} in
the next section).
So, as for classical kernel regression estimators, taking into account
border effects can be crucial for some random forests estimators.

\subsection{Tighter bound for three times differentiable functions}
\label{sec.toy.H3}

\begin{corollary}\label{cor.bias.toy.H3}
  Let $\nfeu \geq 2$, $\epstoy_{\nfeu}=1/\nfeu$ and assume
  \eqref{hyp.s-3-fois-derivable.alt} and \eqref{hyp.unif} hold true.
  Then, for every $ x \in \croch{ \epstoy_{\nfeu} , 1 -
    \epstoy_{\nfeu} }$,
  \begin{align} \label{eq.bias.toy.H3.Biasinfty}
    \absj{ \Biasinfty{\cUtoy{\nfeu}}(x) - \frac{\paren{s^{\prime\prime}(x)}^2 }{144 \nfeu^4} }  &\leq \frac{ 3 \CHtroisa \paren{ \norm{s^{\prime\prime}}_{\infty} + \frac{3 \CHtroisa}{2}}}{4 \nfeu^5}
  \end{align}
  and for every $ x \in (0,1) \backslash \croch{ \epstoy_{\nfeu} , 1 - \epstoy_{\nfeu} }$, 
  \begin{equation} \label{eq.bias.toy.H3.Biasinfty.border}
    \begin{split}
      & \absj{ \Biasinfty{\cUtoy{\nfeu}}(x) - \frac{ \paren{s^{\prime}(x)}^2 (1 - \nfeu \min\set{ x , 1 - x})^4}{16 \nfeu^2} } \\
      & \qquad \leq \frac{\norm{s^{\prime}}_{\infty} \norm{s^{\prime\prime} }_{\infty} + 2  \CHtroisa \norm{s^{\prime}}_{\infty} + \norm{s^{\prime\prime}}_{\infty}^2 +  \CHtroisa \norm{s^{\prime\prime}}_{\infty} + 2 \CHtroisa^2 }{ 4 \nfeu^3 }   \enspace .
    \end{split}
  \end{equation}
  As a consequence, 
  \begin{equation} \label{eq.bias.toy.H3.Biasinfty.integrated}
    \begin{split}
      &\hspace{-1cm} \absj{ \int_0^1  \Biasinfty{\cUtoy{\nfeu}}(x) \, dx  - \frac{ 1 }{16 \nfeu^2} \int_0^{\nfeu^{-1}} \croch{ \paren{s^{\prime}(x)}^2 + \paren{s^{\prime}(1-x)}^2 } (1 - \nfeu x)^4 \, dx } \\
      &\leq \frac{\norm{s^{\prime}}_{\infty} \norm{s^{\prime\prime} }_{\infty} + 2  \CHtroisa \norm{s^{\prime}}_{\infty} + \norm{s^{\prime\prime}}_{\infty}^2 +  2 \CHtroisa \norm{s^{\prime\prime}}_{\infty} + 4 \CHtroisa^2 }{ 2 \nfeu^4 } 
    \end{split}
  \end{equation}
  and
  \begin{align} \label{eq.bias.toy.H3.Biasinfty.integrated-noborder}
    \absj{ \int_{\epstoy_{\nfeu}}^{1-\epstoy_{\nfeu}}  \Biasinfty{\cUtoy{\nfeu}}(x) \, dx  - \frac{1}{144 \nfeu^4} \int_{\epstoy_{\nfeu}}^{1-\epstoy_{\nfeu}} \paren{s^{\prime\prime}(x)}^2 \, dx } &\leq \frac{ 3 \CHtroisa \paren{ \norm{s^{\prime\prime}}_{\infty} + \frac{3 \CHtroisa}{2}}}{4 \nfeu^5} \enspace .
  \end{align}
\end{corollary}
Corollary~\ref{cor.bias.toy.H3} is proved in Section~\ref{sec.pr.toy_model.cor-H3}. 
Hence, if $s$ satisfies \eqref{hyp.s-3-fois-derivable.alt} the infinite
forest bias is of the order of $\nfeu^{-4}$ (without taking into
account borders).  This shows, at least for $s$ smooth enough, that
upper bounds of Corollary~\ref{cor.bias.toy.H2} involve the correct
rates.

\subsection{Size of the forest}

According to Eq.~\eqref{eq.pro.bias_decomposition}, taking
$\narb=\infty$ is not necessary for reducing the bias of a tree from
an order of magnitude. In particular, even without border effects,
$\Bias_{\cUtoy{\nfeu},\narb}$ is of the same order as $\Biasinfty{\cUtoy{\nfeu}}$ when
$\narb \geq \nfeu^2$ under assumption \eqref{hyp.s-3-fois-derivable.alt}.
So, we get a practical hint for choosing the size of the forest,
leading to an estimator that can be computed since it does not need
$\narb$ to be infinite.

\section{Purely uniformly random forests}\label{sec.purf}
We now consider a PRF model introduced by \cite{Gen:2012}, called Purely Uniformly Random Forests (PURF).

For every integer $\nfeu \geq 1$, the random partition $\bU \sim \cUpurf{\nfeu}$ is defined as follows.
Let $\xi_1, \ldots, \xi_{\nfeu}$ be independent random variables with uniform distribution over $\X=[0,1)$ 
and let $\xi_{(1)} < \dots < \xi_{(\nfeu)}$ the corresponding order statistics. 
Then, $\bU$ is defined by
\[ 
\bU = \set{ \left[ 0 ,\xi_{(1)}  \right) \, , \, \left[ \xi_{(1)} , \xi_{(2)}\right) \, , \ldots , \left[ \xi_{(\nfeu)} , 1 \right) }  \, . 
\]

\subsection{Interpretation of the bias of the infinite forest} \label{sec.purf.inf-forest}
Similarly to Proposition~\ref{toy_model_kernel}, we can try to
interpret the bias of the infinite forest for any purely random
forest.  
Indeed, as in the proof of
Proposition~\ref{toy_model_kernel}, for any $x \in [0,1)$, by Fubini's theorem,
\begin{gather} \label{eq.inf-forest.gal}
  \tilde{s}_{\infty}(x) =
  \E_{\bU} \croch{ \tilde{s}_{\bU}(x) } = \int_0^1 s(t) \E_{\bU}
  \croch{ \frac{\un_{t \in I_{\bU}(x)}}{\absj{I_{\bU}(x)}} } \, dt
  = \int_0^1 s(t) h^{\cU}(t,x) dt 
\end{gather}
where $I_{\bU}(x)$ denotes the unique interval of $\bU$
containing $x$ and 
\begin{equation}
\label{def.hcU.gal} 
h^{\cU}(t,x) \egaldef \E_{\bU \sim \cU}
\croch{ \frac{\un_{t \in I_{\bU}(x)}}{\absj{I_{\bU}(x)}} } \enspace. 
\end{equation}
In the toy model case, it turns out that $h^{\cUtoy{\nfeu}}(t,x)$ only
depends on $t-x$ (when $x$ is far enough from the boundary), so we
have an exact link with a kernel estimator. 
In the PURF model case,
$h^{\cUpurf{\nfeu}}(t,x)$ does not only depend on $t-x$, but only
mildly as shown by numerical computations
(Figure~\ref{fig.purf.inf-forest}). Hence, for the PURF model, the
bias of the infinite forest is equal to the bias of an estimator close
to a kernel estimator. Note that $h^{\cUpurf{\nfeu}}$ is compared to
$h^{\cU}$ for the other random forest models considered in this paper
on Figure~\ref{fig.urt+all.inf-forest}.

\begin{figure}
  \begin{center}
    \includegraphics[width=0.45\textwidth]{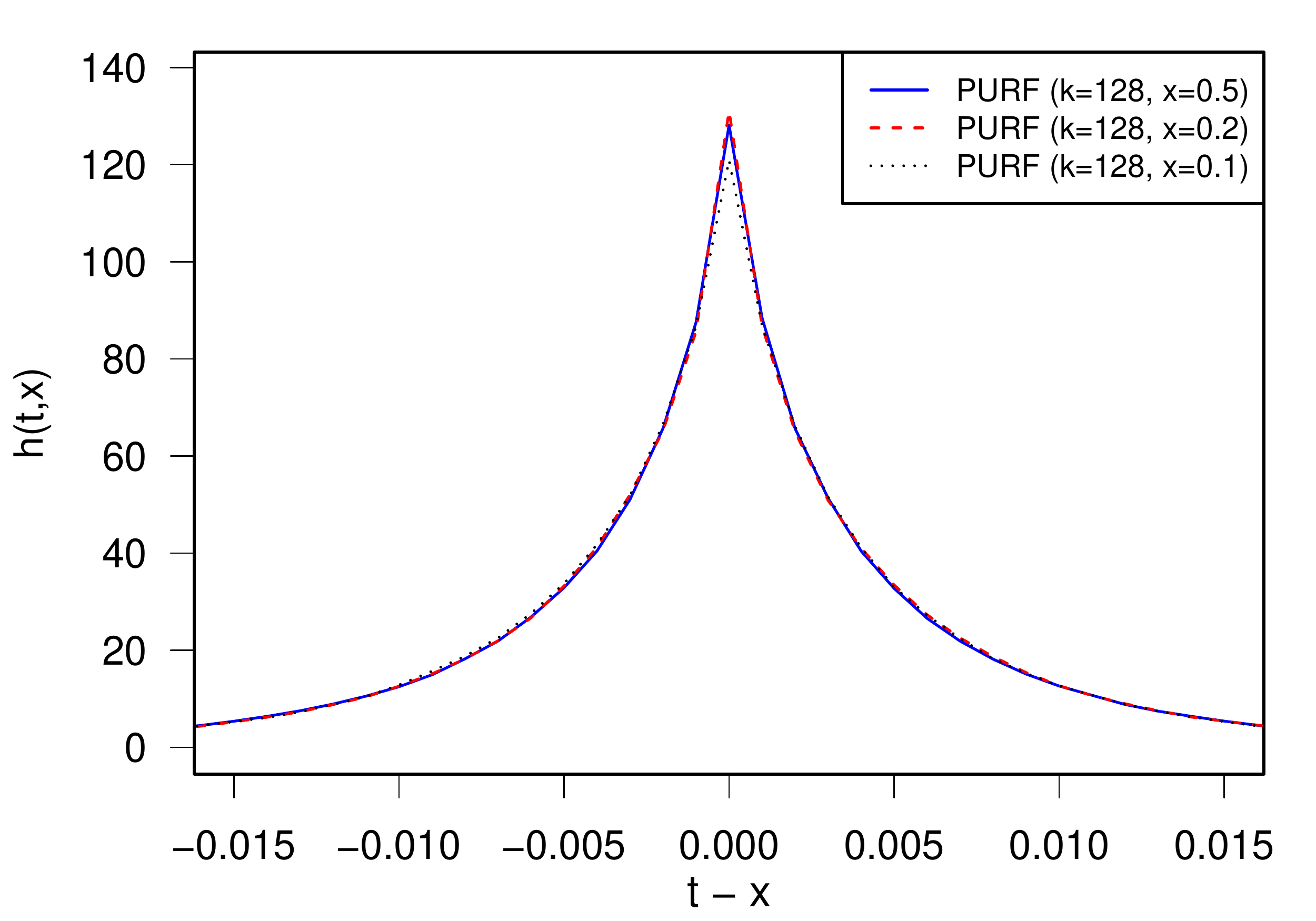}
    \caption{\label{fig.purf.inf-forest} Plot of
      $h^{\cUpurf{\nfeu}}(t,x)$ as a function of $t-x$ for $x\in
      \set{0.1, 0.2, 0.5}$. The values have been estimated by a
      Monte-Carlo approximation with $10\,000$ realizations of $\bU$.}
  \end{center}
\end{figure}

\subsection{Bias for twice differentiable functions}
As a corollary of Proposition~\ref{pro.bias.multidim.H3}, we get the
following estimates of the terms appearing in
decomposition~\eqref{eq.bias-cond.avec-notation} of the bias for the
PURF model.
\begin{corollary}\label{cor.bias.purf.H2}
  Let $\nfeu \geq 1$, $x \in [0,1)$ and assume
  \eqref{hyp.s-2-fois-derivable.alt} and \eqref{hyp.unif} hold true. 
  Then, 
  \begin{align}
    \label{eq.bias.purf.H2.Biasinfty.partout}
    0 \leq \Biasinfty{\cUpurf{\nfeu}}(x) 
    &\leq \frac{ \paren{s^{\prime}(x)}^2 }{2 \nfeu^2} + \frac{2\CHdeuxa^2}{\nfeu^4}
    \\
    \label{eq.bias.purf.H2.Vararbre.partout}
\mbox{and} \quad 
    0 \leq \Vararbre{\cUpurf{\nfeu}}(x)
    &\leq \frac{\paren{ s^{\prime}(x) }^2}{2 \nfeu^2}
    \enspace . 
  \end{align}
  Let $\nfeu\geq 27$ and $\epspurf \egaldef \frac{4\log \nfeu}{\nfeu}$.
  Then, for every $ x \in \croch{ \epspurf , 1 - \epspurf }$, 
  \begin{align}
    \label{eq.bias.purf.H2.Biasinfty.milieu}
    \Biasinfty{\cUpurf{\nfeu}}(x)
    &\leq \frac{2 \CHdeuxa^2}{\nfeu^4}  + \frac{ \paren{ s^{\prime}(x) }^2 }{2 \nfeu^6}
    \\
    \label{eq.bias.purf.H2.Vararbre.milieu}
\mbox{and} \quad 
    \absj{ \Vararbre{\cUpurf{\nfeu}}(x) - \frac{\paren{ s^{\prime}(x) }^2}{2 \nfeu^2} } 
    &\leq \frac{5}{\nfeu^3} \paren{ \absj{ s^{\prime}(x) } + \CHdeuxa }^2
    \enspace . 
  \end{align}
  As a consequence, if $\nfeu\geq 27$, 
  \begin{align}
    \label{eq.bias.purf.H2.Biasinfty.integrated}
    \int_0^1 \Biasinfty{\cUpurf{\nfeu}}(x) dx 
    &\leq 
    \frac{ 4 \norm{s^{\prime}}_{\infty}^2 \log(\nfeu)}{\nfeu^3} + \frac{2 \CHdeuxa^2}{\nfeu^4}  
    \\
    \label{eq.bias.purf.H2.Biasinfty.integrated-noborder}
    \int_{\epspurf}^{1-\epspurf} \Biasinfty{\cUpurf{\nfeu}}(x) dx
    &\leq \frac{2 \CHdeuxa^2}{\nfeu^4}  + \frac{ \norm{s^{\prime}}_{\infty}^2  }{2 \nfeu^6}
    \\
\label{eq.bias.purf.H2.Vararbre.integrated}
\mbox{and} \quad 
\absj{ \Vararbre{\cUpurf{\nfeu}} - \frac{1}{2 \nfeu^2}
      \int_0^1 \paren{ s^{\prime}(x) }^2 \, dx }
    & \leq \frac{6 \paren{ (\log(\nfeu) +1) \norm{s^{\prime}}_{\infty}^2  + \CHdeuxa^2} }{\nfeu^3} 
    \enspace . 
  \end{align}
\end{corollary}
Corollary~\ref{cor.bias.purf.H2} is proved in Section~\ref{sec.pr.purf.cor-H2}.

\subsection{Discussion: single tree vs. infinite forest}\label{subsec_discuss_purf_model}
Results of Corollary~\ref{cor.bias.purf.H2} involve the same rates
as in Corollary~\ref{cor.bias.toy.H2}, so, the discussion of
Section~\ref{subsec_discuss_toy_model} is also valid for the PURF
model (with boundaries of size $\epspurf$ instead of $\epstoy$) except
for the lower bound of the estimation error when avoiding border
effects. 
However, we conjecture that the result is the same than for
the toy model, but solving all technical issues for proving this 
is beyond the scope of the paper. 
So, to sum up, for $\nobs$ sufficiently large, we would again
have that the infinite forest decreasing rate smaller or equal to the square of
the single tree one. This implies that infinite forests would reach the
minimax rate of convergence for $C^2$ functions whereas single tree
does not.

\subsection{Tighter bound for three times differentiable functions}
When $s$ is smooth enough, the rates obtained in Corollary~\ref{cor.bias.purf.H2} are tight, as shown by the following corollary of Proposition~\ref{pro.bias.multidim.H3}. 
\begin{corollary}\label{cor.bias.purf.H3}
  Let $\nfeu\geq 27$ and assume \eqref{hyp.s-3-fois-derivable.alt} and \eqref{hyp.unif} hold true. 
  Then, for every $ x \in \croch{ \epspurf , 1 - \epspurf }$, 
  \begin{align} \label{eq.bias.purf.H3.Biasinfty}
    \absj{ \Biasinfty{\cU}(x) - \frac{ \paren{s^{\prime\prime}(x)}^2 }{4 \nfeu^4} } \leq \frac{ \paren{ 3 \CHtroisa + \frac{1}{27^2}  \absj{s^{\prime} (x)} + 2 \absj{ s^{\prime\prime} (x) } } ^2}{\nfeu^5}
  \end{align}
  and for every $ x \in (0,1) \backslash \croch{ \epspurf , 1 - \epspurf }$, 
  \begin{align} \label{eq.bias.purf.H3.Biasinfty.border}
    \absj{ \Biasinfty{\cU}(x) - \frac{ \paren{s^{\prime}(x)}^2 \paren{ x^{\nfeu+1} - \paren{1-x}^{\nfeu+1} } ^2 }{4 \nfeu^2} }
    \leq \frac{ \paren{ \CHtroisa + \absj{s^{\prime}(x) } + \absj{s^{\prime\prime}(x) } } ^2 }{2 \nfeu^3}
  \end{align}
  As a consequence, 
  \begin{align} \label{eq.bias.purf.H3.Biasinfty.integrated}
    \absj{ \int_0^1  \Biasinfty{\cU}(x) \, dx - \frac{\paren{s^{\prime}(0)}^2 + \paren{s^{\prime}(1)}^2 }{8 \nfeu^3} }
    \leq 
    \frac{ 6 \paren{ \CHtroisa + \norm{s^{\prime} }_{\infty} + \norm{ s^{\prime\prime} }_{\infty} }^2 \log \nfeu }{\nfeu^4}
\\
\mbox{and} \quad 
  \label{eq.bias.purf.H3.Biasinfty.integrated-noborder}
    \absj{ \int_{\epspurf}^{1-\epspurf} \Biasinfty{\cU}(x) \, dx - \frac{1}{4 \nfeu^4} \int_0^1 \paren{s^{\prime\prime}(x)}^2 \, dx } \leq \frac{ \paren{ 3  \CHtroisa +  \frac{1}{27^2} \norm{s^{\prime} }_{\infty} + 2 \norm{ s^{\prime\prime} }_{\infty} }^2 }{\nfeu^5}
  \end{align}
\end{corollary}
Corollary~\ref{cor.bias.purf.H3} is proved in Section~\ref{sec.pr.purf.cor-H3}.
As for the toy model, Corollary~\ref{cor.bias.purf.H3} implies that under \eqref{hyp.s-3-fois-derivable.alt}, 
a PURF with $\narb$ trees behaves as the
infinite forest as soon as $\narb \geq \nfeu^2$.

\section{Balanced purely random forests} \label{sec.multidim.BPRF}
We consider in this section the following multidimensional PRF model,
that we call Balanced Purely Random Forests (BPRF). 

\subsection{Description of the model} \label{sec.BPRF.description}
Let $\dimX \geq 1$ be fixed and $\X = [0,1)^{\dimX}$. 
We define the sequence $(\bU_\prof)_{\prof \in \N}$ of random partitions (or random trees) as follows: 
\begin{itemize}
\item $\bU_0 = [0,1)^{\dimX}$ a.s.
\item for every $\prof \in \N$, given $\bU_\prof$, we define $\bU_{\prof+1}$ by splitting each piece  $\lambda \in \bU_\prof$ into two pieces, where the split is made along some random direction (chosen uniformly over $\set{1, \ldots, \dimX}$) at some point chosen uniformly. \\
  Formally, given $\bU_\prof = \set{ \lambda_{1,\prof}, \ldots, \lambda_{2^{\prof},\prof} }$, let $L_{1,\prof}, \ldots, L_{2^{\prof},\prof} , Z_{1,\prof} , \ldots, Z_{2^{\prof} , \prof}$ be independent random variables, independent from $\bU_\prof$, such that 
  \[ \forall j \in \set{1, \ldots, 2^{\prof}} \, ,  \quad L_{j,\prof} \sim \mathcal{U}\paren{\set{1, \ldots, \dimX}} \quad \mbox{and} \quad Z_{j,\prof} \sim \mathcal{U}\paren{\croch{0,1}} \enspace . \]
  Then, $\bU_{\prof+1}$ is defined as follows: 
  for every $j \in \set{1, \ldots, 2^{\prof}}$, 
  $\lambda_{j,\prof} = \prod_{i=1}^{\dimX} [ A_i , B_i )$ is split into 
  \begin{align*} \lambda_{2j-1,\prof+1} &=  \prod_{i < L_{j,\prof}} [A_i , B_i)  \times [ A_{L_{j,\prof}} , \paren{ 1 - Z_{j,\prof} } A_{L_{j,\prof}} + Z_{j,\prof} B_{L_{j,\prof}} ) \times  \prod_{i > L_{j,\prof}} [ A_i , B_i ) \\
    \mbox{and} \quad 
    \lambda_{2j,\prof+1} &=  \prod_{i < L_{j,\prof}} [ A_i , B_i )  \times  [ \paren{ 1 - Z_{j,\prof} } A_{L_{j,\prof}} + Z_{j,\prof} B_{L_{j,\prof}}   , B_{L_{j,\prof}} ) \times  \prod_{i > L_{j,\prof}} [A_i , B_i )  \enspace . \end{align*}
\end{itemize}

Then, for every $\prof \in \N$, we get a random partition $\bU_\prof \sim \cUurt{\prof}$ of $\X = [0,1)^{\dimX}$ into $\nfeu = 2^{\prof}$ pieces. 

This model is very close to the UBPRF model introduced in
\cite{Bre:2000} and theoretically studied by \cite{Bia_Dev_Lug:2008}.
The only difference is that, at each step {\em all} sets of the
current partition are split in BPRF, resulting with balanced trees,
whereas in UBPRF, only {\em one} set (randomly selected with a uniform
distribution) of the current partition is split; see also
Section~\ref{sec.BPRF.discussion} for a comparison of these two
models.  

We also point out a similitude between BPRF and another model: 
\cite{Rah_Rec:2008} 
use $\cUurt{1}$ as random partitioning scheme, but without considering the same forest estimator at the end: 
instead of averaging the tree estimators with uniform weights as we do, 
\cite{Rah_Rec:2008} make a weighted average with data-driven weights. 

\subsection{Interpretation of the bias of the infinite forest}
As in Section~\ref{sec.purf.inf-forest}, we can try to interpret the bias of
the infinite forest for $\cUurt{\prof}$ as being equal to the bias
of an estimator close to a kernel estimator with ``kernel function''
$h^{\cUurt{\prof}}$ given by Eq.~\eqref{def.hcU.gal}.
Contrary to the PURF model case, $ t-x \mapsto h^{\cUurt{\prof}}(t,x)$
strongly depends on $x$, as shown by the left plot of
Figure~\ref{fig.urt+all.inf-forest}.  The right plot of
Figure~\ref{fig.urt+all.inf-forest} compares $h^{\cU}$ with
$\cU=\cUurt{\prof}$ to $\cU \in \set{ \cUtoy{\nfeu}, \cUpurf{\nfeu} }$
for a fixed $x=1/2$ and $\nfeu=2^{\prof}=128$: it turns out that
$h^{\cUtoy{}}$ and $h^{\cUpurf{}}$ are the narrowest---$h^{\cUpurf{}}$
appearing as a smooth approximation of $h^{\cUtoy{}}$---whereas
$h^{\cUurt{}}$ is significantly flatter than the others.
This relative flatness can explain the slower rates obtained for the
bias of the BPRF model in the next section.

\begin{figure}
  \begin{center}
    \includegraphics[width=0.45\textwidth]{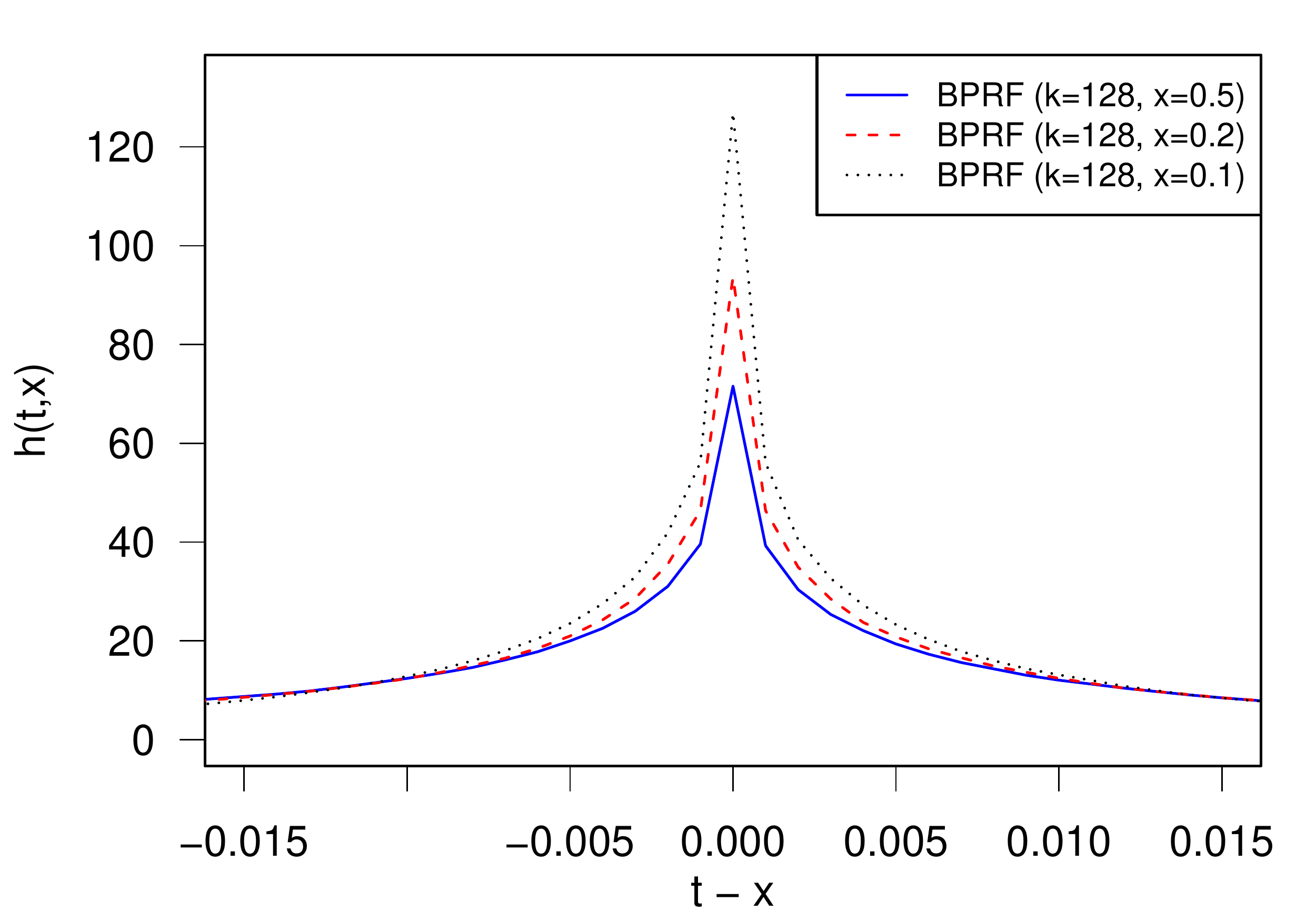}
    \hspace{0.05\textwidth}
    \includegraphics[width=0.45\textwidth]{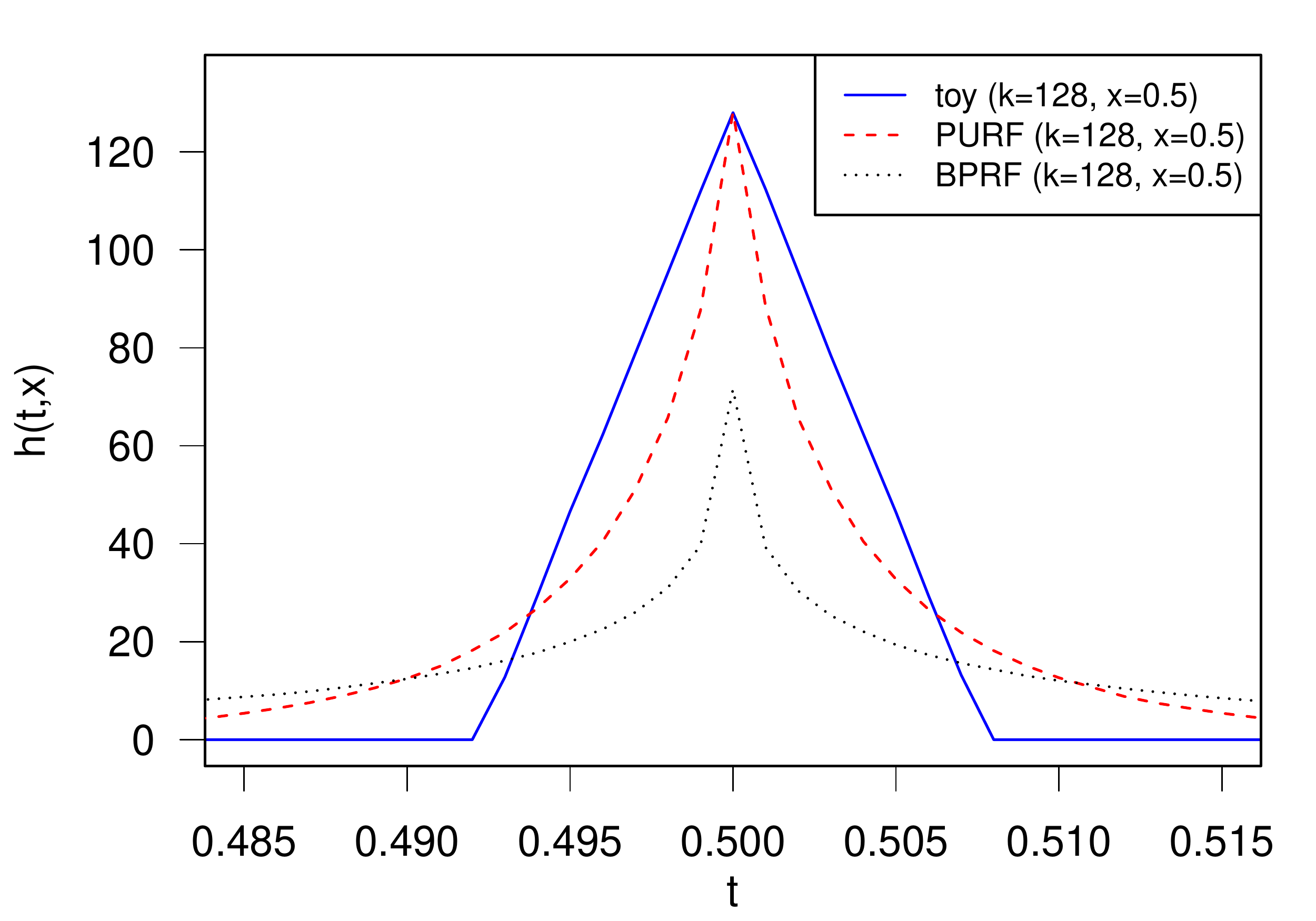}
    \caption{\label{fig.urt+all.inf-forest} Left: Plot of
      $h^{\cUurt{7}}(t,x)$ as a function of $t-x$ for $x\in \set{0.1,
        0.2, 0.5}$. Right: Plot of $h^{\cU}(t,x)$ with $\cU \in \set{
        \cUtoy{128} , \cUpurf{128} , \cUurt{7} }$ as a function of $t$
      for $x=0.5$.
      The values have been estimated by a Monte-Carlo approximation
      with $10\,000$ realizations of $\bU$.}
  \end{center}
\end{figure}

\subsection{Bias for twice differentiable functions}
As a corollary of Proposition~\ref{pro.bias.multidim.H3}, we get the
following estimates of the terms appearing in
decomposition~\eqref{eq.bias-cond.avec-notation} of the bias for the
BPRF model.
\begin{corollary}\label{cor.bias.BPRF.H2}
  Let $\prof\geq 2$  and assume \eqref{hyp.s-2-fois-derivable.alt} and
  \eqref{hyp.unif} hold true. 
  Then, for every $ x \in [0,1)^{\dimX}$, 
  \begin{align} 
    \Biasinfty{\cUurt{\prof}}(x) 
    \label{eq.bias.BPRF.H2.Biasinfty}
    &\leq \paren{1 - \frac{1}{2 \dimX}}^{2 \prof} \croch{ \frac{1}{2}  \paren{ \nabla s(x) \cdot (1-2x) }^2  + 2 \dimX^2 \CHdeuxa^2 }
    \\ 
    \notag 
    &\leq \paren{1 - \frac{1}{2 \dimX}}^{2 \prof} \croch{ \frac{\dimX}{2}  \sup_{x \in [0,1)^{\dimX}} \norm{ \nabla s(x) }_2^2   + 2 \dimX^2 \CHdeuxa^2 } 
    \\ 
    \label{eq.bias.BPRF.H2.Vararbre}
    &\hspace{-3cm}\mbox{and} \qquad \absj{\Vararbre{\cUurt{\prof}}(x) - \frac{1}{2} \paren{1 - \frac{1}{2 \dimX}}^{\prof} \sum_{i=1}^{\dimX} \croch{ \paren{ \frac{\partial s}{\partial x_i} (x) }^2 x_i (1-x_i) } } 
    \\ \notag 
    &\leq 
    \frac{\dimX}{4} \sup_{x \in [0,1)^{\dimX}} \norm{ \nabla s(x) }_2^2 \paren{1 - \frac{1}{2 \dimX}}^{2\prof}  
    \\
    \notag 
    &\qquad + \croch{  
      \dimX \max_{1 \leq i \leq \dimX} \paren{ \frac{\partial s}{\partial x_i} (x) }^2  
      + \dimX^2 \max_{1 \leq i \neq j \leq \dimX} \absj{ \frac{\partial^2 s}{\partial x_i \partial x_j} (x) }
      + 5 \dimX \CHdeuxa^2 } \\
    \notag 
    &\qquad \times \sqrt{ \paren{1 - \frac{1}{2 \dimX}}^{\prof} \paren{1 - \frac{2}{3 \dimX}}^{\prof} }
    \enspace . 
  \end{align}
  As a consequence, 
  \begin{align} 
    \int_{[0,1)^{\dimX}} \Biasinfty{\cUurt{\prof}}(x) dx 
    \label{eq.bias.BPRF.H2.Biasinfty.integrated}
    &\leq \paren{1 - \frac{1}{2 \dimX}}^{2 \prof} \croch{ \frac{1}{2}  \int_{[0,1)^{\dimX}} \paren{ \nabla s(x) \cdot (1-2x) }^2 \, dx + 2 \dimX^2 \CHdeuxa^2 }
    \\
    \notag 
    &\leq \paren{1 - \frac{1}{2 \dimX}}^{2 \prof} \croch{ \frac{\dimX}{2}  \sup_{x \in [0,1)^{\dimX}} \norm{ \nabla s(x) }_2^2   + 2 \dimX^2 \CHdeuxa^2 }
    \\ 
    \notag 
    &\hspace{-3cm}\mbox{and} \qquad \absj{ \int_{[0,1)^{\dimX}} \Vararbre{\cUurt{\prof}}(x) dx - \frac{1}{2} \paren{1 - \frac{1}{2 \dimX}}^{\prof} \sum_{i=1}^{\dimX} \int_{[0,1)^{\dimX}} \croch{ \paren{ \frac{\partial s}{\partial x_i} (x) }^2 x_i (1-x_i) } \, dx } 
    \\
    \label{eq.bias.BPRF.H2.Vararbre.integrated}
    &\leq 
    \frac{\dimX}{4} \sup_{x \in [0,1)^{\dimX}} \norm{ \nabla s(x) }_2^2 \paren{1 - \frac{1}{2 \dimX}}^{2\prof}  
    \\
    \notag 
    &\qquad + 
    \sup_{x \in [0,1)^{\dimX}} \croch{  
      \dimX \max_{ 1 \leq i \leq \dimX} \paren{ \frac{\partial s}{\partial x_i} (x) }^2  
      + \dimX^2 \max_{1 \leq i \neq j \leq \dimX} \absj{ \frac{\partial^2 s}{\partial x_i \partial x_j} (x) }
      + 5 \dimX \CHdeuxa^2 } \\
    \notag 
    &\qquad \times \sqrt{ \paren{1 - \frac{1}{2 \dimX}}^{\prof} \paren{1 - \frac{2}{3 \dimX}}^{\prof} }
    \enspace .
  \end{align}
\end{corollary}
Corollary~\ref{cor.bias.BPRF.H2} is proved in Section~\ref{sec.pr.multidim.BPRF.cor-H2}. 
Remark that contrary to the toy and PURF model, there is no border effect for the approximation error in the BPRF model. 

\subsection{Discussion: single tree vs. infinite forest}
\label{sec.BPRF.discussion}
We can now compare a single tree and an infinite forest for the toy model $\cUurt{\prof}$, first in terms of approximation error for a given $\prof$, then in terms of risk for a well-chosen $\prof$. 
In this section, we assume \eqref{hyp.s-2-fois-derivable.alt} and~\eqref{hyp.unif} hold true. 
\subsubsection*{Approximation error}
Let
\[ \alphaurt \egaldef \frac{- \log\paren{ 1 - \frac{1}{2 \dimX}} }
{\log(2)} > 0 \] be such that
\[ \nfeu^{-\alphaurt} = \paren{ 1 - \frac{1}{2 \dimX} }^{\prof} \quad
\mbox{when} \quad \nfeu = 2^{\prof} 
\, . \]
Corollary~\ref{cor.bias.BPRF.H2} and Eq.~\eqref{eq.comp-biais-arbre-foret.gal}
allow to compare the approximation errors of a single tree and of an
infinite forest: 
\begin{equation}
\label{eq.BPRF.approx-bornes}
\Biasinfty{\cU} \leq L(s,\dimX) \nfeu^{-2\alphaurt} \quad
\mbox{whereas} \quad \Bias_{\cU,1} \geq L(s,\dimX) \nfeu^{-\alphaurt}
- L(s,\dimX) \nfeu^{-2 \alphaurt}
\enspace . 
\end{equation}
 
Therefore, considering an infinite forest instead of a single tree
decreases the approximation error from an order of magnitude, and not
only from a constant factor when the height of the trees tends to
infinity. 
We emphasize that, as in
Section~\ref{subsec_discuss_toy_model} and
\ref{subsec_discuss_purf_model}, we get an infinite forest bias
decreasing rate smaller or equal to the square of the single tree one.

Nevertheless, the single tree bias rate is strictly slower than the
bias rate of a classical regular partitioning estimate (with a cubic
partition in $k$ sets), which is $\nfeu^{-2/\dimX}$ \citep[see
e.g. Chapter~4 in][]{Gyo_Las_Krz_Wal:2002}. Indeed, we have that for all $\dimX \geq 1$, 
\[ \alpha \leq \frac{1}{2 \log(2) \dimX } < \frac{2}{\dimX} \]
since $\log(1+u) \leq u$ for all $u > -1$. 

\subsubsection*{Risk bounds for a well-chosen $\prof$}
The above controls on the approximation errors imply controls 
on the statistical risk of the estimators built on a single
tree and on an infinite forest, respectively.
Indeed, for all $\narb \in [1,+\infty]$, if $\nobs \geq 1$ data points are
available, the statistical risk of the estimator built upon a random
forest of $\narb$ trees with $ \nfeu = 2^{\prof} \geq 2$ leaves can 
be bounded by Eq.~\eqref{eq.risk-cond.three.terms.avec-notation} and
Proposition~\ref{pro.maj-variance}. 
In order to apply Proposition~\ref{pro.maj-variance}, we need the following lemma.
\begin{lemma} \label{le.BPRF} 
  Let $\prof \geq 0$. Then, 
  \begin{equation} \label{eq.le.BPRF.somme-diametres} \E_{\bU \sim
      \cUurt{\prof}} \croch{ \sum_{\lambda \in \bU} \paren{
        \diam_{L^2}(\lambda) }^2 } = 2^{\prof} \paren{ 1 - \frac{2}{3
        \dimX} }^{\prof}
  \end{equation}
and for every $u>0$, 
\begin{equation} \label{eq.BPRF.somme-exp-pl.1}
\E_{\bU \sim \cUurt{\prof}} \croch{ \sum_{\lambda \in \bU} \exp\paren{- \nobs |\lambda| } } 
\leq 2^{\prof} \croch{ \frac{1}{u} + \paren{ 1 - \frac{1}{u}} 
\exp\paren{-\nobs e^{-\paren{ \prof +\sqrt{u \prof}}} } 
}
\enspace . 
\end{equation}
In particular, if $\nobs \geq \exp\paren{\prof+\sqrt{5 \prof}}$, 
\begin{equation} \label{eq.BPRF.somme-exp-pl.2}
\E_{\bU \sim \cUurt{\prof}} \croch{ \sum_{\lambda \in \bU} \exp\paren{- \nobs |\lambda| } } 
\leq 2^{\prof} \kappa 
\quad \mbox{where} \quad 
\kappa \egaldef \frac{  1 + 4e^{-1} }{5} < \frac{1}{2} 
\enspace . 
\end{equation}
\end{lemma}
Lemma~\ref{le.BPRF} is proved in Section~\ref{sec.pr.le.BPRF}. 
The proof of Lemma~\ref{le.BPRF} in Section~\ref{sec.pr.le.BPRF} also shows the 
volume of each element of a partition $\bU \sim \cUurt{\prof}$ is typically of order 
$\exp\paren{-\prof \pm L\sqrt{\prof}}$, so it is hopeless to consider values of $\prof$ such that this typical volume is smaller than $1/\nobs$. 
Hence, throughout this subsection, for comparing risks with a well-chosen $\prof$, 
we only consider values of $\prof$ such that 
\begin{equation} 
\label{eq.BPRF.hyp-prof}
\nobs \geq \exp\paren{\prof+\sqrt{5\prof}}
\quad \Leftrightarrow \quad 
\prof \leq \paren{ \sqrt{\frac{5}{4} +\log \nobs} - \frac{\sqrt{5}}{2} }^2 
\enspace . 
\end{equation}

Remark that under assumption \eqref{hyp.s-2-fois-derivable.alt}, $s$ is
$K$-Lipschitz with respect to the $L^2$ distance on $\X$ with $K =
\sup_{x \in \X} \norm{\nabla s(x)}_2 = \norm{\nabla s}_{\infty,2}$. 
So, Proposition~\ref{pro.maj-variance} shows that for the BPRF model 
with trees having $\nfeu = 2^{\prof}$ leaves, 
if $\nobs \geq 1$ data points are available 
and if Eq.~\eqref{eq.BPRF.hyp-prof} holds true, 
\begin{align*}
\E\croch{\paren{ \ERM_{\infty} \paren{X; D_{\nobs}} - s(X) }^2} 
&\leq 
\Biasinfty{\cUurt{\prof}} 
+ \frac{2 \sigma^2 \nfeu}{\nobs}
+ \frac{2 \norm{\nabla s}_{\infty,2}^2}{\nobs \nfeu^{\betaurt}} 
+ \frac{\norm{s}_{\infty}^2 \nfeu }{\nobs} \croch{ \frac{1}{u} + \exp\paren{ - \nobs \nfeu^{- (1 + \sqrt{u/p}) / \log 2} }}
\end{align*}
for every $u \geq 1$, where 
\[ 
\betaurt \egaldef \frac{- \log \paren{ 2 \paren{ 1 -
      \frac{2}{3 \dimX}}} } {\log(2)} 
\enspace . \]
So, since $\Biasinfty{\cUurt{\prof}}  \leq L(s,d) \nfeu^{-2\alphaurt}$, 
if we are able to choose the number of leaves $\nfeu =2^{\prof}$
optimally (with an estimator selection procedure, such as
cross-validation), the risk of the infinite forest estimator is upper
bounded as follows: 
\begin{align}
\notag 
  &\qquad \E\croch{\paren{ \ERM_{\infty} \paren{X; D_{\nobs}} - s(X) }^2}
  \\
  &\leq 
L(s,d)  \inf_{ u \geq 1 , \, \nfeu = 2^{\prof}, \, 0 \leq \prof \leq \paren{ \sqrt{\frac{5}{4} +\log \nobs} - \frac{\sqrt{5}}{2} }^2  } 
 \set{  
	\nfeu^{-2\alphaurt} 
	+ \frac{\sigma^2 \nfeu}{\nobs} 
	+ \frac{1}{\nobs \nfeu^{\betaurt}}
	+ \frac{\nfeu}{\nobs} 
		 \croch{ \frac{1}{u} + \exp\paren{ - \nobs \nfeu^{- (1 + \sqrt{u/\prof}) / \log 2} }}
  } 
  \label{eq.BPRF.maj-risk-for-infinie.1}
\enspace . 
\end{align}
Now, for upper bounding the infimum, two cases must be distinguished: 
(i) when $\dimX \leq 3$,  so that $1/(1+2\alphaurt) < \log 2$, 
%
and (ii) when $\dimX \geq 4$, so that $1/(1+2\alphaurt) > \log 2$. 

In case (i), some nonnegative integer 
$\prof^* \leq \paren{ \sqrt{\frac{5}{4} +\log \nobs} - \frac{\sqrt{5}}{2} }^2 $ exists such that 
\[ \nfeu^* = 2^{\prof^*} \in \croch{\paren{ \frac{\nobs}{\sigma^2} }^{1/(1+2\alphaurt)},
      2 \paren{ \frac{\nobs}{\sigma^2} }^{1/(1+2\alphaurt)}}
\] if $\nobs \geq L(\sigma^2)$. 
Since $\alphaurt \geq \log(6/5)/\log(2) > (1/\log(2)-1)/2$, for some (small enough) numerical constants $\delta_1,\delta_2>0$,  if $n \geq L(\sigma^2)$, 
\[ 
\nobs \nfeu^{* \, - \frac{1+\delta_1}{ \log 2 }} 
\geq L \nobs \paren{\frac{\nobs}{\sigma^2}}^{- \frac{1+\delta_1}{(1+2\alphaurt) \log(2) } } 
\geq \nobs^{\delta_2}
\enspace , 
\]
taking $u=\delta_1^2 \prof^*$ in Eq.~\eqref{eq.BPRF.maj-risk-for-infinie.1} yields 
\begin{align}
  &\qquad \E\croch{\paren{ \ERM_{\infty} \paren{X; D_{\nobs}} - s(X) }^2}
  \\
  &\leq 
  L(s, d) \croch{ 
  \paren{\frac{\sigma^2}{\nobs}}^{\frac{2\alphaurt}{2\alphaurt + 1}}
  + L(\sigma^2) \nobs^{- \frac{2\alphaurt}{2\alphaurt + 1}} 
		 \croch{ \frac{1}{\delta_1^2 \prof^*} + \exp\paren{ - \nobs^{\delta_2}}}
}
\notag 
\\
&\leq L(s, d) \paren{\frac{\sigma^2}{\nobs}}^{\frac{2\alphaurt}{2\alphaurt + 1}}
\notag 
\end{align}
as soon as $\nobs \geq L(\sigma^2)$. 

In case (ii), a similar reasoning with some integer 
\[ \prof^* \in \left( \paren{ \sqrt{\frac{5}{4} +\log \nobs} - \frac{\sqrt{5}}{2} }^2  -1 , 
\paren{ \sqrt{\frac{5}{4} +\log \nobs} - \frac{\sqrt{5}}{2} }^2  \right] \]
yields 
\begin{align}
\E\croch{\paren{ \ERM_{\infty} \paren{X; D_{\nobs}} - s(X)
    }^2} 
&\leq L(s,\dimX,\sigma^2) \nobs^{-2 \alphaurt \log 2} 
2^{2 \alphaurt \sqrt{5} \sqrt{\log n + \frac{5}{4}}}
\enspace . 
\notag 
\end{align}
In particular, we get a rate of order $\nobs^{-(2 \alphaurt \log 2-\delta)}$ for every $\delta>0$, which is slightly worse than the rate $n^{- 2\alphaurt / (2 \alphaurt + 1)}$ for $\dimX \geq 4$ since $\log 2 \leq 1/(1+2 \alphaurt)$.

\bigskip

For lower bounding the risk of a single tree, we apply 
Eq.~\eqref{eq.pro.maj-variance.arbre-gal-min} in Proposition~\ref{pro.maj-variance}.
By Eq.~\eqref{eq.BPRF.approx-bornes}, 
if $\prof \geq \prof_0 = L(s,\dimX)$ and $\nfeu = 2^{\prof}$, 
\[ 
\Bias_{\cUurt{\prof},1}
\geq 
L(s,\dimX) \nfeu^{-\alphaurt} - L(s,\dimX) \nfeu^{-2\alphaurt}
\geq L(s,\dimX) \nfeu^{-\alphaurt}
\]
so that Eq.~\eqref{eq.risk-cond.three.terms.avec-notation}, Proposition~\ref{pro.maj-variance} and 
Eq.~\eqref{eq.BPRF.somme-exp-pl.2} in Lemma~\ref{le.BPRF} imply, 
if $\bU \sim \cUurt{\prof}$ with $\prof_0 \leq \prof \leq \paren{ \sqrt{\frac{5}{4} +\log \nobs} - \frac{\sqrt{5}}{2} }^2$, 
\begin{align}
  \E\croch{\paren{ \ERM\paren{X; \bU; D_{\nobs}} - s(X) }^2} 
&\geq
  \inf_{\nfeu=2^{\prof}, \prof_0 \leq \prof \leq \paren{ \sqrt{\frac{5}{4} +\log \nobs} - \frac{\sqrt{5}}{2} }^2 } 
  \set{
    L(s,\dimX) \nfeu^{-\alphaurt} + \frac{\sigma^2  \nfeu \paren{ 1 - 2 \kappa } }{\nobs} }
\enspace . 
 \label{eq.BPRF.min-risk-arbre.1}
\end{align}
Here, again, we must distinguish the cases (i) $\dimX \leq 3$ and (ii) $\dimX \geq 4$. 
If $\dimX \leq 3$, 
by Lemma~\ref{le.opt-risk}, Eq.~\eqref{eq.BPRF.min-risk-arbre.1} shows that 
if in addition $\nobs \geq L(\dimX,\sigma^2)$, 
\begin{align*}
  \E\croch{\paren{ \ERM\paren{X; \bU; D_{\nobs}} - s(X) }^2} 
&\geq L(s,\dimX) \paren{ \frac{\sigma^2}{\nobs} }^{\alphaurt/
(\alphaurt + 1)} 
\enspace . 
\end{align*}
If $\dimX \geq 4$, Eq.~\eqref{eq.BPRF.min-risk-arbre.1} shows that 
\begin{align*}
  \E\croch{\paren{ \ERM\paren{X; \bU; D_{\nobs}} - s(X) }^2} 
&\geq
  L(s,\dimX)  \inf_{0 \leq \nfeu \leq 2^{5/2} \nobs^{\log(2)}} 
  \set{
    \nfeu^{-\alphaurt} + \frac{\sigma^2  \nfeu }{\nobs} }
    \\
&\geq 
    L(s,\dimX) \nobs^{-\alphaurt \log(2)}
\end{align*}
for $n \geq L(\sigma^2,\dimX)$, since the function $x \to x^{-\alphaurt} + \frac{\sigma^2 x}{\nobs}$ is then decreasing on $(0, (\nobs \alphaurt /\sigma^2)^{1/(\alphaurt+1)}]$ 
and 
\[ 
\frac{1}{\alphaurt+1} \geq \frac{1}{1 + \frac{\log(8/7)}{\log(2)}} > \log (2)
\enspace . 
\]

\bigskip

So, in both cases ($\dimX \leq 3$ or $\dimX \geq 4$), the infinite forest 
has a faster rate of convergence (in terms of risk) than a single tree. 
But, even with an infinite forest with $\dimX\leq 3$, since 
\[ 
\frac{2 \alphaurt}{2 \alphaurt + 1} \leq \frac{1}{1 + 2 \log(2) \dimX } < \frac{4}{4+\dimX} 
\enspace , 
\]
the rate obtained is slower than the minimax rate $\nobs^{-4/(\dimX+4)}$ over the set of $C^2$
functions \citep[see e.g.][]{Gyo_Las_Krz_Wal:2002}.

Intuitively, the BPRF model is not minimax because it is not adaptive
enough.  Indeed, the partitioning process splits each set of the
current partition regardless of its size: so a relatively small set is
still split the same number of times than a relatively large set. 
We
conjecture that a partitioning scheme with a random choice of the next
set to be split, with a probability of choosing each set proportional 
to its size---as in the PURF model, see Section~\ref{sec:rf-prf-partitioning}---, would be better
and could reach the minimax rate for $C^2$ functions. 
This is proved for $\dimX=1$ in Section~\ref{sec.purf}.

Finally, we note that the UBPRF model~\ref{algo.ubprf}
would certainly suffer from the same lack of adaptivity because the next 
set to be split is chosen with a uniform distribution on all
sets. So, this model would certainly not be minimax either, and we conjecture
that it would be even worse than the BPRF model.

\subsection{Tighter bound for three times differentiable functions}
The bounds in Corollary~\ref{cor.bias.BPRF.H2} are tight when $s$ is smooth enough, as shown by the following corollary of Proposition~\ref{pro.bias.multidim.H3}. 
\begin{corollary}\label{cor.bias.BPRF.H3}
  Let $\prof \geq 2$ and assume \eqref{hyp.s-3-fois-derivable.alt} and
  \eqref{hyp.unif} hold true. 
  Then, for every $ x \in [0, 1)^{\dimX}$, 
  \begin{equation} \label{eq.bias.BPRF.H3.Biasinfty}
    \begin{split}
      & \quad \absj{ \Biasinfty{\cUurt{\prof}}(x) - \frac{1}{4} \paren{1 - \frac{1}{2 \dimX}}^{2 \prof} \paren{  \nabla s(x) \cdot (1-2x) + \sum_{i=1}^{\dimX} \croch{ \frac{\partial^2 s}{\partial x_i^2} (x) x_i (1-x_i) } }^2 }  
      \\
      &\leq 
      6 \dimX^4 \paren{1 - \frac{2}{3 \dimX}}^{3 \prof / 4}  \paren{1 - \frac{1}{2 \dimX}}^{\prof}  \paren{  \norm{\nabla s(x)}_2^2  +   \max_{i,j} \paren{\frac{\partial^2 s}{\partial x_i \partial x_j} (x)}^2  +\CHtroisa^2  }
      \enspace .
    \end{split}
  \end{equation}
  As a consequence, 
  \begin{equation} \label{eq.bias.BPRF.H3.Biasinfty.integrated}
    \begin{split}
      &\hspace{-2cm} \absj{ \int_{[0,1)^{\dimX}} \Biasinfty{\cUurt{\prof}}(x) \, dx  - \frac{1}{4} \paren{1 - \frac{1}{2 \dimX}}^{2 \prof} \int_{[0,1)^{\dimX}} \paren{  \nabla s(x) \cdot (1-2x) + \sum_{i=1}^{\dimX} \croch{ \frac{\partial^2 s}{\partial x_i^2} (x) x_i (1-x_i) } }^2  \, dx } \\
      &\hspace{-1cm}\leq 
      6 \dimX^4 \paren{1 - \frac{2}{3 \dimX}}^{3 \prof / 4}  \paren{1 - \frac{1}{2 \dimX}}^{\prof}  \paren{ \sup_{x \in [0,1)^{\dimX}} \norm{\nabla s(x)}_2^2  +  \sup_{x \in [0,1)^{\dimX}} \max_{i,j} \paren{\frac{\partial^2 s}{\partial x_i \partial x_j} (x)}^2  + \CHtroisa^2  }
      \enspace .
    \end{split}
  \end{equation}
\end{corollary}
Corollary~\ref{cor.bias.BPRF.H3} is proved in Section~\ref{sec.pr.multidim.BPRF.cor-H3}.

\subsection{Size of the forest}
As for the previous models, under \eqref{hyp.s-3-fois-derivable.alt}, Corollary~\ref{cor.bias.BPRF.H3} and
Eq.~\eqref{eq.pro.bias_decomposition} show a BPRF forest of size $\narb$ has
an approximation error of the same order of magnitude as an infinite
forest when
\[ 
\narb \geq \nfeu^{\alpha} = 2^{\alpha \prof} = \paren{1 - \frac{1}{2 \dimX}}^{-\prof} 
\, . 
\]

\section{Simulation experiments}\label{sec.simu}
In order to illustrate mathematical results from previous sections, we
lead some simulation experiments with R \citep{R_Core:2014}, focusing
on approximation errors, as defined in Section~\ref{sec.general.risk}.
We consider the models from
Sections~\ref{sec_toy_model}--\ref{sec.multidim.BPRF} (toy, PURF,
BPRF), with $\dimX=1$ for toy and PURF, and $\dimX \in \{1,5,10\}$ for
BPRF.
In addition, we consider a PRF model discussed in Section~3 of
\cite{Bia:2012}, that we call Hold-out RF in the following. 
Hold-out RF is the original RF model~\ref{model.original} except that the tree partitioning is performed
using an extra sample $D_{\nobs}'$, independent from the learning
sample $D_{\nobs}$. 
As a consequence, assumption~\eqref{hyp.purely-random} holds for
the Hold-out RF model, so decomposition~\eqref{eq.risk-cond.three.terms} is valid
and we can compute the corresponding approximation error, as a
function of the number $\narb$ of trees in the forest.

\subsection{Framework}

For all experiments, we take the input space $\X=[0,1)^{\dimX}$ and
suppose that \eqref{hyp.unif} holds. We choose the following
regression functions:
\begin{itemize}
\item \textbf{sinusoidal} (if $\dimX=1$): $x \mapsto \sin(2 \pi x)$,
\item \textbf{absolute value} (if $\dimX=1$): $x \mapsto \absj{
    x-\frac{1}{2} }$,
\item \textbf{sum} (for any $\dimX \geq 1$): $x \mapsto
  \sum_{j=1}^\dimX x_j$,
\item \textbf{Friedman1} (for any $\dimX\geq 5$): 
  \[ x \mapsto \textbf{1/10} \times \croch{ 10 \sin(\pi x_1 x_2) + 20 (x_3 -
  0.5)^2 + 10 x_4 + 5 x_5 } \] which is proportional to the
  \textbf{Friedman1} function that was introduced by \citet{Fri:1991}.
  Here we add the scaling factor $1/10$ in order to have a function with
  a range comparable to that of \textbf{sum}.
\end{itemize}
For all PRF models, we choose $\nfeu$, the number of leaves (minus one for toy
and PURF), among $\set{ 2^5, 2^6, 2^7, 2^8 , 2^9 }$; 
the last value $2^9$ is sometimes removed for computational reasons.

Quantities $\Bias_{\cU,1}$ and $\Biasinfty{\cU}$ are estimated 
by Monte-Carlo approximation using:
\begin{itemize}
\item $1000$ realizations of $X$,
\item for $\Bias_{\cU,1}$, $500$ realizations of $\bU$,
\item for $\Biasinfty{\cU}$, $\nfeu^2$ realizations of $\bU$ for toy
  and PURF models, $\nfeu^{2\alphaurt}$ realizations for BPRF model
  with $\alphaurt = - \log(1-1/(2 \dimX))/\log(2)$ 
  (which ensures our estimation of the convergence rates is precise enough, 
  according to our theoretical results), 
  and $\nfeu^2$
  realizations for Hold-out RF model, which empirically appears to be
  sufficient for estimating the convergence rates correctly for this RF model.
\end{itemize}
Furthermore, for each computation of $\Bias_{\cU,1}$ and
$\Biasinfty{\cU}$ we add some ``borderless'' estimations of the bias, 
that is, integrating only over $x \in [ \epsilon , 1 - \epsilon]$ with 
$\epsilon = \epstoy_{\nfeu}$ or $\epspurf$ depending on the model. 

\bigskip

In addition, for the Hold-out RF model:
\begin{itemize}
\item we simulate the $D_{\nobs}'$ sample with $\nobs = \nfeu^2$ (for
  each value of $\nfeu$) and choose a gaussian random noise with
  variance $\sigma^2 = 1/16$,
\item we use the \texttt{randomForest} R package \citep{Lia_Wie:2002}
  to build the trees on the sample $D_{\nobs}'$: we use parameters
  \texttt{maxnodes} (to control the number of leaves) and
  \texttt{ntree} (to set the number of trees), and take the default
  values for all other parameters (in particular \texttt{mtry}).
\end{itemize}

\medskip

Finally, for each scenario, we plot the bias as a function of $\nfeu$ in $\log_2$-$\log_2$ scale, 
and estimate the slope of the plot by fitting a simple linear model 
in order to get an approximation of
the convergence rates.

\subsection{One-dimensional input space}
\label{sec.simu.one-dim}
We consider in this subsection the one-dimensional case ($\dimX=1$). 
Figure~\ref{fig.sinus} shows results for the \textbf{sinusoidal}
regression function.
\begin{figure}[!ht]
  \begin{center}
  \begin{minipage}{0.45\textwidth}
    \includegraphics[width=\textwidth]{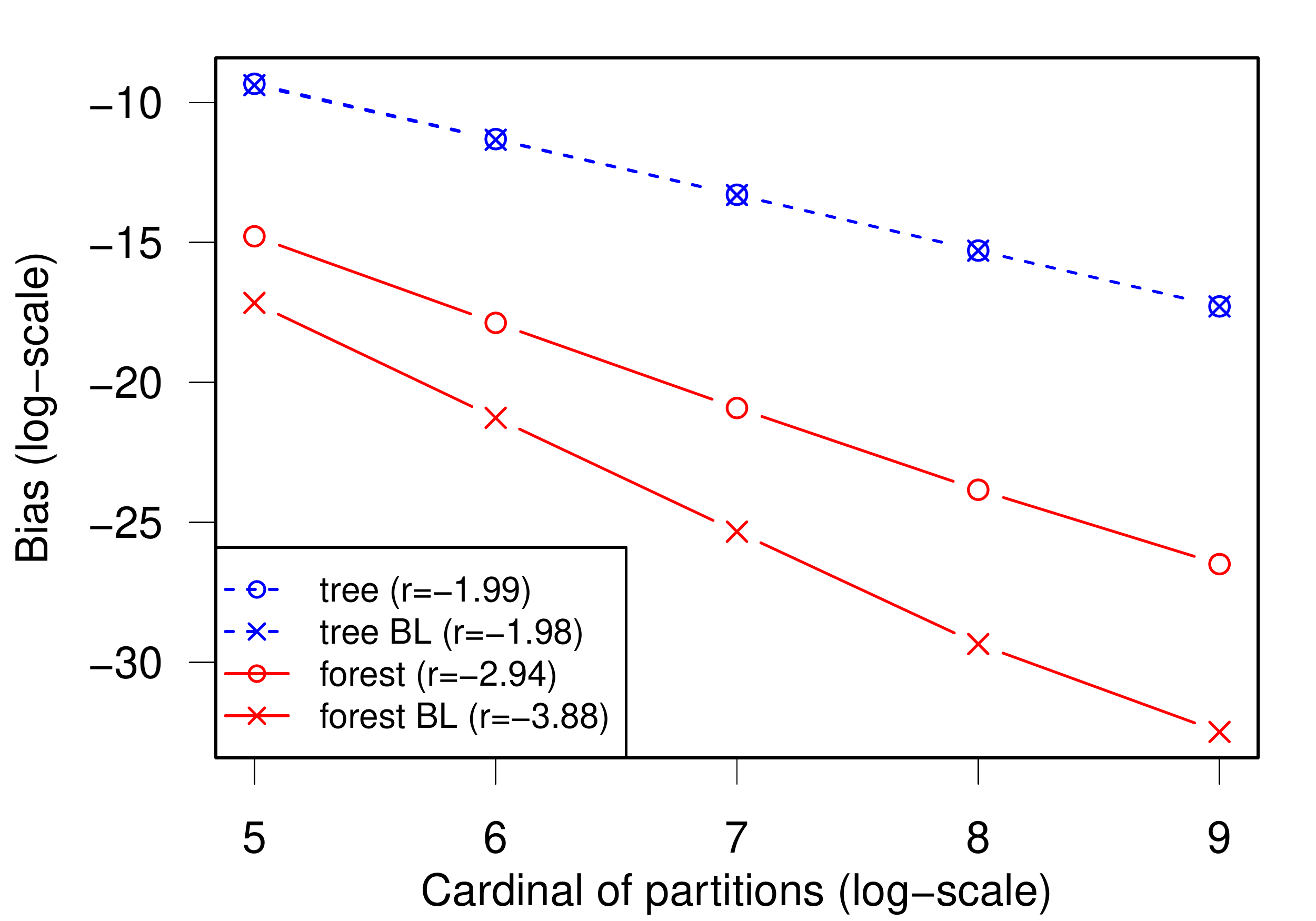}\\
    \centerline{(a) toy, $\dimX=1$}
    \end{minipage}
    \hspace{0.01\textwidth}
  \begin{minipage}{0.45\textwidth}
  \includegraphics[width=\textwidth]{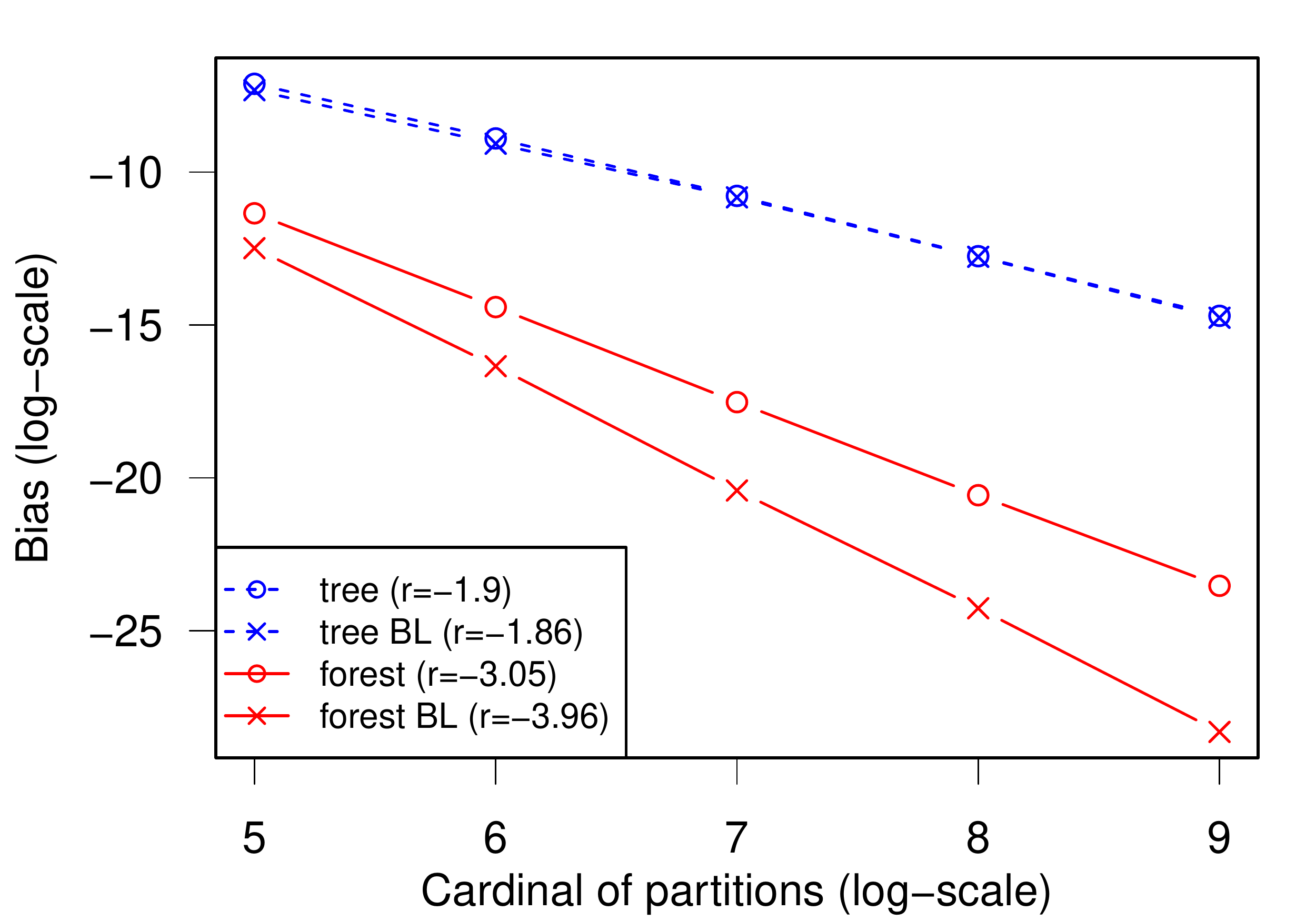}\\
  \centerline{(b) PURF, $\dimX=1$}
  \end{minipage} \\
  \begin{minipage}{0.45\textwidth}
    \includegraphics[width=\textwidth]{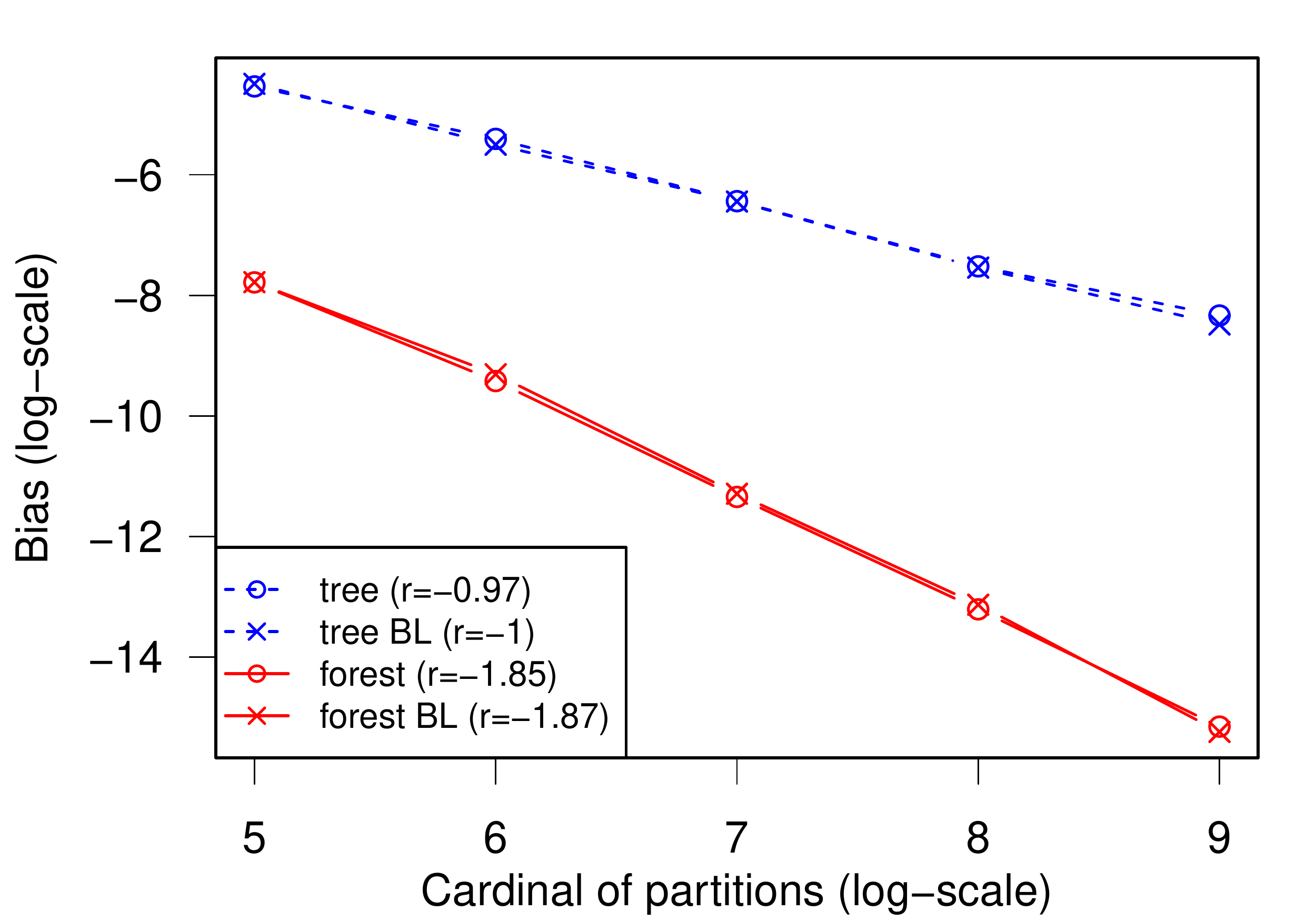}\\
    \centerline{(c) BPRF, $\dimX=1$}
    \end{minipage}
    \hspace{0.01\textwidth}
  \begin{minipage}{0.45\textwidth}
    \includegraphics[width=\textwidth]{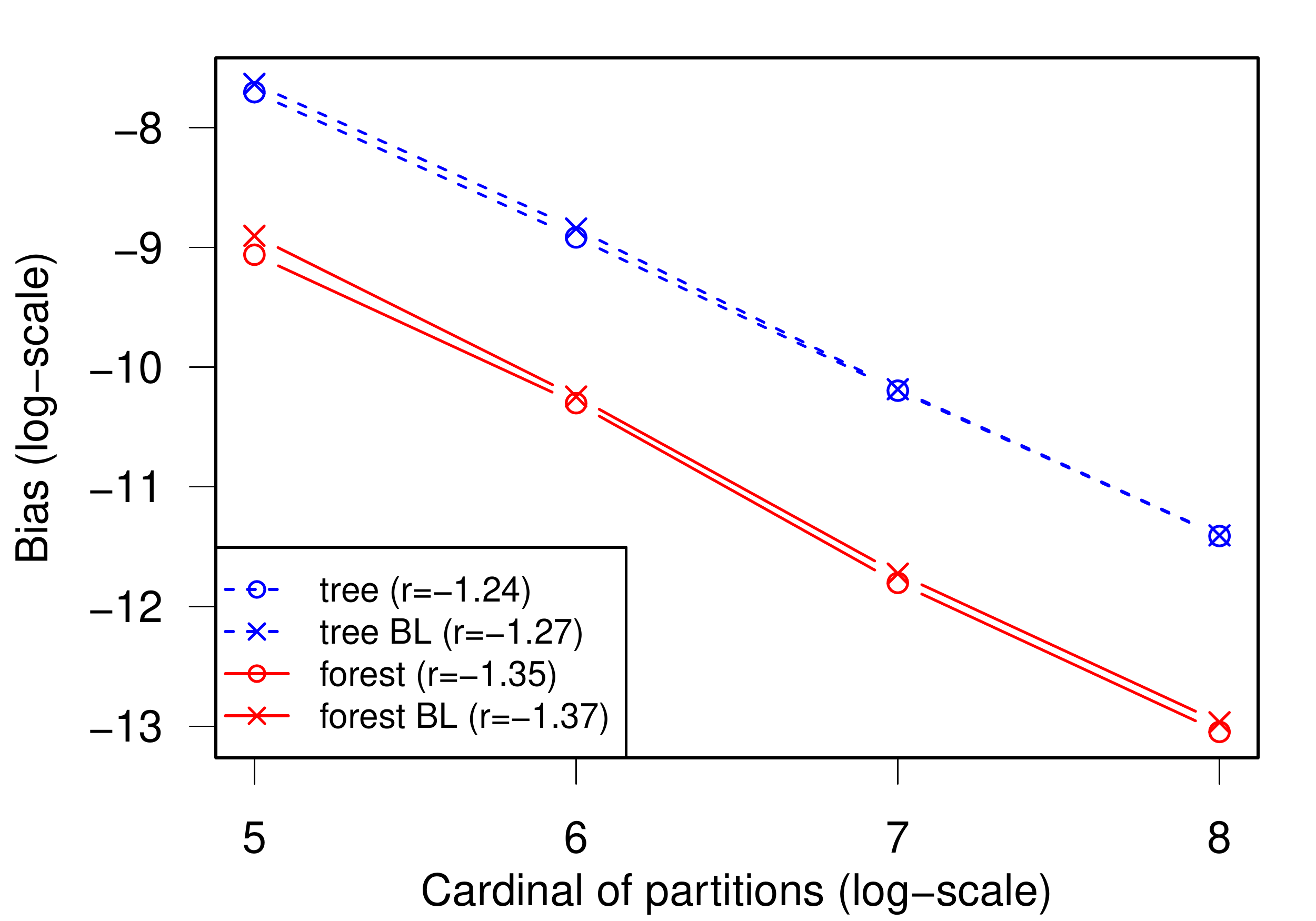}\\
    \centerline{(d) Hold-out RF, $\dimX=1$}
    \end{minipage}
    \caption{\label{fig.sinus} Plot of $\Bias_{\cU,1}$ and
      $\Biasinfty{\cU}$ (in $\log_2$-scale) against $\nfeu$ (in
      $\log_2$-scale) for (a) toy, (b) PURF, (c) BPRF and (d)
      Hold-out RF models, for the \textbf{sinusoidal} regression
      function. BL corresponds to borderless situations and $r$
      denotes the slope of a linear model fitted to the scatter
      plot.}
  \end{center}
\end{figure}
Plots are in $\log_2$-$\log_2$ scale, so as expected we obtain linear
behaviors. 
For toy and PURF models (top graphs) we get decreasing
rates very close to what can be expected from Sections~\ref{sec_toy_model}--\ref{sec.purf}: 
$\nfeu^{-2}$ for trees (with or without borders),
$\nfeu^{-3}$ for forests and $\nfeu^{-4}$ for borderless
forests. 
Similarly, we get the right decreasing rates for BPRF model
(bottom left graph): indeed, if $\dimX=1$ then $\alphaurt = 1$, so
trees and forests rates are respectively $\nfeu^{-1}$ and
$\nfeu^{-2}$. 
For Hold-out RF model we get rates about $\nfeu^{-1.25}$
for trees and $\nfeu^{-1.35}$ for forests as expected from Section~\ref{sec.multidim.BPRF}. 
These rates are
surprisingly slow (in particular compared to toy and PURF models) 
and a forest does not bring much improvement compared to a single tree. 
But as shown in the next section, the one-dimensional case
is not the best framework for the Hold-Out-RF model compared to other PRF models.

Results for the \textbf{absolute value} regression function are
presented in Figure~\ref{fig.abs}.
\begin{figure}[!ht]
  \begin{center}
  \begin{minipage}{0.45\textwidth}
    \includegraphics[width=\textwidth]{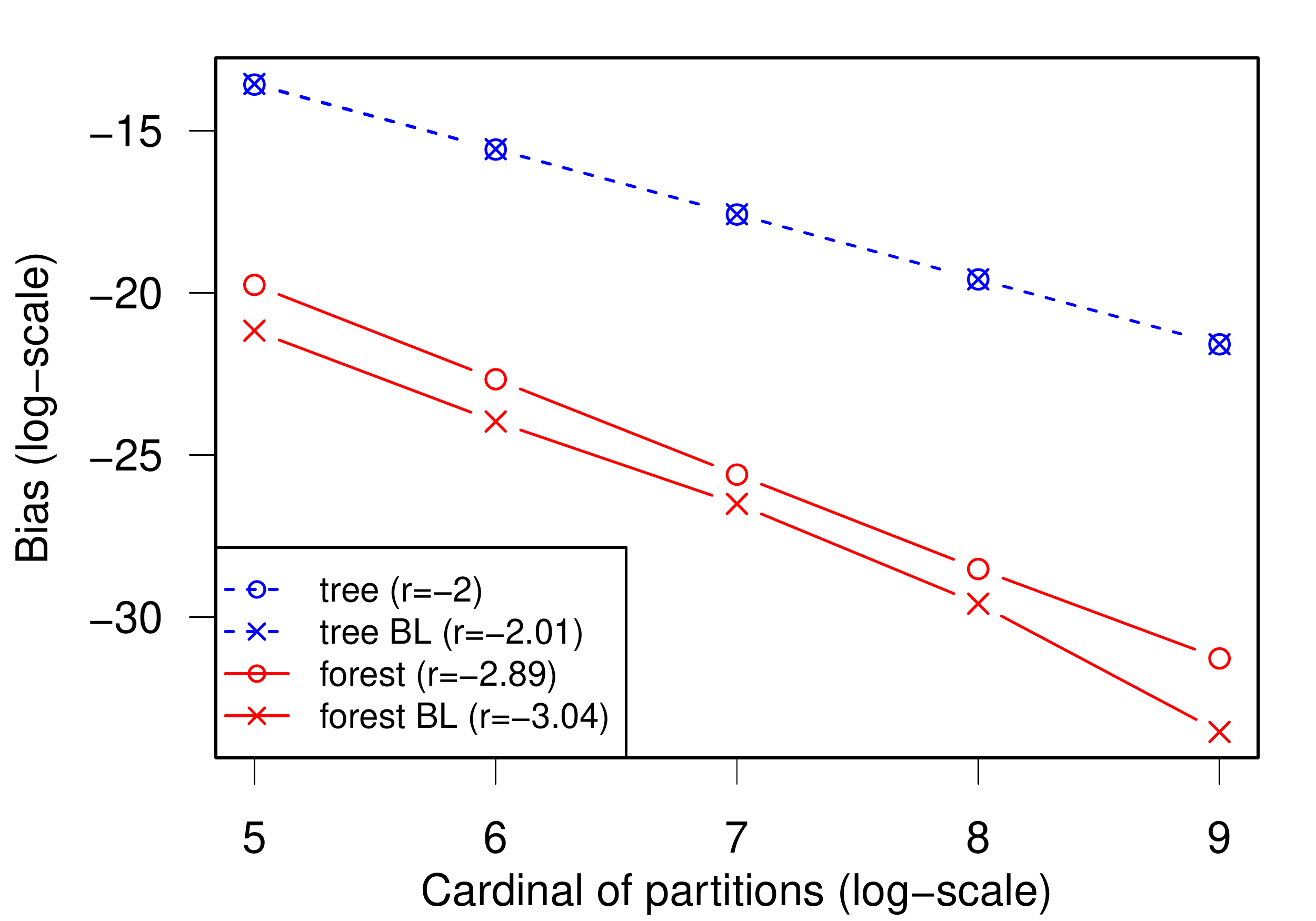}
\centerline{(a) toy, $\dimX=1$}
\end{minipage}
    \hspace{0.01\textwidth}
  \begin{minipage}{0.45\textwidth}
    \includegraphics[width=\textwidth]{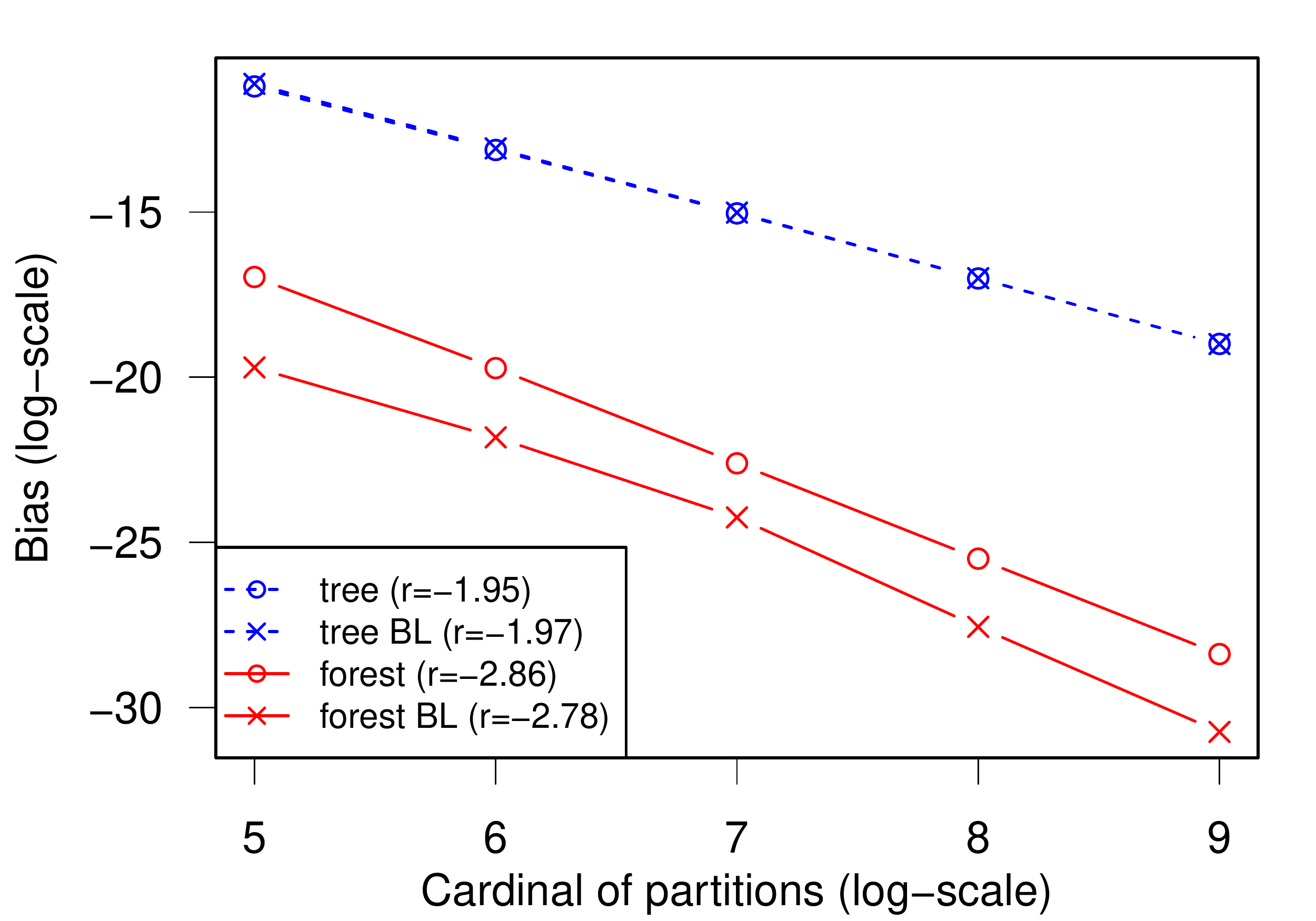}
\centerline{(b) PURF, $\dimX=1$}
\end{minipage}
\\
  \begin{minipage}{0.45\textwidth}
    \includegraphics[width=\textwidth]{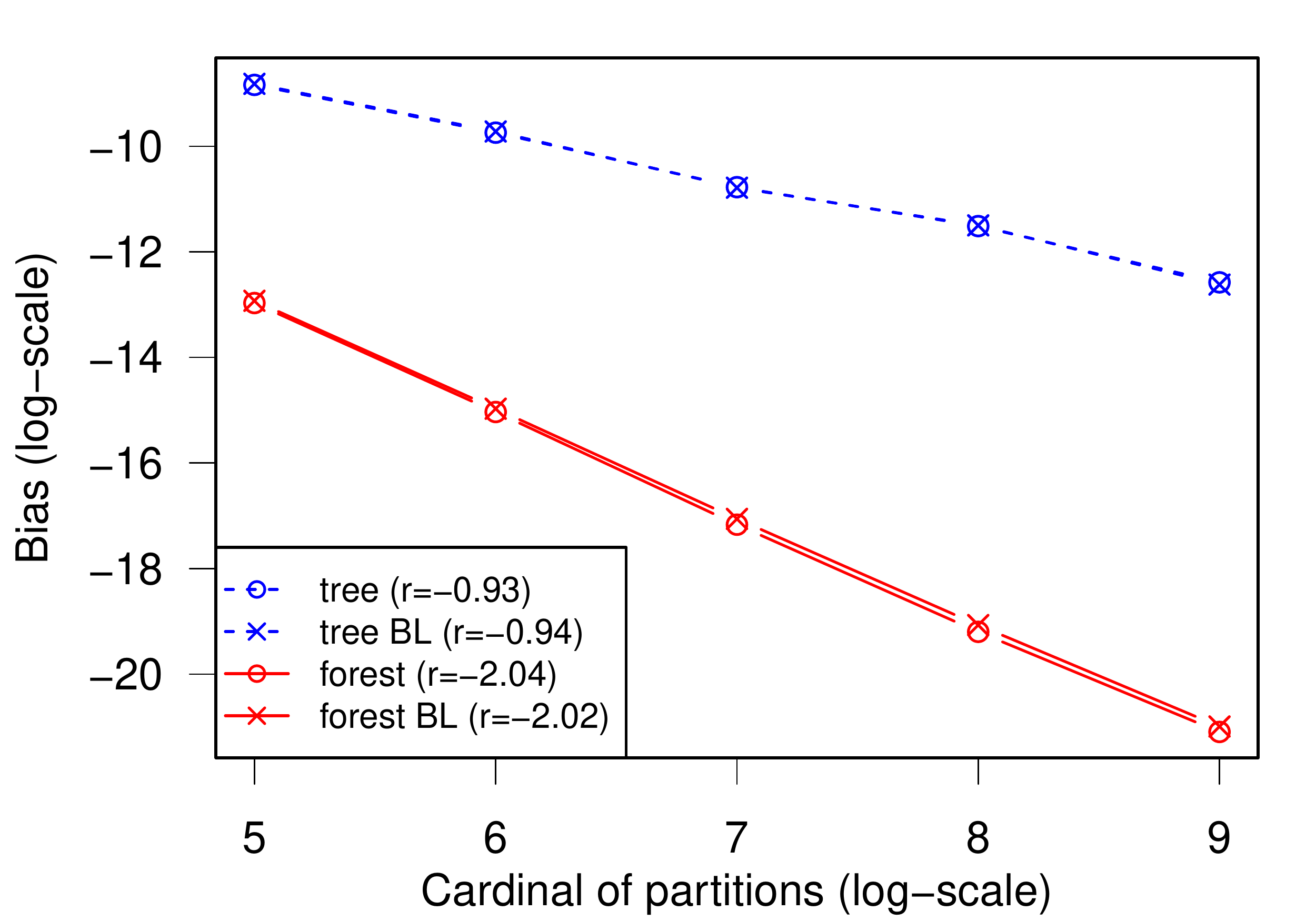}
\centerline{(c) BPRF, $\dimX=1$}
\end{minipage}
    \hspace{0.01\textwidth}
  \begin{minipage}{0.45\textwidth}
    \includegraphics[width=\textwidth]{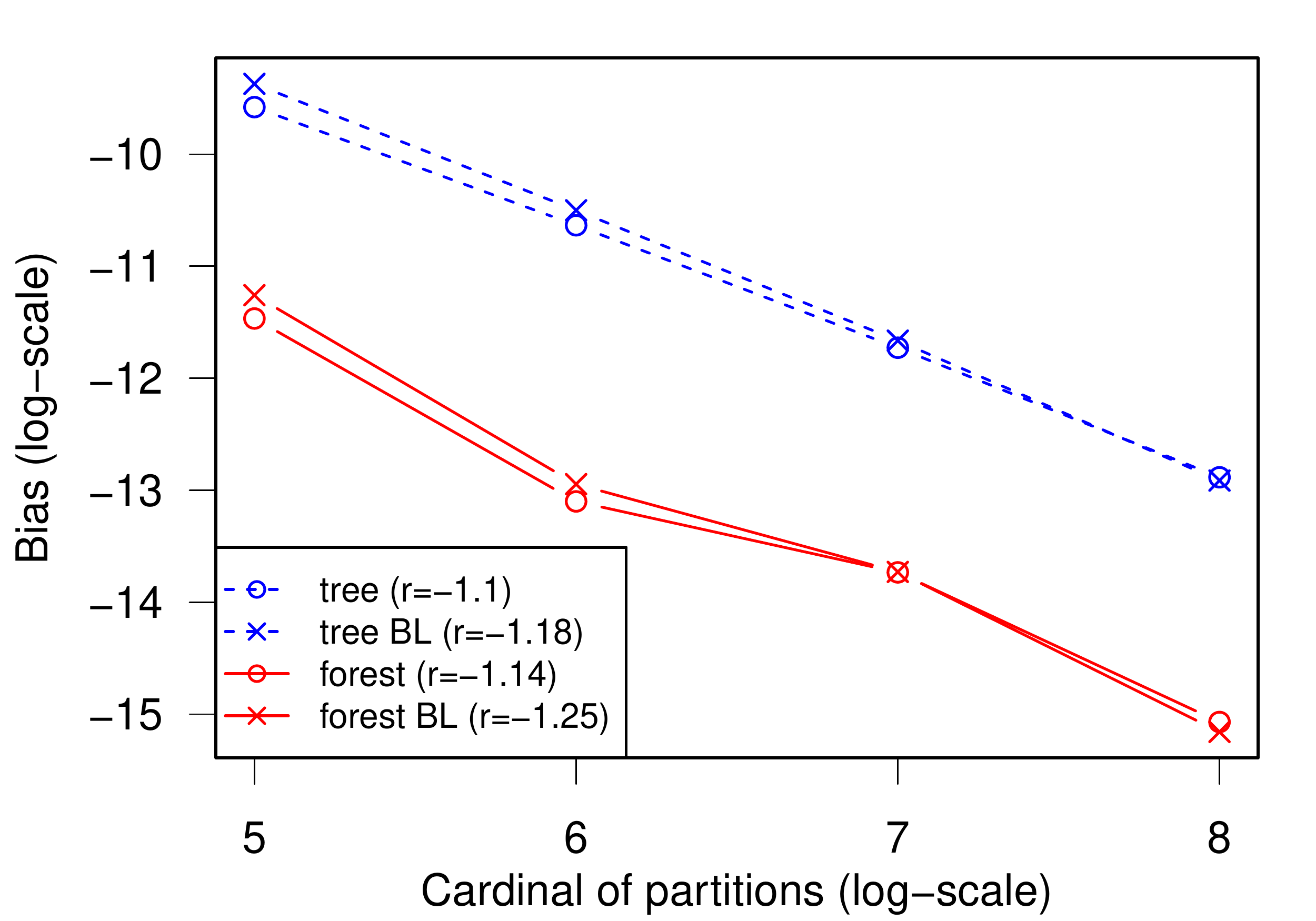}
\centerline{(d) Hold-out RF, $\dimX=1$}
\end{minipage}
\caption{\label{fig.abs} Plot of $\Bias_{\cU,1}$ and
  $\Biasinfty{\cU}$ (in $\log_2$-scale) against $\nfeu$ (in
  $\log_2$-scale) for (a) toy, (b) PURF, (c) BPRF and (d) Hold-out RF
  models, for the \textbf{absolute value} regression function. BL
  corresponds to borderless situations and $r$ denotes the slope of a
  linear model fitted to the scatter plot.}
  \end{center}
\end{figure}
The \textbf{absolute value} regression function presents a singularity
at point $x=1/2$, and it acts as a border point.  
Hence, compared to
the sinusoidal regression function, the only change is that there is
no differences between borderless and regular approximation errors of forests: 
both reach the rate $\nfeu^{-3}$ for toy and PURF models. 
The Hold-out RF model again reaches relatively poor rates,
and the forest does not improve significantly the bias compared to a
single tree.

\subsection{Multidimensional input space}
\label{sec.simu.multi-dim}
For $\dimX>1$, we investigate the behaviors of BPRF and Hold-out RF models.
First, Figure~\ref{fig.id} shows the results for the \textbf{sum}
regression function when $\dimX \in \{5,10\}$.
\begin{figure}[!ht]
  \begin{center}
  \begin{minipage}{0.45\textwidth}
    \includegraphics[width=\textwidth]{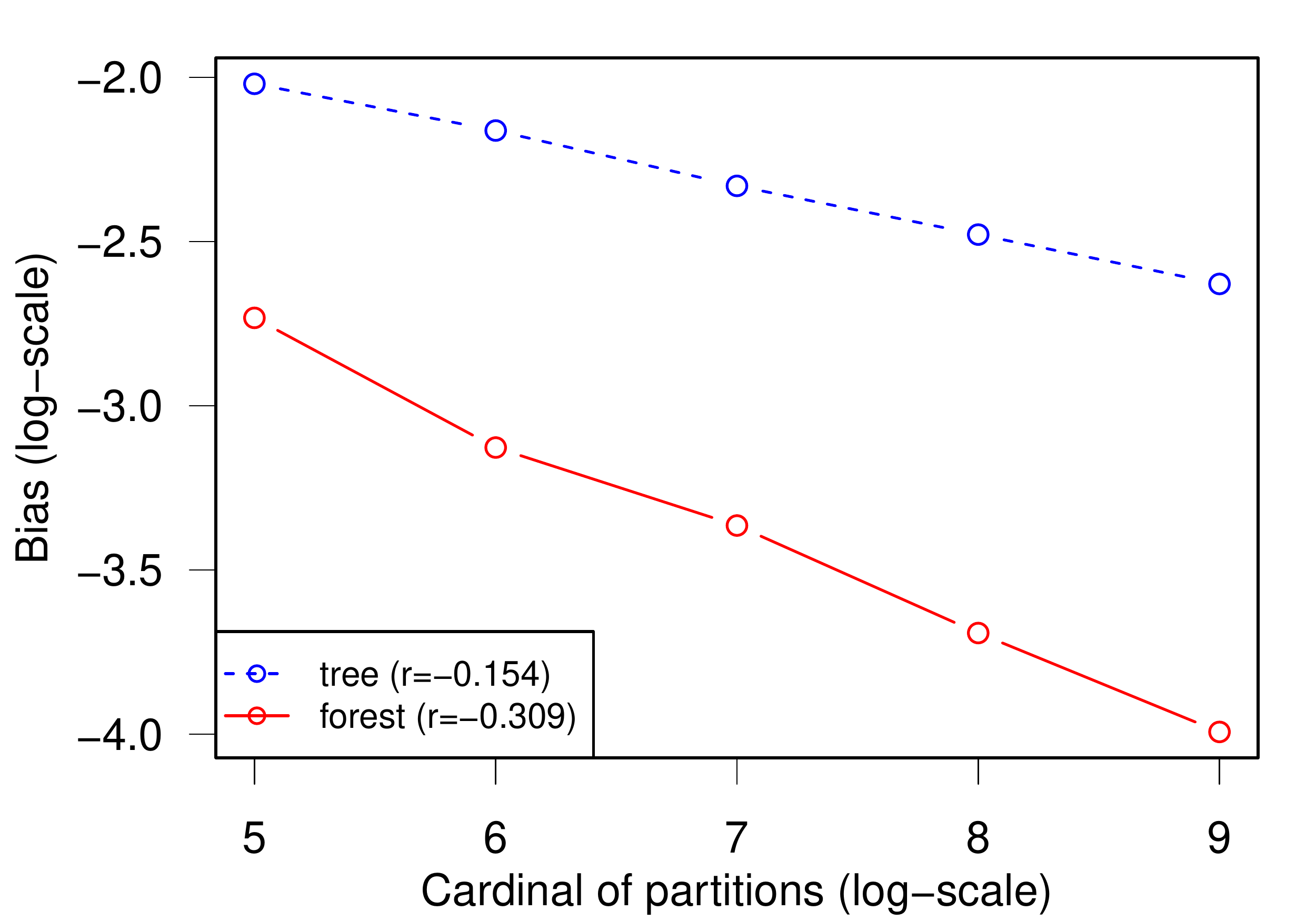}
\centerline{(a) BPRF, $\dimX=5$}
\end{minipage}
    \hspace{0.01\textwidth}
  \begin{minipage}{0.45\textwidth}
    \includegraphics[width=\textwidth]{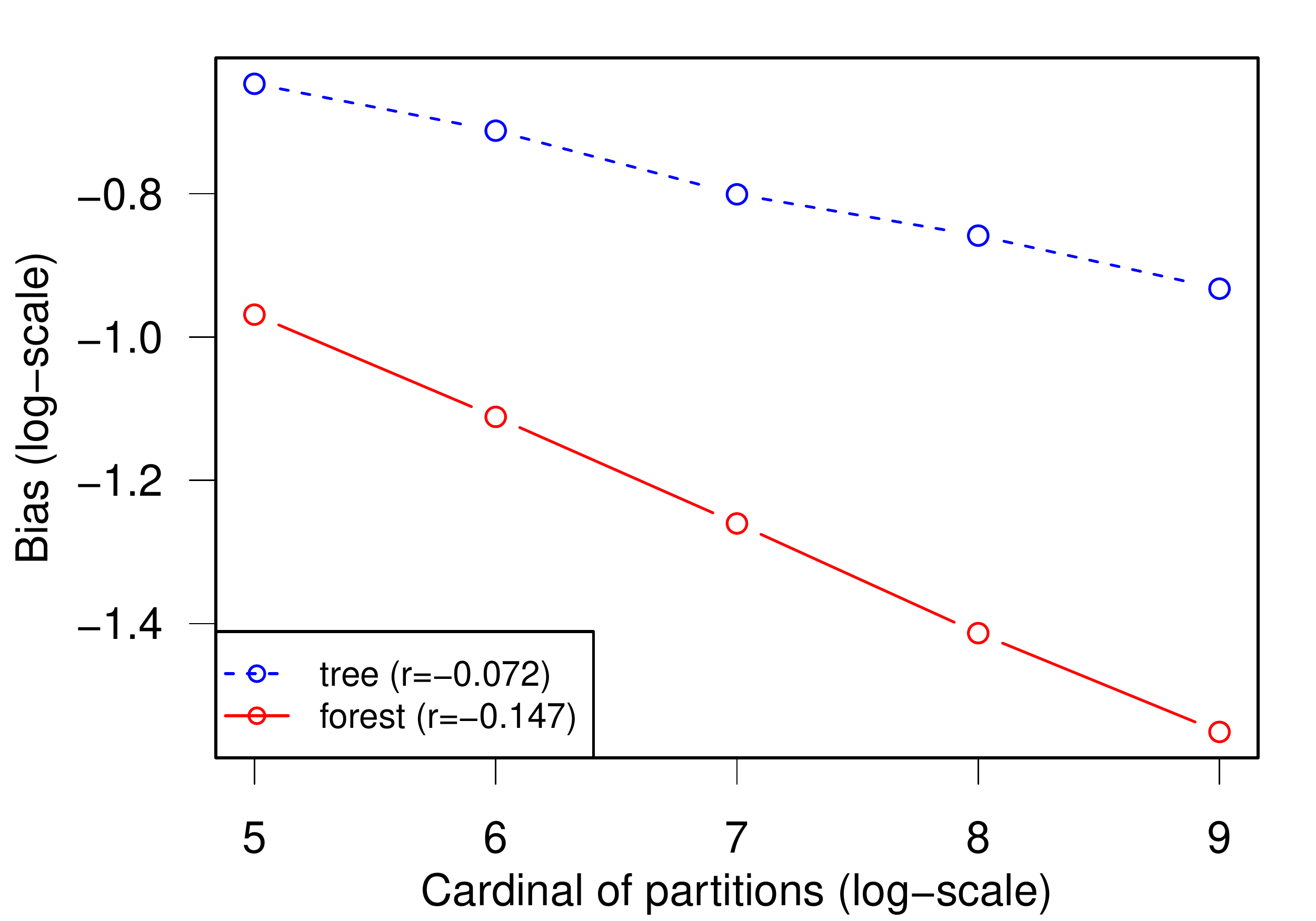}
\centerline{(b) BPRF, $\dimX=10$}
\end{minipage}
\\
  \begin{minipage}{0.45\textwidth}
    \includegraphics[width=\textwidth]{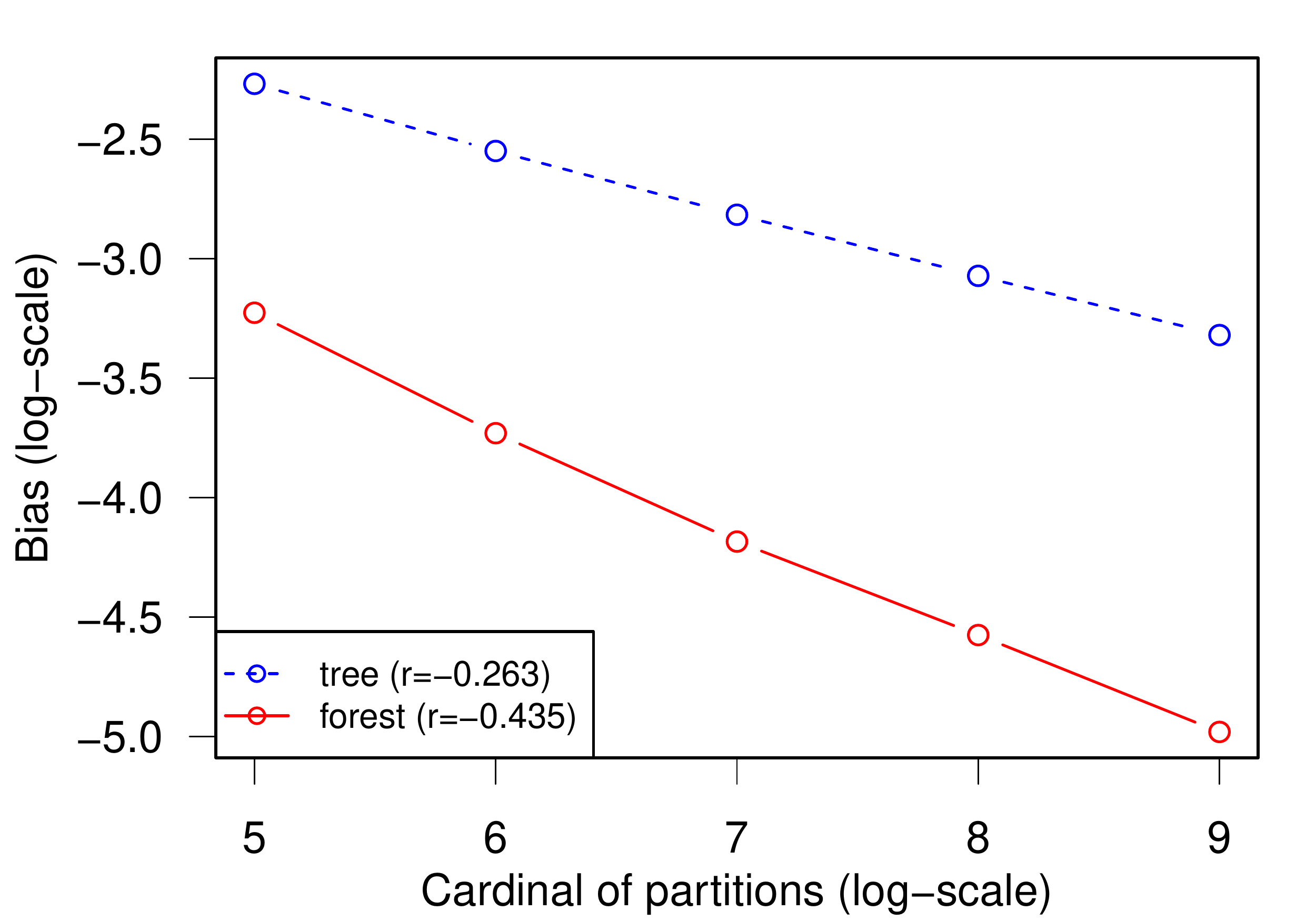}
\centerline{(c) Hold-out RF, $\dimX=5$}
\end{minipage}
    \hspace{0.01\textwidth}
  \begin{minipage}{0.45\textwidth}
    \includegraphics[width=\textwidth]{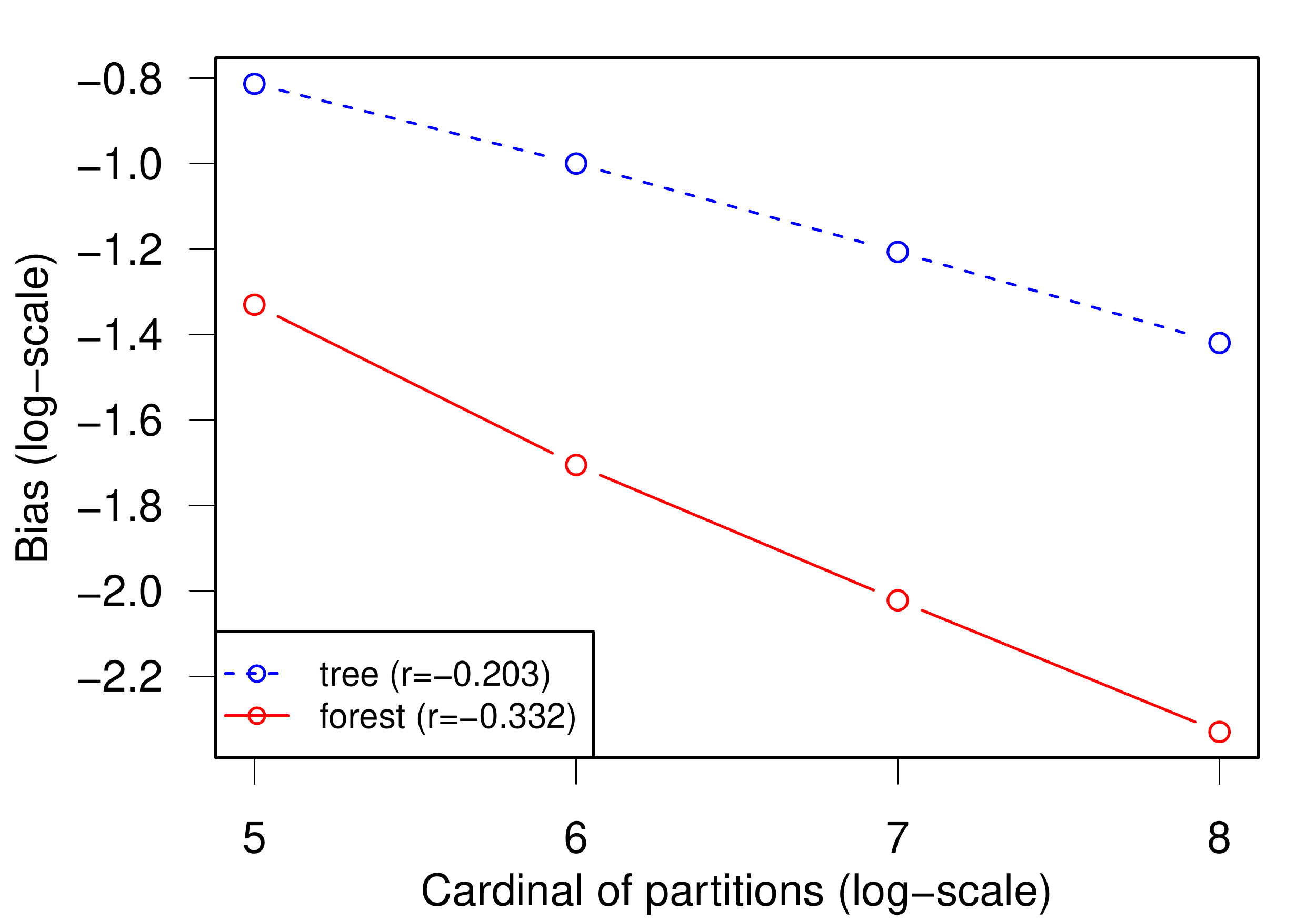}
\centerline{(d) Hold-out RF, $\dimX=10$}
\end{minipage}
\caption{\label{fig.id} Plot of $\Bias_{\cU,1}$ and
  $\Biasinfty{\cU}$ (in $\log_2$-scale) against $\nfeu$ (in
  $\log_2$-scale) for BPRF with 
  (a) $\dimX=5$ and (b) $\dimX=10$, and
  Hold-out RF model with (c) $\dimX=5$ and (d) $\dimX=10$, for the
  \textbf{sum} regression function. $r$ denotes the slope of a linear
  model fitted to the scatter plot.}
  \end{center}
\end{figure}

As in the one-dimensional case, we observe linear behaviors in $\log_2$-$\log_2$ scale. 
For BPRF (top graphs), trees and forests reach approximately the decreasing rates we can expect 
from Section~\ref{sec.multidim.BPRF}, 
respectively $\nfeu^{-\alpha}$ and $\nfeu^{-2\alpha}$ with $\alphaurt
= \log(10/9) / \log(2) \approx 0.152$ when $\dimX=5$ and $\alphaurt =
\log(20/19) / \log(2) \approx 0.074$ when $\dimX = 10$.

Compared to BPRF, the Hold-out RF model reaches better rates in
the multidimensional framework. Moreover, it suffers less from the
increase of the dimension: BPRF rates are divided by $2.1$ when
$\dimX$ increases from $5$ to $10$, whereas Hold-out RF model
rates are only divided by $1.3$.

Forests rates with the Hold-out RF model are about $1.6$ times faster than tree
rates, which illustrates a significant gain brought by forests.
Note however the comparison with BPRF model is partly
unfair, because Hold-out RF can make use of an extra sample
$D_{\nobs}^{\prime}$ for building appropriate partitions of $\X$; nevertheless, with 
BPRF, if such an extra sample is available, it can only be used for reducing the final risk 
by a constant factor (since it doubles the sample size) but not for improving the risk rate. 

\begin{figure}[!ht]
  \begin{center}
  \begin{minipage}{0.45\textwidth}
    \includegraphics[width=\textwidth]{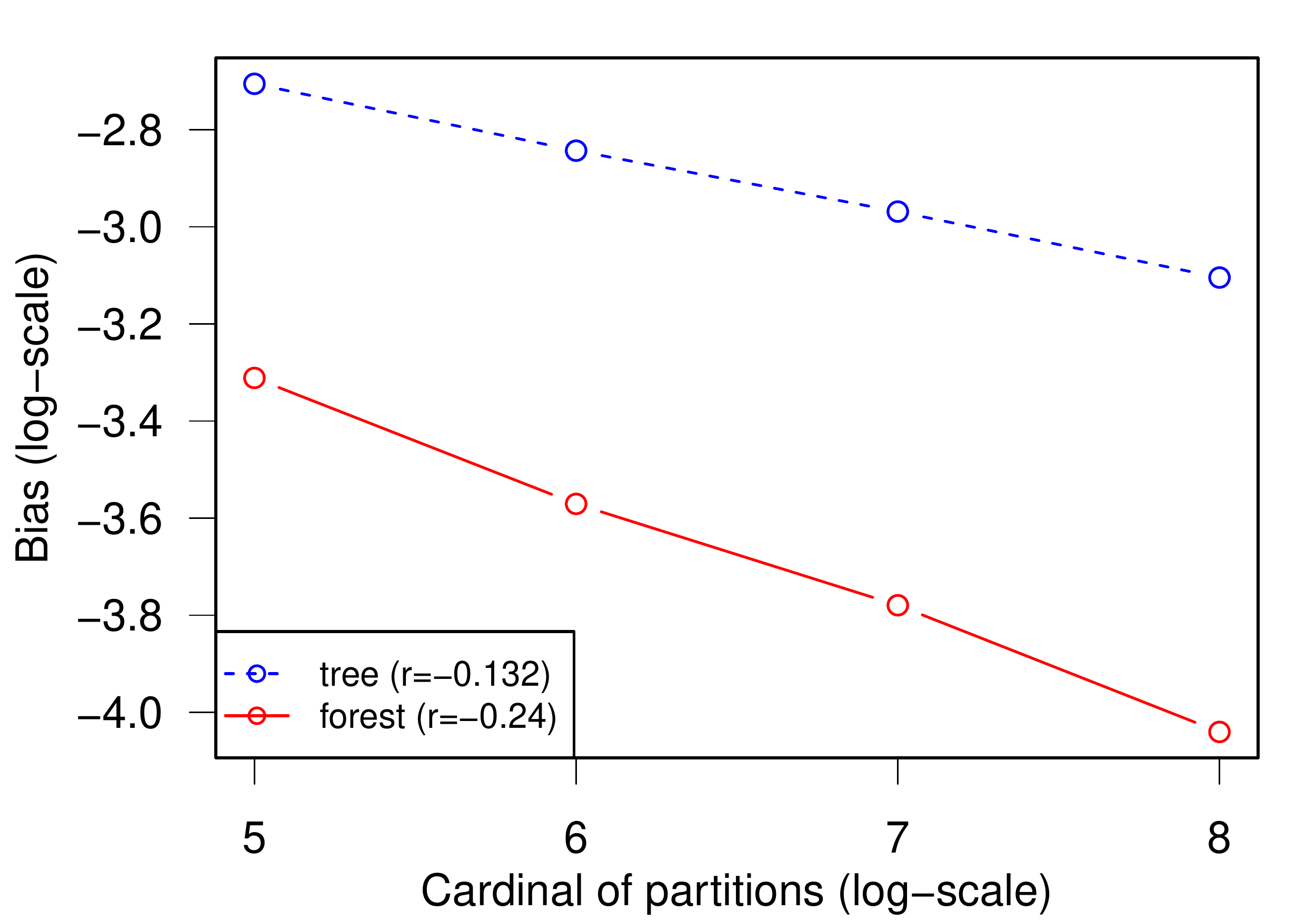}
\centerline{(a) BPRF, $\dimX=5$}
\end{minipage}
    \hspace{0.01\textwidth}
  \begin{minipage}{0.45\textwidth}
    \includegraphics[width=\textwidth]{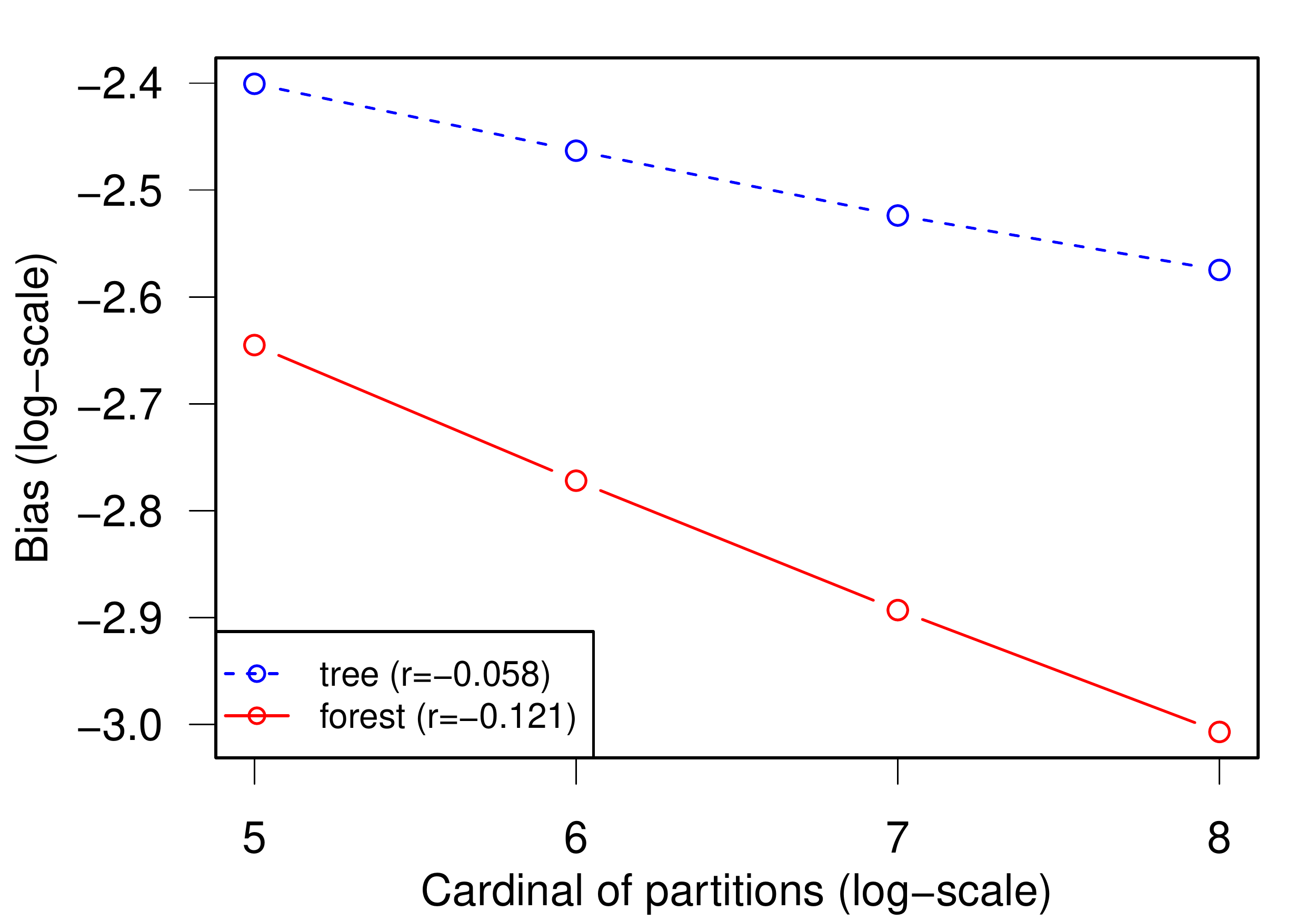}
\centerline{(b) BPRF, $\dimX=10$}
\end{minipage}
\\
  \begin{minipage}{0.45\textwidth}
    \includegraphics[width=\textwidth]{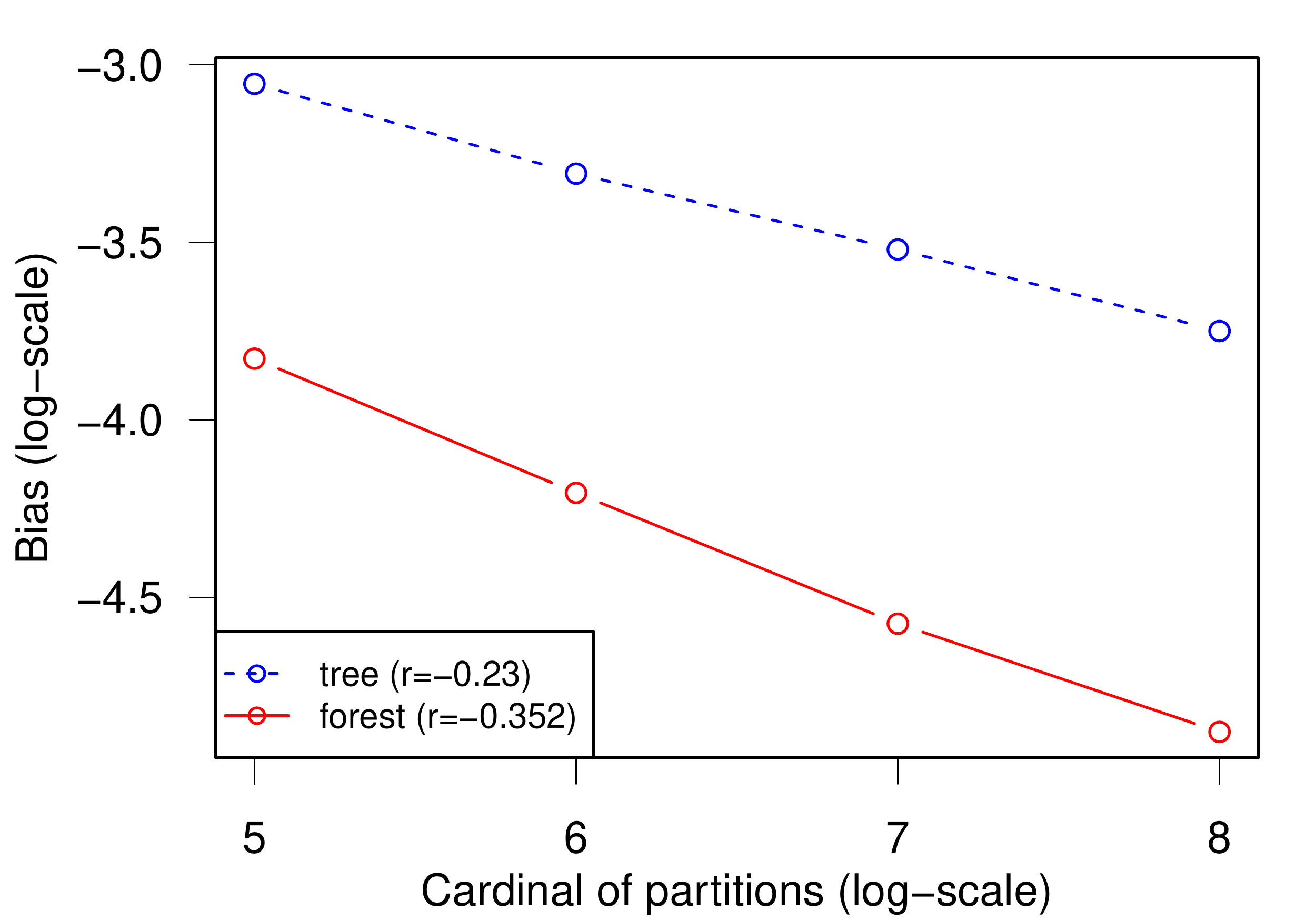}
\centerline{(c) Hold-out RF, $\dimX=5$}
\end{minipage}
    \hspace{0.01\textwidth}
  \begin{minipage}{0.45\textwidth}
    \includegraphics[width=\textwidth]{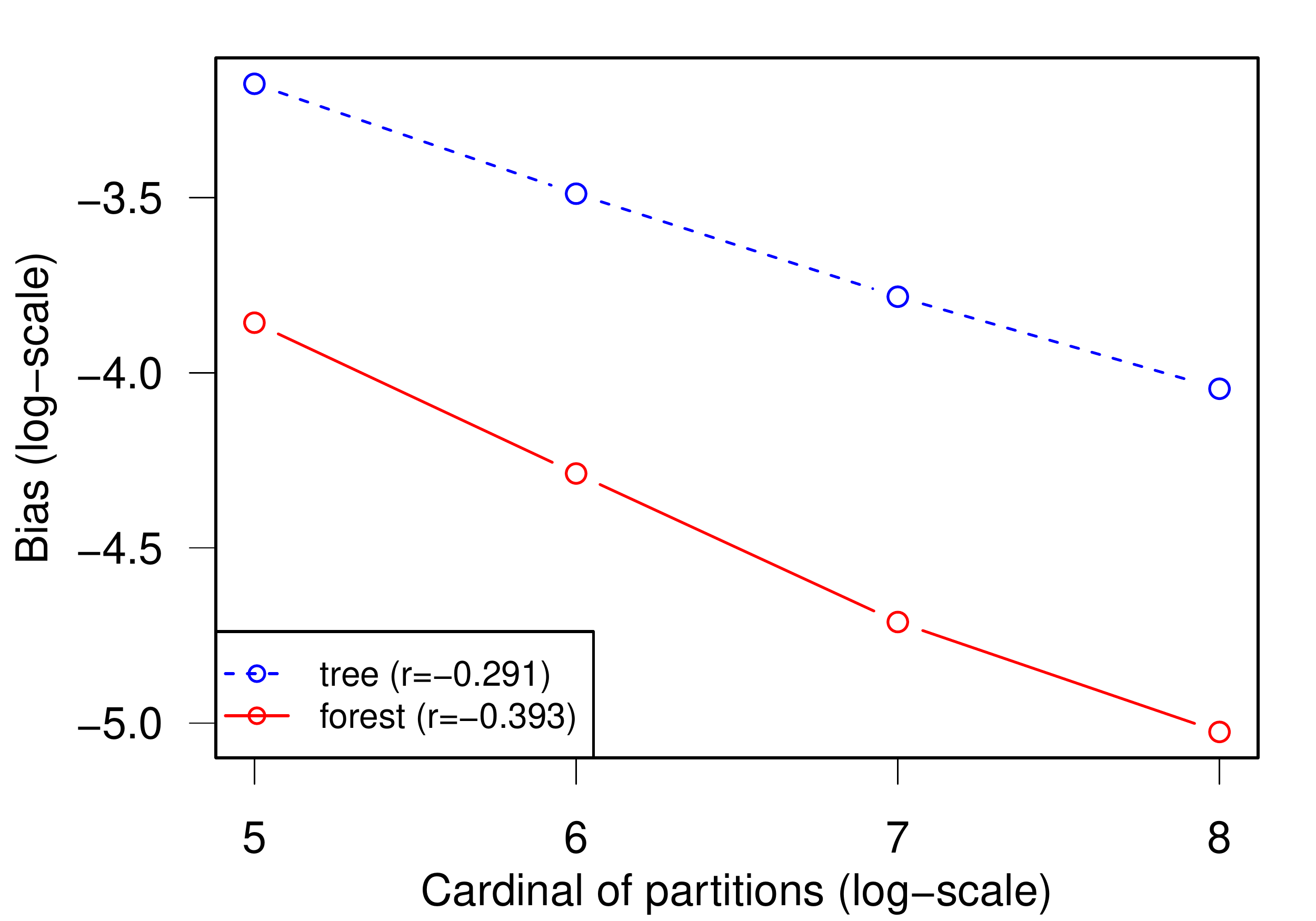}
\centerline{(d) Hold-out RF, $\dimX=10$}
\end{minipage}
\caption{\label{fig.fried1} Plot of $\Bias_{\cU,1}$ and
  $\Biasinfty{\cU}$ (in $\log_2$-scale) against $\nfeu$ (in
  $\log_2$-scale) for BPRF with (a) $\dimX=5$ and (b) $\dimX=10$, and
  Hold-out RF model with (c) $\dimX=5$ and (d) $\dimX=10$, for the
  \textbf{Friedman1} regression function. $r$ denotes the slope of a linear
  model fitted to the scatter plot.}
  \end{center}
\end{figure}
Results for the \textbf{Friedman1} regression function are shown in
Figure~\ref{fig.fried1}.
Note that when $\dimX=10$, the last five variables are non-informative
since the regression function does not depend on them.

For BPRF, rates are slightly worse than for \textbf{sum}, and we still
observe a factor of $2$ between trees and forests rates. 
The decrease of the rates might be explained by two reasons: the complexity of
\textbf{Friedman1} function, and when $\dimX=10$ the presence of five
non-informative variables.

For the Hold-out RF model, rates are still much better than for
BPRF, and forest rates are better than tree rates by a factor $1.5$.
When $\dimX=5$, rates are smaller than for the \textbf{sum}
regression function, probably due to the complexity of the
regression function.  When $\dimX=10$, we surprisingly get faster
rates than when $\dimX=5$.
Obtaining approximation rates as good for $\dimX=10$ as for $\dimX=5$
is expected since the number of informative variables is the same in
both cases, and Hold-out RF are known to adapt to the sparsity
of the regression function (see Section~\ref{sec.disc.comp} for more
details).
But obtaining significantly better rates for a more difficult problem
(when $\dimX=10$) is quite surprising.  Investigating this phenomenon
requires a more systematic simulation study with more examples, which
is out of the scope of the present paper.

\section{Discussion}\label{sec.disc}

In this paper, we analyze several purely random forests models and
show, for each of them, that an infinite forest improves by an order
of magnitude the approximation error of a single tree, assuming the
regression function is smooth.  Since the estimation error of a forest
is smaller or equal to that of a single tree, we deduce that forests
reach a strictly better risk convergence rate.

\subsection{Comparison between the toy, PURF and BPRF models}
\label{sec.disc.comp}
In dimension $\dimX=1$, we can compare the results for the BPRF model in
Section~\ref{sec.multidim.BPRF} with the results obtained in
Sections~\ref{sec_toy_model} and~\ref{sec.purf} for the toy and PURF models. 
Indeed, we get
\[ \alphaurt = 1 \]
so that the approximation error of an infinite forest is of order
$\nfeu^{-2}$ for BPRF, instead of $\nfeu^{-4}$ for toy and PURF (when
avoiding border effects).

Intuitively, it seems the BPRF model tends to leave large ``holes'' in
the space $\X$, whereas the toy model is almost regular, and the PURF
model stays close to regular.  Recall the PURF model can be seen as a
recursive tree construction, where only one leaf is split into two
leaves at each step, and the choice of this leaf is made with
probability equal to its size.  On the contrary, the BPRF model keeps
splitting all leaves whatever their size, which may lead to very small
leaves but also to much larger ones in some significant part of the
space $\X$.

\subsection{Comparison with other random forest models}\label{sec.disc.others}
Another random forest model has been suggested by \citet{Bre:2004}
and was more precisely analyzed by \cite{Bia:2012}.
In the latter paper, the random partitioning---which depends on some
parameters $(p_i)_{1 \leq i \leq \dimX} \in [0,1)^{\dimX}$ with $p_1 +
\cdots + p_{\dimX}=1$---is as follows:
\begin{algo} \label{model.Biau2012} \hfill\\ \vspace*{-3ex}
\begin{itemize}
\item Start from $[0,1)^\dimX$,
\item Repeat $\nfeu$ times: 
for each set $\lambda$ of the current partition, 
  \begin{itemize}
  \item choose a split variable $j$ randomly, with probability distribution over $\{1, \ldots, \dimX\}$ given by $(p_i)_{1 \leq i \leq \dimX}$, 
  \item split $\lambda$ along coordinate $j$ at the midpoint $t_j$ of
    $\lambda_j$, that is put $\{ x \in \lambda \, / \, x_j < t_j\}$
    and $\{ x \in \lambda \, / \, x_j \geq t_j\}$ at the two children
    nodes below $\lambda$.
  \end{itemize}
\end{itemize}
\end{algo}

In a framework where only $S$ variables (among $\dimX$) are
``strong'' (i.e. have an effect on the output variable), the main
result of \citet{Bia:2012} is that if the probability weights are well
tuned (i.e., roughly, if $p_j \approx 1/S$ for a strong variable and
$p_j \approx 0$ for a noisy variable), then for model~\ref{model.Biau2012}, the infinite forest rate
of convergence only depends on $S$ and is faster than the minimax rate
in dimension $\dimX$.
In other words, \cite{Bia:2012} shows that such random forests adapt to
sparsity.

Even if the framework of \cite{Bia:2012} is quite different from ours,
let us give a quick comparison between the different rates obtained when $S=\dimX$.
Assuming the regression function is Lipschitz, for the model studied
by \cite{Bia:2012}, the infinite forest bias is at most
of order $\nfeu^{-\theta}$ with $\theta = 3/(4 \log(2) \dimX)$.  This
is comparable to our result $k^{-2\alphaurt}$ for BPRF, 
because $2 \alphaurt \approx 1/ ( \log(2) \dimX)$ when $\dimX$ is large enough.  
Our rate for BPRF is a little bit faster,
but recall that we make a stronger assumption on the smoothness of the
regression function ($C^2$ instead of Lipschitz), and we only consider
$\dimX=S$.  The problem of knowing if the combination of these two
analyses could give better rates (in a sparse framework with $C^2$
regression function) is beyond the scope of the paper and we let this
point for future research. 

\bigskip

Finally, as mentioned in Section~\ref{sec.BPRF.discussion}, we
conjecture the BPRF model reaches better rates than the UBPRF model~\ref{algo.ubprf}.
Intuitively, as we discuss in Section~\ref{sec.disc.comp}, UBPRF
model tends to leave even larger ``holes'' in $\X$ than BPRF, because
instead of constructing balanced trees, it randomly chooses at each
step the next leaf to be split, with a uniform distribution on all
leaves.

\subsection{General conclusions}
\label{sec.disc.general}

For all PRF models studied in this paper, we get
that the infinite forest bias order of magnitude is equal to the
square of the single tree bias. 
Consequently, if a single tree reaches the 
minimax convergence rate for $C^1$ functions, we directly have that a
large enough forest is minimax for $C^2$ functions. 
So, compared to
trees, forests can well approximate regression functions with one more
level of smoothness.  
Further research is needed to know whether this
phenomenon is general for all PRF models.

Interestingly, our analysis helps to suggest better PRF partitioning
mechanism.  It seems PRF models benefit from a choice of the next set
of the current partition to be split with a probability proportional to the size
of the set. This statement is justified by the comparison between
BPRF and PURF models in dimension 1, and it leads us to conjecture
that PRF models reaching $C^2$-minimax rates of convergence could also
be derived in dimension $\dimX$. For instance, this should be the case
with the generalization of the PURF model to any $\dimX \geq 1$,
consisting in replacing ``length'' by ``volume'' in
Model~\ref{algo.purf}, and by choosing
uniformly a coordinate $j$ before performing the split along it.

For practical use of PRF models, we suggest an order of magnitude for
the number of trees in a forest that is sufficient to get a bias term as small as
for an infinite forest.
More precisely, if the bias of a single
tree is of order $k^{-\gamma}$  for some $\gamma >0$, 
our results suggest it is sufficient to build $\narb =
k^{\gamma}$ trees to get a forest which reaches same rates as the
(theoretical) infinite forest.

Finally, we mention that all our general results in
Sections~\ref{sec.general}--\ref{sec.multidim} can be applied for any
random forest satisfying assumption~\eqref{hyp.purely-random}, that
is, when random partitions are obtained independently from the
learning sample $D_{\nobs}$. 
Hence, if for a random forest model, we
are able to compute quantities appearing in
Proposition~\ref{pro.bias.multidim.H3}, we can deduce results on the bias
and risk convergence rates for these random
forests. 
In particular, we have in mind the Hold-out RF model defined in Section~\ref{sec.simu}.
Addressing this point appears to be an interesting future research
topic, not only from the theoretical point of view, but also in
practice because the Hold-out RF model can achieve very good performances.

\section*{Acknowledgments}
The authors are grateful to Guillaume Obozinski and Francis Bach for
several discussions. The authors acknowledge the partial support of
French Agence Nationale de la Recherche, under grants {\sc Detect}
(ANR-09-JCJC-0027-01) and {\sc Calibration} (Blanc SIMI 1 2011 projet
Calibration).

\bibliographystyle{apalike}
\bibliography{purfbias}


\appendix

\section{Proofs: general results on the variance and bias terms}
\label{sec.pr.gen_res}

\subsection{General bounds on the variance term}
\label{sec.pr.gen_res.variance}

\begin{proof}[Proof of Proposition~\ref{pro.maj-variance}]
  %
  %
  First, by convexity of $u \mapsto u^2$, 
  \begin{align*}
    \E\croch{ \paren{ \tilde{s}_{\bV_{\narb}}(X) - \ERM (X ; \bV_{\narb})}^2 } 
    &= \E\croch{ \paren{ \frac{1}{\narb} \sum_{j=1}^{\narb} \paren{ \tilde{s}_{\bU^j}(X) - \ERM (X ; \bU^j)} }^2 } 
    \\
    &\leq \frac{1}{\narb}  \sum_{j=1}^{\narb} \E\croch{  \paren{ \tilde{s}_{\bU^j}(X) - \ERM (X ; \bU^j) }^2 } 
    \\
    &= \E\croch{  \paren{ \tilde{s}_{\bU^1}(X) - \ERM (X ; \bU^1) }^2 } 
  \end{align*}
  which proves Eq.~\eqref{eq.pro.maj-variance.foret-vs-arbre}.

  %
  Now, remark that conditionally to $\bU^1$, $\ERM (\cdot ; \bU^1)$ is a classical regressogram estimator, and $\E\scroch{  \sparen{ \tilde{s}_{\bU^1}(X) - \ERM (X ; \bU^1) }^2 \sachant \bU^1} $ is its estimation error. 
  Therefore, using Proposition~1 in \cite{Arl:2008a}, we get 
  \begin{align} 
\notag 
      \E\croch{  \paren{ \tilde{s}_{\bU^1}(X) - \ERM (X ; \bU^1) }^2 \sachant \bU^1} &= \frac{1}{\nobs} \sum_{\lambda \in \bU^1} \paren{ 1 + \delta_{ \nobs , \pl }} \paren{ \sigma^2 + \E\croch{ \paren{s(X) - \widetilde{s}_{\bU^1}(X) }^2 \sachant X \in \lambda \, , \, \bU^1} } \\
    \label{eq.pr.pro.maj-variance.Ep1} 
      & \quad + \frac{1}{\nobs} \sum_{\lambda \in \bU^1} \betl^2 (1 - \pl)^{\nobs} 
      \\
\notag 
      \mbox{where} \qquad 
      - 2 \exp\paren{ - \nobs \pl } 
      &\leq - 2 (1 - \pl)^{\nobs} 
      \leq \delta_{ \nobs , \pl } 
      \leq 1  
      \enspace , 
      \pl \egaldef \P \paren{ X \in \lambda}
  \end{align}
and the last term---which does not appear in Proposition~1 of \cite{Arl:2008a}---comes from our convention to take  $\ERM(\cdot;\bU^1)=0$ on each $\lambda\in \bU^1$ such that no data point in $D_n$ belongs to $\lambda$. 
Since $\absj{s(X) - \widetilde{s}_{\bU^1}(X)} \leq 2 \norm{s}_{\infty}$ almost surely 
and $\absj{\betl}\leq \norm{s}_{\infty}$ for all $\lambda \in \bU^1$, 
Eq.~\eqref{eq.pr.pro.maj-variance.Ep1} implies Eq.~\eqref{eq.pro.maj-variance.arbre-gal} 
and Eq.~\eqref{eq.pro.maj-variance.arbre-gal-min} by integrating over $\bU^1$ 
(which can be done separately thanks to assumption \eqref{hyp.purely-random}). 

%
When $s$ is $K$-Lipschitz w.r.t. $\delta$, for any $\lambda \in \bU^1$
and $x \in \lambda$,
\[ \absj{ s(x) - \widetilde{s}_{\bU^1}(x) } \leq K \diam_{\delta}
  (\lambda) 
\] 
so that Eq.~\eqref{eq.pr.pro.maj-variance.Ep1} implies
\[ 
\E\croch{ \paren{ \tilde{s}_{\bU^1}(X) - \ERM (X ; \bU^1) }^2
    \sachant \bU^1} \leq \frac{1}{\nobs} \paren{2 \sigma^2 \nfeu 
    + 2 K^2
    \sum_{\lambda \in \bU^1} \paren{ \diam_{\delta} (\lambda) }^2 } 
    +  \frac{\norm{s}_{\infty}^2}{\nobs} \sum_{\lambda \in \bU^1}  \exp\paren{- \nobs \pl} 
\enspace ,
\]
thus Eq.~\eqref{eq.pro.maj-variance.arbre-Lipschitz} by integrating
over $\bU^1$ (which can be done separately thanks to assumption
\eqref{hyp.purely-random}).
\end{proof}

\subsection{Approximation of the bias term}
\label{sec.pr.gen_res.bias}

\begin{proof}[Proof of Proposition~\ref{pro.bias.multidim.H3}]
  First, we assume \eqref{hyp.s-2-fois-derivable.alt} holds true. Let 
  \begin{align*}
    T_2(x , \bU) &\egaldef \frac{1}{\absj{I_{\bU}(x)}} \int_{I_{\bU}(x)} \croch{ s(t) - s(x) - \nabla s(x) \cdot (t-x) } dt 
    \\
    &= \tilde{s}_{\bU}(x) - s(x) - \nabla s(x) \cdot \paren{ C_{\bU}(x) - x} 
    \\
    \mbox{where} \quad C_{\bU}(x) &\egaldef \frac{1}{\absj{I_{\bU}(x)}} \int_{I_{\bU}(x)} t dt = \paren{ \frac{A_{i,\bU}(x) + B_{i,\bU}(x)}{2} }_{1 \leq i \leq \dimX}  \qquad \mbox{is the center of } I_{\bU}(x) \enspace  . 
  \end{align*}
  Then, integrating \eqref{hyp.s-2-fois-derivable.alt} over $t \in I_{\bU}(x)$ yields 
  \begin{align}
    \notag \absj{T_2(x, \bU)} &\leq \CHdeuxa \frac{1}{\absj{I_{\bU}(x)}} \int_{I_{\bU}(x)} \norm{t-x}^2_2 dt \\
    &= \frac{\CHdeuxa}{3} \sum_{i=1}^{\dimX} \croch{ \paren{ x_i - A_{i,\bU}(x)}^2 + \paren{ B_{i,\bU}(x) - x_i }^2 - \paren{ x_i - A_{i,\bU}(x)} \paren{ B_{i,\bU}(x) - x_i } } \enspace .
    \label{eq.pr.pro.bias.multidim.maj-T2}
  \end{align}
  By definition, 
  \begin{align*}
    \Biasinfty{\cU}(x) &= \paren{ s(x) - \E_{\bU} \croch{ \tilde{s}_{\bU}(x) } }^2 \\
    &= \paren{ \nabla s(x) \cdot \E_{\bU} \croch{ C_{\bU}(x) - x} + \E_{\bU} \croch{ T_2(x, \bU) } }^2 \\
    &= \paren{ \nabla s(x) \cdot \E_{\bU} \croch{ C_{\bU}(x) - x} }^2  + 2 \paren{ \nabla s(x) \cdot \E_{\bU} \croch{ C_{\bU}(x) - x} } \E_{\bU} \croch{ T_2(x, \bU) } + \paren{ \E_{\bU} \croch{ T_2(x, \bU) } }^2 \\
    &= \mathcal{M}_{1,\cU,x}^2 + 2 \mathcal{M}_{1,\cU,x} \E_{\bU} \croch{ T_2(x, \bU) } + \paren{ \E_{\bU} \croch{ T_2(x, \bU) } }^2 \enspace .
  \end{align*}
  We conclude the proof of Eq.~\eqref{eq.pro.bias.multidim.Bias-H2} by remarking that \eqref{eq.pr.pro.bias.multidim.maj-T2} implies 
  \begin{align*}
    \absj{ \E_{\bU} \croch{ T_2(x, \bU) } } &\leq \E_{\bU} \absj{ T_2(x, \bU) } 
    \leq \frac{\CHdeuxa}{3} \sum_{i=1}^{\dimX} \croch{ m_{AA,i,x,\cU} + m_{BB,i,x,\cU} - m_{AB,i,x,\cU} }  = \mathcal{R}_{2,\cU,x} 
    \enspace .
  \end{align*}

  Let us now prove Eq.~\eqref{eq.pro.bias.multidim.Var-H2}. By definition of $\Vararbre{\cU}(x)$ and $T_2(x,\bU)\,$, 
  \begin{align}
    \notag 
    \Vararbre{\cU}(x) &= \var_{\bU \sim \cU} \paren{ \tilde{s}_{\bU}(x) - s(x) } \\
    \notag 
    &= \var_{\bU \sim \cU} \paren{ \nabla s(x) \cdot \paren{ C_{\bU}(x) - x} + T_2(x,\bU) } \\
    \notag 
    &= \var_{\bU \sim \cU} \paren{ \nabla s(x) \cdot \paren{ C_{\bU}(x) - x} } 
    + \var_{\bU \sim \cU} \paren{ T_2(x,\bU) }
    \\
    \notag 
    &\qquad + 2 \cov_{\bU \sim \cU} \paren{  \nabla s(x) \cdot \paren{ C_{\bU}(x) - x} + T_2(x,\bU) \, , \, T_2(x,\bU) }
  \end{align}
  so that
  \begin{equation}
    \label{eq.pr.pro.bias.multidim.VU}
    \begin{split}
      & \absj{ \Vararbre{\cU}(x) - \var_{\bU \sim \cU} \paren{ \nabla s(x) \cdot \paren{ C_{\bU}(x) - x} } }
      \\
      & \qquad \leq \E_{\bU \sim \cU} \croch{ T_2^2(x,\bU) } 
      + 2 \sqrt{ \var_{\bU \sim \cU} \paren{ \nabla s(x) \cdot \paren{ C_{\bU}(x) - x} } \E_{\bU \sim \cU} \croch{ T_2^2(x,\bU) }  }
    \end{split}
  \end{equation}
  To prove Eq.~\eqref{eq.pro.bias.multidim.Var-H2}, we compute the two terms appearing in Eq.~\eqref{eq.pr.pro.bias.multidim.VU}. 
  First, using computations made when proving Eq.~\eqref{eq.pro.bias.multidim.Bias-H2}, 
  \begin{align*}
    & \quad \var_{\bU \sim \cU} \paren{ \nabla s(x) \cdot \paren{ C_{\bU}(x) - x} } \\
    &= \E_{\bU \sim \cU}\croch{ \paren{ \nabla s(x) \cdot \paren{ C_{\bU}(x) - x}  }^2 } - \mathcal{M}_{1,\cU,x}^2 \\
    &= \E\left[ \frac{1}{4} \sum_{i=1}^{\dimX} \croch{ \paren{ \frac{\partial s}{\partial x_i} (x)}^2 (B_{i,\bU}(x) - x_i -(x_i - A_{i,\bU}(x)))^2 } \right. \\
    &\qquad \left. + \frac{1}{4} \sum_{i\neq j} \croch{ \frac{\partial s}{\partial x_i} (x) \frac{\partial s}{\partial x_j} (x) (B_{i,\bU}(x) - x_i -(x_i - A_{i,\bU}(x))) (B_{j,\bU}(x)  - x_j -(x_j - A_{j,\bU}(x))) } \right] 
    - \mathcal{M}_{1,\cU,x}^2 \\
    &= \frac{1}{4} \sum_{i=1}^{\dimX} \croch{ \paren{ \frac{\partial s}{\partial x_i} (x)}^2 \paren{ m_{AA,i,x,\cU} + m_{BB,i,x,\cU} - 2 \, m_{AB,i,x,\cU} } } 
    \\ & \qquad 
    + \frac{1}{4} \sum_{1\leq i\neq j \leq \dimX} \croch{ \frac{\partial s}{\partial x_i} (x) \frac{\partial s}{\partial x_j} (x) m_{B-A,i,j,x,\cU} } - \mathcal{M}_{1,\cU,x}^2 \\
    &= \mathcal{N}_{2,\cU,x} - \mathcal{M}_{1,\cU,x}^2  \enspace .
  \end{align*}
  Second, using Eq.~\eqref{eq.pr.pro.bias.multidim.maj-T2} and Cauchy-Schwarz inequality, 
  \begin{align*}
    \E_{\bU \sim \cU} \croch{ T_2^2(x,\bU) } 
    &\leq \frac{\CHdeuxa^2}{9} \E_{\bU \sim \cU} \croch{ \paren{ \sum_{i=1}^{\dimX} \croch{ \paren{ x_i - A_{i,\bU}(x)}^2 + \paren{ B_{i,\bU}(x) - x_i }^2 } }^2 } 
    \\
    &\leq  \frac{2 \dimX \CHdeuxa^2}{9} \E_{\bU \sim \cU} \croch{ \sum_{i=1}^{\dimX} \croch{ \paren{ x_i - A_{i,\bU}(x)}^4 + \paren{ B_{i,\bU}(x) - x_i }^4 } }
    \\
    &=    \frac{2 \dimX \CHdeuxa^2}{9} \sum_{i=1}^{\dimX} \croch{  m_{AAAA,i,x,\cU} + m_{BBBB,i,x,\cU} } = \mathcal{R}_{4,\cU,x}
  \end{align*}
  which concludes the proof of Eq.~\eqref{eq.pro.bias.multidim.Var-H2}. 

  \medskip

  Let us now assume \eqref{hyp.s-3-fois-derivable.alt} holds true. 
  Let 
  \begin{align} \notag 
    T_3(x , \bU) &\egaldef \frac{1}{\absj{I_{\bU}(x)}} \int_{I_{\bU}(x)} \croch{ s(t) - s(x) -  \nabla s(x) \cdot (t-x) - \frac{1}{2} (t-x)^T \paren{ \sparen{ \nabla^{(2)} s} (x) } (t-x) } \, dt 
    \\
    &= \tilde{s}_{\bU}(x) - s(x) - \nabla s(x) \cdot \paren{ C_{\bU}(x) - x} - \frac{1}{2 \absj{I_{\bU}(x)}} \int_{I_{\bU}(x)} (t-x)^T \paren{ \sparen{ \nabla^{(2)} s} (x) } (t-x) \, dt
    \enspace  . 
    \label{eq.pr.pro.bias.multidim.eq-T3}
  \end{align}
  The last term of Eq.~\eqref{eq.pr.pro.bias.multidim.eq-T3} is equal to 
  \begin{align}
    \notag 
    & \qquad \frac{1}{2 \absj{I_{\bU}(x)}} \int_{I_{\bU}(x)} (t-x)^T \nabla^{(2)} s (x) (t-x) \, dt \\
    \notag 
    &= \frac{1}{2 \absj{I_{\bU}(x)}} \int_{I_{\bU}(x)} \sum_{1\leq i,j \leq \dimX} \frac{\partial^2 s}{\partial x_i \partial x_j}(x) (t_i - x_i)(t_j - x_j) \, dt \\ 
    \notag 
    & = \frac{1}{2} \sum_{i=1}^{\dimX} \croch{ \frac{\partial^2 s}{\partial x_i^2}(x) \frac{1}{B_{i,\bU}(x) - A_{i,\bU}(x)} \int_{A_{i,\bU}(x)}^{B_{i,\bU}(x)} (t_i-x_i)^2 \, dt_i } \\ 
    \notag 
    & + \frac{1}{2} \sum_{1\leq i \neq j \leq \dimX} \croch{ \frac{\partial^2 s}{\partial x_i \partial x_j}(x) \frac{1}{B_{i,\bU}(x) - A_{i,\bU}(x)} \int_{A_{i,\bU}(x)}^{B_{i,\bU}(x)} (t_i-x_i) \, dt_i \frac{1}{B_{j,\bU}(x) - A_{j,\bU}(x)} \int_{A_{j,\bU}(x)}^{B_{j,\bU}(x)} (t_j-x_j) \, dt_j } \\ 
    \label{eq.pr.pro.bias.multidim.last-T3}
    & = \frac{1}{6} \sum_{i=1}^{\dimX} \croch{ \frac{\partial^2 s}{\partial x_i^2}(x) \paren{ (B_{i,\bU}(x) - x_i)^2 + (x_i -A_{i,\bU}(x))^2 - (B_{i,\bU}(x) - x_i)(x_i -A_{i,\bU}(x)) } } \\ 
    \notag 
    & \qquad + \frac{1}{8} \sum_{1\leq i \neq j \leq \dimX} \croch{ \frac{\partial^2 s}{\partial x_i \partial x_j}(x) (B_{i,\bU}(x) - x_i -(x_i - A_{i,\bU}(x))) (B_{j,\bU}(x) - x_j -(x_j - A_{j,\bU}(x))) }
    \enspace . 
  \end{align}
  So, combining  Eq.~\eqref{eq.pr.pro.bias.multidim.eq-T3} with Eq.~\eqref{eq.pr.pro.bias.multidim.last-T3}, 
  \begin{align} 
    \notag 
    \E_{\bU \sim \cU} \croch{ \tilde{s}_{\bU}(x) - s(x) } 
    &= \mathcal{M}_{1,\cU,x}
    +  \frac{1}{6} \sum_{i=1}^{\dimX} \croch{ \frac{\partial^2 s}{\partial x_i^2}(x) \paren{ m_{BB,i,x,\cU} + m_{AA,i,x,\cU} - m_{AB,i,x,\cU} } }  
    \\ 
    \notag 
    & \qquad +  \frac{1}{8} \sum_{1\leq i \neq j \leq \dimX} \croch{ \frac{\partial^2 s}{\partial x_i \partial x_j}(x)  m_{B-A,i,j,x,\cU} }
    + \E_{\bU \sim \cU} \croch{ T_3(x,\bU)} 
    \\
    &= \mathcal{M}_{1,\cU,x}  + \mathcal{M}_{2,\cU,x} + \E_{\bU \sim \cU} \croch{ T_3(x,\bU)} 
    \label{eq.pr.pro.bias.multidim.eq-biais-H3}
    \enspace .
  \end{align}
  We now bound the last term in Eq.~\eqref{eq.pr.pro.bias.multidim.eq-biais-H3}. 
  Integrating \eqref{hyp.s-3-fois-derivable.alt} over $t \in I_{\bU}(x)$ yields 
  \begin{align}
    \notag \absj{T_3(x, \bU)} &\leq \CHtroisa \frac{1}{\absj{I_{\bU}(x)}} \int_{I_{\bU}(x)} \norm{t-x}^3_3 dt \\
    \notag &= \CHtroisa \sum_{i=1}^{\dimX} \croch{ \frac{1}{B_{i,\bU}(x) - A_{i,\bU}(x)} \int_{A_{i,\bU}(x)}^{B_{i,\bU}(x)} \absj{t_i - x_i}^3 dt_i }\\
    \notag &= \frac{\CHtroisa}{4} \sum_{i=1}^{\dimX} \frac{\paren{ B_{i,\bU}(x) - x_i}^4 + \paren{x_i - A_{i,\bU}(x)}^4 }{B_{i,\bU}(x) - A_{i,\bU}(x)} \\
    &\leq \frac{\CHtroisa}{4} \sum_{i=1}^{\dimX} \croch{\paren{ B_{i,\bU}(x) - x_i}^3 + \paren{x_i - A_{i,\bU}(x)}^3 }
    \label{eq.pr.pro.bias.multidim.maj-T3}
  \end{align}
  since $B_{i,\bU}(x) - A_{i,\bU}(x) \geq B_{i,\bU}(x) - x_i$ and $B_{i,\bU}(x) - A_{i,\bU}(x) \geq x_i - A_{i,\bU}(x) \,$.
  Therefore, 
  \begin{align}
    \E_{\bU\sim\cU} \absj{T_3(x, \bU)} &\leq \frac{\CHtroisa}{4} \sum_{i=1}^{\dimX} \croch{ m_{BBB,i,x,\cU} + m_{AAA,i,x,\cU} } = \mathcal{R}_{3,\cU,x} \enspace .
    \label{eq.pr.pro.bias.multidim.maj-ET3}
  \end{align}
  By definition of $\Biasinfty{\cU}(x)$ and Eq.~\eqref{eq.pr.pro.bias.multidim.eq-biais-H3}, 
  \begin{align*}
    \Biasinfty{\cU}(x) &= \paren{ s(x) - \E_{\bU} \croch{ \tilde{s}_{\bU}(x) } }^2 \\
    &= \paren{ \mathcal{M}_{1,\cU,x}  + \mathcal{M}_{2,\cU,x} + \E_{\bU \sim \cU} \croch{ T_3(x,\bU)}  }^2 \\
    &= \paren{ \mathcal{M}_{1,\cU,x}  + \mathcal{M}_{2,\cU,x} }^2 + \paren{ \E_{\bU \sim \cU} \croch{ T_3(x,\bU)}  }^2 + 2  \paren{ \mathcal{M}_{1,\cU,x}  + \mathcal{M}_{2,\cU,x} } \E_{\bU \sim \cU} \croch{ T_3(x,\bU)}  
  \end{align*}
  which leads to  Eq.~\eqref{eq.pro.bias.multidim.Bias-H3} thanks to Eq.~\eqref{eq.pr.pro.bias.multidim.maj-ET3}. 

\end{proof}

\section{Proofs: the (one-dimensional) toy model} \label{sec.pr.toy_model}

\subsection{Distribution of $A_{\bU,x}$ and $B_{\bU,x}$} \label{sec.pr.toy_model.distrib-A-B}
The purpose of the section is to specify the distributions of $A_{\bU,x}$ and $B_{\bU,x}$ in the ``toy model'' case. We will prove the following proposition. 
\begin{proposition} \label{pro.toy_model.distrib-A-B}
  Let $\nfeu \geq 2$ and $\bU \sim \cUtoy{\nfeu}$ as defined in Section~\ref{sec_toy_model}.
  \begin{enumerate}
  \item For every $x \in \croch{ \frac{1}{\nfeu}, 1-\frac{1}{\nfeu} }$, 
    \begin{equation} \label{eq.pro.toy_model.distrib-A-B.centre}
      A_{\bU,x} = x+\frac{V_x-1}{\nfeu} \qquad \mbox{and} \quad B_{\bU,x} = x+\frac{V_x}{\nfeu} \end{equation}
    where 
    \[ V_x \egaldef 1 + \left\lfloor \nfeu x + T \right\rfloor - (\nfeu x+T)  \]
    has a uniform distribution over $(0,1)$. 
  \item For every $x \in \croch{ 0 , \frac{1}{\nfeu} }$, 
    \begin{equation} \label{eq.pro.toy_model.distrib-A-B.gauche}
      \begin{cases}
        \mbox{with probability } \nfeu x \, ,  \quad &A_{\bU,x} \sim \mathcal{U}\paren{\croch{0,x}} \quad \mbox{and} \quad B_{\bU,x} = A_{\bU,x} + \frac{1}{\nfeu} 
        \\
        \mbox{with probability } 1-\nfeu x \, , \quad &A_{\bU,x} = 0 \quad \mbox{and} \quad B_{\bU,x} \sim \mathcal{U}\paren{\croch{0,\frac{1}{\nfeu} - x}} \, . 
      \end{cases}
    \end{equation}
  \item For every $x \in \croch{ \frac{\nfeu-1}{\nfeu} , 1 }$, 
    \begin{equation} \label{eq.pro.toy_model.distrib-A-B.droite}
      \begin{cases}
        \mbox{with probability } \nfeu(1-x) \, ,  \quad &A_{\bU,x} = B_{\bU,x} - \frac{1}{\nfeu}  \quad \mbox{and} \quad B_{\bU,x} \sim \mathcal{U}\paren{\croch{x,1}}
        \\
        \mbox{with probability } 1-\nfeu(1-x) \, , \quad &A_{\bU,x} \sim \mathcal{U}\paren{\croch{1 - \frac{1}{\nfeu} , x}} \quad \mbox{and} \quad  B_{\bU,x} = 1 \, . 
      \end{cases}
    \end{equation}
  \end{enumerate}
\end{proposition}
\begin{proof}[Proof of Proposition~\ref{pro.toy_model.distrib-A-B}]
  {\bf Proof of Eq.~\eqref{eq.pro.toy_model.distrib-A-B.centre}} \quad 
  If $x \in \croch{ \frac{1}{\nfeu}, 1-\frac{1}{\nfeu} }$, 
  \[ A_{\bU,x}  = \frac{i(x) - T}{\nfeu} \leq x < \frac{1+i(x)-T}{\nfeu} = B_{\bU,x} \]
  for some integer $i(x)$. So, 
  \[ i(x) \leq \nfeu x + T < i(x) + 1 \]
  hence 
  \[ i(x) = \left\lfloor \nfeu x + T \right\rfloor \enspace , \]
  which proves Eq.~\eqref{eq.pro.toy_model.distrib-A-B.centre}; $V_x$ has a uniform distribution over $(0,1)$ since $T$ has a uniform distribution over $[0,1)$.

  \noindent {\bf Proof of Eq.~\eqref{eq.pro.toy_model.distrib-A-B.gauche}} \quad 
  We now assume $x \in \croch{ 0 , \frac{1}{\nfeu} }$. Then, two events can occur: 
  \begin{itemize}
  \item[(a)] if $(1-T)/\nfeu \leq x\,$, then $A_{\bU,x} = \frac{1-T}{\nfeu}$ and $B_{\bU,x} = \frac{2-T}{\nfeu} = \nfeu^{-1} + A_{\bU,x}\,$. Therefore, conditionally to $\set{(1-T)/\nfeu \leq x}\,$, $A_{\bU,x} \sim \mathcal{U}\paren{\croch{0,x}}$. 
  \item[(b)] if $(1-T)/\nfeu > x\,$, then $A_{\bU,x} = 0$ and $B_{\bU,x} = \frac{1-T}{\nfeu}\,$. Therefore, conditionally to $\set{(1-T)/\nfeu > x}\,$, $B_{\bU,x} \sim \mathcal{U}\paren{\croch{x,\frac{1}{\nfeu}}}\,$. 
  \end{itemize}
  Since case (a) has probability $\nfeu x$ and case (b) has probability $1-\nfeu x$, we have proved Eq.~\eqref{eq.pro.toy_model.distrib-A-B.gauche}. 

  \noindent {\bf Proof of Eq.~\eqref{eq.pro.toy_model.distrib-A-B.droite}} \quad 
  We simply deduce Eq.~\eqref{eq.pro.toy_model.distrib-A-B.droite} from Eq.~\eqref{eq.pro.toy_model.distrib-A-B.gauche} by symmetry, since reflecting the unit interval $[0,1)$ only changes $x$ into $1-x$ and exchanges the roles of $A_{\bU,x}$ and $B_{\bU,x}\,$.
\end{proof}

\subsection{Proof of Proposition~\ref{toy_model_kernel}} \label{sec.pr.toy_model.link-kernel}
By Eq.~\eqref{eq.valeur-foret-x.multidim} and Eq.~\eqref{eq.pro.toy_model.distrib-A-B.centre}, for all $x \in \croch{ \frac{1}{\nfeu} , 1 - \frac{1}{\nfeu}}$, 
\begin{align}
  \notag 
  \tilde{s}_{\bU}(x) &= \frac{1}{\absj{I_{\bU}(x)}} \int_{I_{\bU}(x)} s(t)  dt \\
  &= \nfeu \int_0^1 s(t) \un_{x+\frac{V_x-1}{\nfeu} \leq t < x+\frac{V_x}{\nfeu}} \, dt \enspace .
  \notag
\end{align}
So, using the Fubini theorem:
\begin{align*}
  \tilde{s}_{\infty}(x) &= \E_{\bU} \croch{ \tilde{s}_{\bU}(x) } \\
  &= \nfeu \int_0^1 s(t) \P \paren{ x+\frac{V_x-1}{\nfeu} \leq t < x+\frac{V_x}{\nfeu} } \, dt \enspace .
\end{align*}
We conclude by remarking that 
\begin{align*}
  \P \paren{ x+\frac{V_x-1}{\nfeu} \leq t < x+\frac{V_x}{\nfeu} } &= \P \paren{ \nfeu (t-x) < V_x \leq \nfeu(t-x)+1 } \\
  &=\begin{cases} 1-\nfeu(t-x) \mbox{ if } x \leq t \leq x+\frac{1}{\nfeu} \\
    1+\nfeu(t-x) \mbox{ if } x-\frac{1}{\nfeu} \leq t < x
  \end{cases}
\end{align*}
since the distribution of $V_x$ is uniform over $[0,1)$.
\qed

\subsection{Computation of the key quantities} \label{sec.pr.toy_model.key-quant}

\begin{proposition} \label{pro.toy_model.moments}
  Let $\nfeu \geq 2$, $\bU \sim \cUtoy{\nfeu}$, $x \in [0,1)$, $g(x) \egaldef \nfeu \min\set{x , 1-x, \nfeu^{-1}}$, $\alpha_{\nfeu} = x - A_{\bU,x}$ and $\beta_{\nfeu} = B_{\bU,x} - x$. 
  Then, 
  \begin{align} 
    \label{eq.toy_model.AB.bord} 
    \E\croch{ \alpha_{\nfeu} \beta_{\nfeu} } &= \frac{g(x)}{\nfeu^2} \croch{ \frac{1}{2} - g(x) + g(x)^2 - \frac{g(x)^3}{3} }
    \\ \label{eq.toy_model.A-B.bord}
    \E\croch{ \alpha_{\nfeu}  - \beta_{\nfeu} } &= \paren{ \un_{x > 1 - \nfeu^{-1}} - \un_{x < \nfeu^{-1}} } \times \frac{(1-g(x))^2} {2 \nfeu} 
    \\ \label{eq.toy_model.AA+BB.bord}
    \E\croch{ \alpha_{\nfeu}^2 + \beta_{\nfeu}^2 } &= \frac{-2 g(x)^3 + 3 g(x)^2 + 1}{3 \nfeu^2}
    \\ \label{eq.toy_model.AAA+BBB.bord}
    \E\croch{ \alpha_{\nfeu}^3 + \beta_{\nfeu}^3 } &= \frac{-4 g(x)^3 + 9 g(x)^2 - 6 g(x) + 2}{2 \nfeu^3} 
    \\ \label{eq.toy_model.AAAA+BBBB.bord}
    \E\croch{ \alpha_{\nfeu}^4 + \beta_{\nfeu}^4 } &= \frac{1 + 5 g(x)^4 - 4 g(x)^5}{5 \nfeu^4} 
    \enspace .
  \end{align}
  In particular, if $x \in \croch{\nfeu^{-1} , 1-\nfeu^{-1}}$, 
  \begin{align} 
    \label{eq.toy_model.AB} 
    \E\croch{ \alpha_{\nfeu} \beta_{\nfeu} } &= \frac{1}{6 \nfeu^2}
    \\ \label{eq.toy_model.A-B}
    \E\croch{ \alpha_{\nfeu}  - \beta_{\nfeu} } &= 0
    \\ \label{eq.toy_model.Akap+Bkap}
    \forall \kappa \geq 1 \, , \quad 
    \E\croch{ \alpha_{\nfeu}^{\kappa} + \beta_{\nfeu}^{\kappa} } &= \frac{2}{(1+\kappa) \nfeu^{\kappa}}
    \enspace .
  \end{align}
\end{proposition}
The main results of Proposition~\ref{pro.toy_model.moments} are summarized in Table~\ref{tab.pro.toy_model.moments}. 
\begin{table}
  \begin{center}
    \begin{tabular}{|l|c|c|}
      \hline
      Quantity & Order of magnitude & Eq. number \\
      \hline
      $\E\croch{ \alpha_{\nfeu} -\beta_{\nfeu} }$ & $ \nfeu^{-1} \un_{x \notin [\nfeu^{-1}  , 1-\nfeu^{-1} ] }$ & \eqref{eq.toy_model.A-B.bord}, \eqref{eq.toy_model.A-B} \\
      $\E\croch{ \alpha_{\nfeu} \beta_{\nfeu}   }$  & $\nfeu^{-2}$ & \eqref{eq.toy_model.AB.bord}, \eqref{eq.toy_model.AB}  \\
      $\E\croch{ \alpha_{\nfeu}^{\kappa} + \beta_{\nfeu}^{\kappa} }$ ($\kappa=2,3,4$) & $\nfeu^{-\kappa}$ & \eqref{eq.toy_model.AA+BB.bord}, \eqref{eq.toy_model.AAA+BBB.bord}, \eqref{eq.toy_model.AAAA+BBBB.bord}, \eqref{eq.toy_model.Akap+Bkap} \\
      \hline
    \end{tabular}
  \end{center}
  \caption{Summary of the results proved by Proposition~\ref{pro.toy_model.moments} for the one-dimensional ``toy model''. \label{tab.pro.toy_model.moments}}
\end{table}

\begin{proof}[Proof of Proposition~\ref{pro.toy_model.moments}]
  \noindent {\bf First case: $x \in \croch{\nfeu^{-1} , 1-\nfeu^{-1}}$} \quad 
  By Proposition~\ref{pro.toy_model.distrib-A-B}, we get that $(\alpha_{\nfeu},\beta_{\nfeu}) = ((1-V)/\nfeu,V/\nfeu)$ for some random variable $V=V_x$ with uniform distribution over $[0,1)$. Therefore, 
  \begin{align*} 
    \E\croch{ \alpha_{\nfeu} \beta_{\nfeu} } &= \frac{1}{\nfeu^2} \E\croch{ (1-V) V } = \frac{1}{6 \nfeu^2}
    \\ 
    \E\croch{ \alpha_{\nfeu}  - \beta_{\nfeu} } &= \frac{1}{\nfeu} \E\croch{ 1 - 2 V} = 0
    \\ 
    \E\croch{ \alpha_{\nfeu}^{\kappa} + \beta_{\nfeu}^{\kappa} } &= \frac{2}{\nfeu^{\kappa}} \E\croch{ V^{\kappa}} = \frac{2}{(1 + {\kappa}) \nfeu^{{\kappa}}}
  \end{align*}
  for any $\kappa \geq 1$. 

  \noindent {\bf Second case: $x \notin \croch{\nfeu^{-1} , 1-\nfeu^{-1}}$} \quad 
  It is sufficient to consider the case $x < \nfeu^{-1}$ since we can deduce results when $x > 1 - \nfeu^{-1}$ by symmetry (exchanging $\alpha_{\nfeu}$ and $\beta_{\nfeu}$, and replacing $x$ by $1-x$). 
  By Proposition~\ref{pro.toy_model.distrib-A-B}, some random variable $V$ with uniform distribution over $[0,1)$ exists such that 
  \begin{equation*}
    \begin{cases}
      (\alpha_{\nfeu},\beta_{\nfeu}) = \paren{ x V, \nfeu^{-1} - xV } &\mbox{with probability } \nfeu x \\
      (\alpha_{\nfeu},\beta_{\nfeu}) = \paren{ x  , (\nfeu^{-1} - x) V } &\mbox{with probability } 1 - \nfeu x
    \end{cases}
  \end{equation*}
  Therefore, we can compute all the desired quantities as follows. 
  \begin{align*}
    \E\croch{ \alpha_{\nfeu} \beta_{\nfeu} } &= \nfeu x \E\croch{ x V \paren{ \nfeu^{-1} - xV }} + \paren{ 1 - \nfeu x} \E\croch{ x \paren{ \nfeu^{-1} - x } V } \\
    &= \frac{g(x)^2 \E\croch{ V (1- g(x) V)}} {\nfeu^2} + \frac{(1-g(x))^2 g(x) \E\croch{ V }}{\nfeu^2} \\
    &= \frac{g(x)}{\nfeu^2} \croch{ g(x) \paren{ \frac{g(x)}{2} - \frac{g(x)^2}{3}} + \frac{(1-g(x))^2}{2} }
    \\ 
    &= \frac{g(x)}{\nfeu^2} \croch{ \frac{1}{2} - g(x) + g(x)^2 - \frac{g(x)^3}{3} }
    \\ 
    \E\croch{ \alpha_{\nfeu} - \beta_{\nfeu} } &= \nfeu x \E\croch{ 2xV - \nfeu^{-1} } + (1 - \nfeu x) \E\croch{ x - (\nfeu^{-1} - x) V } \\
    &= \frac{1}{\nfeu} \paren{ g(x) \E\croch{ 2 g(x) V - 1} + (1 - g(x)) \E\croch{ g(x) - (1-g(x)) V} } \\
    &= \frac{1}{2 \nfeu} \paren{ 2 g(x) (g(x)-1) + (1-g(x)) (3g(x) - 1) } = \frac{- (1-g(x))^2} {2 \nfeu} 
    \\ 
    \E\croch{ \alpha_{\nfeu}^2 + \beta_{\nfeu}^2 } &= \nfeu x \E\croch{ x^2 V^2 + \paren{\nfeu^{-1} - xV}^2 } + (1 - \nfeu x) \paren{ x^2 + \nfeu^{-2} (1-\nfeu x)^2 \E\croch{V^2} } \\
    &= \frac{1}{\nfeu^2} \croch{ g(x) \paren{ \frac{g(x)^2}{3} + 1 - g(x) + \frac{g(x)^2}{3} } + (1-g(x)) \paren{ g(x)^2 + \frac{(1-g(x))^2 }{3} } } \\
    &= \frac{-2 g(x)^3 + 3 g(x)^2 + 1}{3 \nfeu^2}
    \\
    \E\croch{ \alpha_{\nfeu}^3 + \beta_{\nfeu}^3 } &= \nfeu x \E\croch{ x^3 V^3 + \nfeu^{-3} (1 - \nfeu x V)^3} + (1-\nfeu x) \croch{ x^3 + \nfeu^{-3} (1-\nfeu x)^3 \E\croch{V^3}} \\
    &= \frac{1}{\nfeu^3} \croch{ g(x) \paren{ \frac{g(x)^3}{4} + 1 - \frac{3g(x)}{2} + g(x)^2 - \frac{g(x)^3}{4}} + (1-g(x)) \paren{ g(x)^3 + \frac{(1-g(x))^3}{4}} } \\
    &= \frac{-4 g(x)^3 + 9 g(x)^2 - 6 g(x) + 2}{2 \nfeu^3} 
    \\ 
    \E\croch{ \alpha_{\nfeu}^4 + \beta_{\nfeu}^4 } &= \nfeu x \E\croch{ x^4 V^4 + \nfeu^{-4} (1 - \nfeu x V)^4} + (1-\nfeu x) \croch{ x^4 + \nfeu^{-4} (1-\nfeu x)^4 \E\croch{V^4}} \\
    &= \frac{1}{\nfeu^4} \croch{ g(x) \paren{ \frac{g(x)^4}{5} + 1 - 2g(x) + 2g(x)^2 - g(x)^3 + \frac{g(x)^4}{5} } + (1-g(x)) \paren{ g(x)^4 + \frac{(1-g(x))^4}{5}} } \\
    &= \frac{1 + 5 g(x)^4 - 4 g(x)^5}{5 \nfeu^4} 
  \end{align*}
\end{proof}

\subsection{Proof of Corollary~\ref{cor.bias.toy.H2}} \label{sec.pr.toy_model.cor-H2}
The proof directly follows from the combination of Proposition~\ref{pro.bias.multidim.H3} and Proposition~\ref{pro.toy_model.moments}. 

First, using Proposition~\ref{pro.toy_model.moments}, we compute the key quantities appearing in the result of Proposition~\ref{pro.bias.multidim.H3} under assumptions \eqref{hyp.s-2-fois-derivable.alt} (with $\dimX=1$) and \eqref{hyp.unif}. 
\begin{align*}
  \mathcal{M}_{1,\cU,x} &= s^{\prime}(x) \croch{ \un_{x > 1 - \nfeu^{-1}} - \un_{x < \nfeu^{-1} } } \frac{ \PolMun(g(x)) }{\nfeu}
  \qquad \mbox{where} \quad \PolMun(X) \egaldef \frac{(1-X)^2}{4} 
  \\
  \mathcal{N}_{2,\cU,x} &= \paren{ s^{\prime}(x) }^2 \frac{ \PolNdeu(g(x)) }{\nfeu^2} 
  \qquad \mbox{where} \quad \PolNdeu(X) \egaldef \frac{1}{12} \paren{ 2 X^4 - 8 X^3 + 9 X^2 - 3 X + 1} 
  \\
  \mathcal{R}_{2,\cU,x} &= \CHdeuxa \frac{ \PolRdeu(g(x)) }{\nfeu^2} 
  \qquad \mbox{where} \quad \PolRdeu(X) \egaldef \frac{1}{9} \paren{ X^4 - 5 X^3 + 6 X^2 - \frac{3}{2} X + 1} 
  \\
  \mathcal{R}_{4,\cU,x} &= \CHdeuxa^2 \frac{ \PolRqua(g(x)) }{\nfeu^4} 
  \qquad \mbox{where} \quad \PolRqua(X) \egaldef \frac{2}{45} \paren{ - 4 X^5 + 5 X^4 + 1 } 
  \enspace .
\end{align*}
So, Eq.~\eqref{eq.pro.bias.multidim.Bias-H2} yields, for every $x \in [0,1)$, 
\begin{align*}
  \Biasinfty{\cU}(x) &\leq \paren{ \mathcal{M}_{1,\cU,x} + \mathcal{R}_{2,\cU,x}}^2 
  \leq \paren{s^{\prime}(x)}^2 \frac{ \un_{g(x)<1} }{16 \nfeu^2} + \frac{\CHdeuxa \norm{s^{\prime}}_{\infty} \un_{g(x)<1} }{2 \nfeu^3} + \frac{\CHdeuxa^2 \PolRdeu(g(x))^2}{\nfeu^4}
\end{align*}
where we used that $\croch{ \un_{x > 1 - \nfeu^{-1}} - \un_{x < \nfeu^{-1} } }^2 \PolMun(g(x))^2 = \PolMun(g(x))^2 \leq 1/16 \un_{g(x)<1}$ and $\absj{\PolRdeu(g(x))} \leq 1$. 
Considering separately the cases $g(x)=1$ and $g(x)<1$ yields Eq.~\eqref{eq.bias.toy.H2.Biasinfty} and~\eqref{eq.bias.toy.H2.Biasinfty.border} (since $\absj{\PolRdeu(g(x))} \leq 1$). 

Integrating Eq.~\eqref{eq.bias.toy.H2.Biasinfty} over $x \in [\nfeu^{-1} , 1 - \nfeu^{-1} ]$ directly leads to Eq.~\eqref{eq.bias.toy.H2.Biasinfty.integrated-noborder}. 
Integrating Eq.~\eqref{eq.bias.toy.H2.Biasinfty.border} over $x \in [0,\nfeu^{-1}] \cup [1 - \nfeu^{-1} , 1]$ yields 
\begin{align*}
  \int_{[0,\nfeu^{-1}] \cup [1 - \nfeu^{-1} , 1]} \Biasinfty{\cU}(x) \, dx 
  &\leq \frac{1}{16 \nfeu^2 } \int_{[0,\nfeu^{-1}] \cup [1 - \nfeu^{-1} , 1]} \paren{ s^{\prime}(x) }^2 \, dx + \frac{ \CHdeuxa \norm{s^{\prime}}_{\infty} + \CHdeuxa^2}{\nfeu^4} 
  \\
  &\leq \frac{\norm{s^{\prime}}_{\infty}^2}{8 \nfeu^3 } + \frac{ \CHdeuxa \norm{s^{\prime}}_{\infty} + \CHdeuxa^2}{\nfeu^4}
\end{align*}
which implies Eq.~\eqref{eq.bias.toy.H2.Biasinfty.integrated} together with Eq.~\eqref{eq.bias.toy.H2.Biasinfty.integrated-noborder}.

Furthermore, Eq.~\eqref{eq.pro.bias.multidim.Var-H2} yields, for every $x \in [0,1)$, 
\begin{align*}
  \absj{ \Vararbre{\cU}(x) - \frac{\paren{ s^{\prime}(x) }^2 \paren{ \PolNdeu(g(x)) - \PolMun(g(x))^2 } }{\nfeu^2} } 
  &\leq \frac{ 2 \norm{s^{\prime}}_{\infty}}{\nfeu} \sqrt{ \mathcal{R}_{4,\cU,x} } + \mathcal{R}_{4,\cU,x} \\
  &\leq \frac{ 2 \norm{s^{\prime}}_{\infty} \CHdeuxa }{\nfeu^3} + \frac{ \CHdeuxa^2}{3 \nfeu^4}  \\
  &\leq \frac{ 2 \norm{s^{\prime}}_{\infty} \CHdeuxa + \CHdeuxa^2}{\nfeu^3} 
\end{align*}
where we used that $\croch{ \un_{x > 1 - \nfeu^{-1}} - \un_{x < \nfeu^{-1} } }^2 \PolMun(g(x))^2 = \PolMun(g(x))^2$, $\absj{ \PolNdeu(g(x)) } \leq 1$ and $\absj{ \PolRqua(g(x)) } \leq 1/3$. 
If $x \in [\nfeu^{-1} , 1 - \nfeu^{-1}]$, we have $g(x)=1$ and $\PolNdeu(1) - \PolMun(1)^2 = 1/12$, which proves Eq.~\eqref{eq.bias.toy.H2.Vararbre}. 
Otherwise, we have proved Eq.~\eqref{eq.bias.toy.H2.Vararbre.border} with 
\[ \PolVarU(X) \egaldef \PolNdeu(X) - \PolMun(X)^2 = \frac{1}{12} \paren{ 2 X^4 - 8 X^3 + 9 X^2 - 3 X + 1}  -  \frac{(1-X)^4}{16} \enspace .  \]
Integrating Eq.~\eqref{eq.bias.toy.H2.Vararbre} and~\eqref{eq.bias.toy.H2.Vararbre.border} over $x$ yields 
\begin{align*}
  &\qquad \absj{ \Vararbre{\cU} - \frac{1}{12 \nfeu^2} \int_0^1 \paren{ s^{\prime}(x) }^2 \, dx } 
  \\
  &\leq \frac{2 \norm{s^{\prime}}_{\infty} \CHdeuxa + \CHdeuxa^2 }{ \nfeu^3 } 
  + \frac{1}{12 \nfeu^2}  \int_0^{\nfeu^{-1}} \paren{ s^{\prime}(x) }^2 \absj{ 12 Q(\nfeu x) - 1} \, dx 
  \\
  &\qquad + \frac{1}{12 \nfeu^2}  \int_{1-\nfeu^{-1}}^{1} \paren{ s^{\prime}(x) }^2 \absj{ 12 Q(\nfeu(1-x)) - 1} \, dx 
  \\
  &\leq \frac{2 \norm{s^{\prime}}_{\infty} \CHdeuxa + \CHdeuxa^2 + 3 \norm{s^{\prime}}_{\infty}^2 }{ \nfeu^3 } 
  \enspace . 
\end{align*}
\qed

\subsection{Proof of Corollary~\ref{cor.bias.toy.H3}} \label{sec.pr.toy_model.cor-H3}
Using Proposition~\ref{pro.toy_model.moments}, we compute the quantities appearing in the result of Proposition~\ref{pro.bias.multidim.H3} under assumptions \eqref{hyp.s-3-fois-derivable.alt} (with $\dimX=1$) and \eqref{hyp.unif}
\begin{align*}
  \mathcal{M}_{2,\cU,x} &=  \frac{ s^{\prime\prime}(x) \PolRdeu(g(x)) }{ 2 \nfeu^2} 
  \\
  \mathcal{R}_{3,\cU,x} &= \CHtroisa \frac{\PolRtro(g(x))}{\nfeu^3} 
  \qquad \mbox{where} \quad \PolRtro(X) \egaldef \frac{1}{8} \paren{ -4X^3 + 9 X^2 - 6 X^2 + 2}
  \enspace .
\end{align*}
Then, Eq.~\eqref{eq.pro.bias.multidim.Bias-H3} yields
\begin{align*}
  & \qquad \absj{ \Biasinfty{\cU}(x) - \paren{ \mathcal{M}_{1,\cU,x}^2 + \mathcal{M}_{2,\cU,x}^2 } } \\
  &= \absj{ \Biasinfty{\cU}(x) - \frac{ \paren{s^{\prime}(x)}^2 \PolMun(g(x))^2}{\nfeu^2} - \frac{\paren{s^{\prime\prime}(x)}^2 \PolRdeu(g(x))^2 }{4 \nfeu^4} }  \\
  &\leq 2 \absj{\mathcal{M}_{1,\cU,x}} \absj{\mathcal{M}_{2,\cU,x}} + 2 \absj{ \mathcal{R}_{3,\cU,x} \paren{ \mathcal{M}_{1,\cU,x} + \mathcal{M}_{2,\cU,x} } } + \paren{ \mathcal{R}_{3,\cU,x} }^2 \\
  &\leq \frac{\norm{s^{\prime}}_{\infty} \norm{s^{\prime\prime} }_{\infty} \PolMun(g(x))}{ \nfeu^3 } 
  +  \frac{3 \CHtroisa }{\nfeu^4} \paren{ \norm{s^{\prime}}_{\infty} \PolMun(g(x)) + \frac{\norm{s^{\prime\prime}}_{\infty} }{4 \nfeu} } + \frac{9 \CHtroisa^2 }{4 \nfeu^6}
\end{align*}
where we used that 
$\absj{ \un_{x > 1 - \nfeu^{-1}} - \un_{x < \nfeu^{-1} } } \PolMun(g(x)) = \PolMun(g(x))$, 
$\absj{ \PolRdeu(g(x)) } \leq 1$, 
$\absj{\PolRtro(g(x))} \leq 3/2 $. 
When $x \in [\nfeu^{-1} , 1-\nfeu^{-1}]$, we get that 
\begin{align*}
  \absj{ \Biasinfty{\cU}(x) - \frac{\paren{s^{\prime\prime}(x)}^2 \PolRdeu(g(x))^2 }{4 \nfeu^4} }  
  \leq \frac{3 \CHtroisa \norm{s^{\prime\prime}}_{\infty} }{4 \nfeu^5}  + \frac{9 \CHtroisa^2 }{4 \nfeu^6} 
  \leq \frac{ 3 \CHtroisa \paren{ \norm{s^{\prime\prime}}_{\infty} + \frac{3 \CHtroisa}{2}}}{4 \nfeu^5}
\end{align*}
which implies Eq.~\eqref{eq.bias.toy.H3.Biasinfty}. 
When $x \in [0,\nfeu^{-1}] \cup [1-\nfeu^{-1} , 1]$, we get that 
\begin{align*}
  &\qquad  \absj{ \Biasinfty{\cU}(x) - \frac{ \paren{s^{\prime}(x)}^2 \PolMun(g(x))^2}{\nfeu^2}  }  \\
  &\leq \frac{\norm{s^{\prime}}_{\infty} \norm{s^{\prime\prime} }_{\infty} \PolMun(g(x))}{ \nfeu^3 } 
  +  \frac{3 \CHtroisa }{\nfeu^4} \paren{ \norm{s^{\prime}}_{\infty} \PolMun(g(x)) + \frac{\norm{s^{\prime\prime}}_{\infty} }{4 \nfeu} } 
  + \frac{9 \CHtroisa^2 }{4 \nfeu^6} 
  + \frac{\paren{s^{\prime\prime}(x)}^2 \PolRdeu(g(x))^2 }{4 \nfeu^4} \\
  &\leq \frac{\norm{s^{\prime}}_{\infty} \norm{s^{\prime\prime} }_{\infty}}{ 4 \nfeu^3 } 
  + \frac{3 \CHtroisa \norm{s^{\prime}}_{\infty} + \norm{s^{\prime\prime}}_{\infty}^2 }{4 \nfeu^4}  
  + \frac{3 \CHtroisa \norm{s^{\prime\prime}}_{\infty} }{4 \nfeu^5}
  + \frac{9 \CHtroisa^2 }{4 \nfeu^6} \\
  &\leq \frac{\norm{s^{\prime}}_{\infty} \norm{s^{\prime\prime} }_{\infty} + \frac{3}{2}  \CHtroisa \norm{s^{\prime}}_{\infty} + \frac{1}{2} \norm{s^{\prime\prime}}_{\infty}^2 + \frac{3}{4}  \CHtroisa \norm{s^{\prime\prime}}_{\infty} + \frac{9}{8} \CHtroisa^2 }{ 4 \nfeu^3 } \\
  &\leq \frac{\norm{s^{\prime}}_{\infty} \norm{s^{\prime\prime} }_{\infty} + 2  \CHtroisa \norm{s^{\prime}}_{\infty} + \norm{s^{\prime\prime}}_{\infty}^2 +  \CHtroisa \norm{s^{\prime\prime}}_{\infty} + 2 \CHtroisa^2 }{ 4 \nfeu^3 } 
\end{align*}
which implies Eq.~\eqref{eq.bias.toy.H3.Biasinfty.border}. 

Finally, integrating Eq.~\eqref{eq.bias.toy.H3.Biasinfty} yields Eq.~\eqref{eq.bias.toy.H3.Biasinfty.integrated-noborder}, and integrating Eq.~\eqref{eq.bias.toy.H3.Biasinfty.border} yields Eq.~\eqref{eq.bias.toy.H3.Biasinfty.integrated}. 
\qed

\section{Proofs: the (one-dimensional) purely uniformly random forest model} \label{sec.pr.purf}

\subsection{Distribution of $A_{\bU,x}$ and $B_{\bU,x}$}\label{sec.pr.purf.distrib-A-B}

Let $\nfeu \geq 1$ and $x \in (0,1)$ be fixed, $\bU \sim \cUpurf{\nfeu}$ and $\xi_1, \ldots, \xi_{\nfeu}$ be the i.i.d. uniform variables used for defining $\bU$. 
Then, 
\begin{align*}
  A_{\bU,x} &= \max \set{ \xi_i \telque i\in\set{1,\ldots,\nfeu} \mbox{ and } \xi_i \leq x } 
  \\
  B_{\bU,x} &= \min \set{ \xi_i \telque i\in\set{1,\ldots,\nfeu} \mbox{ and } \xi_i > x } 
\end{align*}
with the conventions $\max \emptyset = 0$ and $\min \emptyset = 1$.
Let us define 
\[ N_{\bU,x} \egaldef \card\set{ i\in\set{1,\ldots,\nfeu} \telque \xi_i \leq x } \enspace . \]
Then, we have the following proposition.
\begin{prop} \label{pro.dist-A-B.purf}
  Let $\nfeu \geq 1$ be some integer, $x\in(0,1)$, $l \in \set{ 0 , \ldots, \nfeu }$, $\bU \sim \cUpurf{\nfeu}$, $V_1, \ldots, V_\nfeu$ be i.i.d. uniform random variables over $[0,x]$ and $W_1, \ldots, W_\nfeu$ be i.i.d. uniform random variables over $(0,1-x)$.
  Then, conditionally to the event $\set{ N_{\bU,x} = \narb }$, $A_{\bU,x}$ and $B_{\bU,x}$ are independent with the following distributions:
  \begin{align*}
    \mbox{if } l \neq 0 \, , \qquad  A_{\bU,x} &\egalloi \max\set{ V_1, \ldots, V_l } \enspace ; \\
    \mbox{if } l \neq \nfeu \, , \qquad   B_{\bU,x} &\egalloi \min\set{ W_1, \ldots, W_{\nfeu-l} }  + x \enspace ;
  \end{align*} 
  $A_{\bU,x}=0$ a.s. if $l=0$\,\textup{;} and $B_{\bU,x}=1$ a.s. if $l=\nfeu$.

  As a consequence, for every $s \in [0,1-x]$ and $t \in [0,x]$, 
  \begin{align}
    \label{eq.purf.loi-jointe-alpha-beta}
    \P\paren{ A_{\bU,x} \leq x - t \, , \, B_{\bU,x} \geq  x + s }
    &= (1-t-s)^{\nfeu} 
    \\
    \label{eq.purf.loi-alpha}
    \P\paren{ A_{\bU,x} \leq x - t }
    &= (1-t)^{\nfeu} 
    \\
    \label{eq.purf.loi-beta}
    \P\paren{ B_{\bU,x} \geq x + s }
    &= (1-s)^{\nfeu} 
    \enspace .
  \end{align}
\end{prop}
\begin{proof}[Proof of Proposition~\ref{pro.dist-A-B.purf}]
\hfill \\ 
  \paragraph{Distribution conditionally to $N_{\bU,x}$}
  First, remark that for every $i$, $\loi\paren{\xi_i \sachant \xi_i \leq x} = \loi\paren{ V_1}$ and $\loi\paren{\xi_i - x\sachant \xi_i > x} = \loi\paren{ W_1}$. 
  This implies the result when $l \in \set{0,\nfeu}$, the independence between a deterministic variable and any random variable being straightforward.

  Let us now assume $l \in \set{1 , \ldots, \nfeu-1}$.
  For every $\delta = (\delta_1, \ldots, \delta_\nfeu) \in \set{0,1}^{\nfeu}$ such that $\sum_{i=1}^{\nfeu} \delta_i = l$, conditionally to 
  \[ \Omega(\delta) \egaldef \set{ \un_{\xi_1 \leq x} = \delta_1 , \ldots , \un_{\xi_\nfeu \leq x} = \delta_\nfeu } \enspace , \]
  we have 
  \begin{align*} 
    A_{\bU,x} &= \max\set{ \xi_i \telque \delta_i = 1} \egalloi \max\set{ V_1, \ldots, V_l } \\
    \mbox{and} \quad 
    B_{\bU,x} &= \max\set{ \xi_i \telque \delta_i = 0} \egalloi \max\set{ W_1, \ldots, W_{\nfeu-l} } + x
    \enspace . \end{align*} 
  Since 
  \[ \loi\paren{ (A_{\bU,x},B_{\bU,x}) \sachant \Omega(\delta) } \] is the same for all $\delta\in \set{0,1}^{\nfeu}$ such that $\sum_{i=1}^{\nfeu} \delta_i = l$, we get that 
  \[ \loi\paren{ (A_{\bU,x},B_{\bU,x}) \sachant N_{\bU,x}=l } = \loi\paren{ (A_{\bU,x},B_{\bU,x}) \sachant \Omega \paren{ \delta^{(l)} } } \]
  where $\delta^{(l)}_1=\dots=\delta^{(l)}_l=1$ and $\delta^{(l)}_{l+1}=\dots=\delta^{(l)}_{\nfeu}=0$, hence the result. 

  \paragraph{Joint unconditional distribution}
  Since $N_{\bU,x}$ has a binomial distribution with parameters $(n,x)$, we get from the conditional distribution of $(A_{\bU,x},B_{\bU,x})$ that for every $t \in [0,x]$ and $s \in [0,1-x]$, 
  \begin{align*}
    & \quad \P\paren{ A_{\bU,x} \leq x - t \, , \, B_{\bU,x} \geq x + s }
    \\
    &= \sum_{l=0}^{\nfeu} \croch{ \P\paren{N_{\bU,x} = l} \P\paren{ \max_{1 \leq i \leq l} V_i \leq x - t}  \P\paren{ \min_{1 \leq i \leq n-l} W_i \geq s} }
    \\
    &= \sum_{l=0}^{\nfeu} \croch{ {\nfeu \choose l} x^l (1-x)^{\nfeu - l} \paren{ \frac{x-t}{x}}^{l} \paren{ \frac{1 - x - s}{1-x}}^{\nfeu - l} }
    \\
    &= \sum_{l=0}^{\nfeu} \croch{ {\nfeu \choose l} \paren{ x - t}^{l} \paren{ 1 - x - s}^{\nfeu - l} }
    \\
    &= (1-t-s)^{\nfeu} 
    \enspace ,
  \end{align*}
  were on the second line, we used the convention $\max \emptyset = 0$ and $\min \emptyset = 1-x$. 
  Eq.~\eqref{eq.purf.loi-alpha} and~\eqref{eq.purf.loi-beta} respectively follow by taking $t=0$ (resp. $s=0$), since $A_{\bU,x} \leq x$ and $B_{\bU,x} \geq x$ a.s.
\end{proof}

\subsection{Computation of the key quantities} \label{sec.pr.purf.key-quant}

\begin{proposition} \label{pro.purf.moments}
  Let  $\nfeu \geq 1$, $\bU \sim \cUpurf{\nfeu}$, $x \in [0,1)$, $\alpha_{\nfeu} \egaldef x - A_{\bU,x}$, $\beta_{\nfeu} \egaldef B_{\bU,x} - x$, and for every $j$, $P_j(x) \egaldef x^j + (1-x)^j$.
  Then, 
  \begin{align} 
    \label{eq.purf.A-B.general}
    \E\croch{ \alpha_{\nfeu}  - \beta_{\nfeu} } 
    &= \frac{\PolPURFdiff{\nfeu}(x)}{\nfeu +1} 
    \quad \mbox{with} \quad 
    \PolPURFdiff{\nfeu}(x) \egaldef x^{\nfeu +1} - (1-x)^{\nfeu +1}
    \\ \label{eq.purf.AB.general} 
    \E\croch{ \alpha_{\nfeu} \beta_{\nfeu} } 
    &= \frac{1 + \PolPURFprod{\nfeu}(x)}{(\nfeu+1)(\nfeu+2)} 
    \quad \mbox{with} \quad 
    \PolPURFprod{\nfeu}(x) \egaldef - P_{\nfeu+2}(x)
    \\ \label{eq.purf.Akap+Bkap.general}
    \forall \kappa \in \set{2,3,4} \, , 
    \quad 
    \E\croch{ \alpha_{\nfeu}^{\kappa} + \beta_{\nfeu}^{\kappa} } &= \frac{2 (\kappa!) + \RestePURF_{1+\kappa,\nfeu}(x)}{\prod_{j=1}^{\kappa} (n+j)}
    \\
    \notag 
    \mbox{with} \quad 
    \PolPURFdeu{\nfeu}(x) &\egaldef -2 P_{\nfeu+2}(x) - 2x(1-x) (\nfeu+2) P_{\nfeu}(x)
    \\
    \notag 
    \PolPURFtro{\nfeu}(x) &\egaldef -6P_{\nfeu+3}(x) - 6x(1-x)(\nfeu+3) P_{\nfeu+2}(x) \\
    \notag 
    &\quad - 3 x^2 (1-x)^2 (\nfeu+2) (\nfeu+3) P_{\nfeu-1}(x)
    \\
    \notag 
    \mbox{and} \quad 
    \PolPURFqua{\nfeu}(x) &\egaldef -4 x^3(1-x)^3 (\nfeu+2)(\nfeu+3)(\nfeu+4) P_{\nfeu-2}(x) 
    \\
    \notag 
    & -12 x^2(1-x)^2 (\nfeu+3)(\nfeu+4) P_{\nfeu}(x) 
    -24 x(1-x) (\nfeu+4) P_{\nfeu+2}(x) 
    -24 P_{\nfeu+4}(x) 
    \enspace .
  \end{align}
  Note that whatever $x \in [0,1)$, 
  \begin{gather}
    \label{eq.purf.maj.restes.general} 
    \absj{ \PolPURFdiff{\nfeu}(x)} \leq 1
    \qquad 
    -1 \leq \PolPURFprod{\nfeu}(x) \leq 0 
    \qquad 
    \forall \kappa\in\set{2,3,4} \, , \quad 
    -2(\kappa!) \leq \RestePURF_{1+\kappa,\nfeu}(x) \leq 0 
  \end{gather}

  Assume $\nfeu \geq 27$.  
  Let $\epspurf \egaldef \frac{4\log \nfeu}{\nfeu} \leq 1/2$, $I_{\nfeu} \egaldef [\epspurf,1-\epspurf]$. 
  Then, 
  \begin{align}
    \label{eq.purf.A-B.maj.milieu} 
    \sup_{x \in I_{\nfeu}} \absj{ \PolPURFdiff{\nfeu}(x)} &\leq 
    \nfeu^{-4}
    \\
    \label{eq.purf.AB.maj.milieu} 
    \sup_{x \in I_{\nfeu}} \absj{ \PolPURFprod{\nfeu}(x)} &\leq 
    2 \nfeu^{-4}
    \\
    \label{eq.purf.A2+B2.maj.milieu} 
    \sup_{x \in I_{\nfeu}} \absj{ \PolPURFdeu{\nfeu}(x)} &\leq 
    \frac{11}{9} \nfeu^{-3}
    \\
    \label{eq.purf.A3+B3.maj.milieu} 
    \sup_{x \in I_{\nfeu}} \absj{ \PolPURFtro{\nfeu}(x)} &\leq 
    \nfeu^{-2}
    \\
    \label{eq.purf.A4+B4.maj.milieu} 
    \sup_{x \in I_{\nfeu}} \absj{ \PolPURFqua{\nfeu}(x)} &\leq 
    0.548 \nfeu^{-1} \leq 
    \nfeu^{-1}
    \enspace . 
  \end{align}
\end{proposition}
The main results of Proposition~\ref{pro.purf.moments} are summarized in Table~\ref{tab.pro.purf.moments}. 
\begin{table}
  \begin{center}
    \begin{tabular}{|l|c|c|}
      \hline
      Quantity & Order of magnitude & Eq. number \\
      \hline
      $\E\croch{ \alpha_{\nfeu} -\beta_{\nfeu} }$ & $\leq \nfeu^{-4} + \nfeu^{-1}\un_{x \notin\croch{\epspurf , 1-\epspurf} } $ & \eqref{eq.purf.A-B.general} \\
      $\E\croch{ \alpha_{\nfeu} \beta_{\nfeu}   }$  & $\nfeu^{-2}$ & \eqref{eq.purf.AB.general}  \\
      $\E\croch{ \alpha_{\nfeu}^{\kappa} + \beta_{\nfeu}^{\kappa} }$ ($\kappa=2,3,4$) & $\nfeu^{-\kappa}$ & \eqref{eq.purf.Akap+Bkap.general}\\
      \hline
    \end{tabular}
  \end{center}
  \caption{Summary of the results proved by Proposition~\ref{pro.purf.moments} for the one-dimensional PURF model. \label{tab.pro.purf.moments}}
\end{table}

\begin{proof}[Proof of Proposition~\ref{pro.purf.moments}]
  Since $\alpha_{\nfeu} = x - A_{\bU,x}$ and $\beta_{\nfeu} = B_{\bU,x} - x$, from Eq.~\eqref{eq.purf.loi-jointe-alpha-beta}, we get the joint distribution of $(\alpha_{\nfeu},\beta_{\nfeu})$: 
  \begin{equation}
    \label{eq.purf.joint-distribution-alpha-beta} 
    \forall t \in [0,x] \, , \, \forall x \in [0,1-x] \, , \quad 
    \P\paren{ \alpha_{\nfeu} \leq t \, , \, \beta_{\nfeu} \geq s } = (1-t-s)^{\nfeu} \enspace . 
  \end{equation}
  We deduce the marginal distributions of $\alpha_{\nfeu}$ and $\beta_{\nfeu}$ similarly to Eq.~\eqref{eq.purf.loi-alpha} and~\eqref{eq.purf.loi-beta}. 

  \paragraph{Formulas for all $x\in[0,1)$}
  We deduce from Eq.~\eqref{eq.purf.joint-distribution-alpha-beta} that 
  \begin{align*}
    \E\croch{ \alpha_{\nfeu}  - \beta_{\nfeu} } 
    &= \int_0^x (1-t)^{\nfeu} \, dt - \int_0^{1-x} (1-t)^{\nfeu} \, dt
    = \frac{x^{\nfeu+1}-(1-x)^{\nfeu+1}}{\nfeu+1}
    \\ 
    \E\croch{ \alpha_{\nfeu} \beta_{\nfeu} } &= \int_0^x \paren{ \int_0^{1-x} (1-t-s)^{\nfeu} \, ds} dt
    \\
    &=\frac{1}{(\nfeu+1)(\nfeu+2)} - \frac{x^{\nfeu+2}+(1-x)^{\nfeu+2}}{(\nfeu+1)(\nfeu+2)}
    \\ 
    \E\croch{ \alpha_{\nfeu}^2 + \beta_{\nfeu}^2 } &= \int_0^x 2t(1-t)^{\nfeu} \, dt + \int_0^{1-x} 2t(1-t)^{\nfeu} \, dt
    \\
    &=\frac{4}{(\nfeu+1)(\nfeu+2)} -2 \, \frac{(1-x)x^{\nfeu+1}+x(1-x)^{\nfeu+1}}{\nfeu+1} -2 \, \frac{x^{\nfeu+2}+(1-x)^{\nfeu+2}}{(\nfeu+1)(\nfeu+2)}
    \\ 
    \E\croch{ \alpha_{\nfeu}^3 + \beta_{\nfeu}^3 } &= \int_0^x 3t^2(1-t)^{\nfeu} \, dt + \int_0^{1-x} 3t^2(1-t)^{\nfeu} \, dt
    \\
    &= \frac{12}{(\nfeu+1)(\nfeu+2)(\nfeu+3)} -3 \, \frac{(1-x)^2 x^{\nfeu+1} + x^2 (1-x)^{\nfeu+1}}{\nfeu+1}
    \\
    & \quad -6 \, \frac{(1-x) x^{\nfeu+2} + x(1-x)^{\nfeu+2}}{(\nfeu+1)(\nfeu+2)} -6 \, \frac{x^{\nfeu+3} + (1-x)^{\nfeu+3}}{(\nfeu+1)(\nfeu+2)(\nfeu+3)}
    \\ 
    \E\croch{ \alpha_{\nfeu}^4 + \beta_{\nfeu}^4 } &= \int_0^x 4t^3 (1-t)^{\nfeu} \, dt + \int_0^{1-x} 4t^3 (1-t)^{\nfeu} \, dt
    \\
    &= \frac{48}{(\nfeu+1)(\nfeu+2)(\nfeu+3)(\nfeu+4)} -4 \, \frac{(1-x)^3 x^{\nfeu+1} + x^3 (1-x)^{\nfeu+1}}{\nfeu+1}
    \\
    & \quad -12 \, \frac{(1-x)^2 x^{\nfeu+2} + x^2 (1-x)^{\nfeu+2}}{(\nfeu+1)(\nfeu+2)} -24 \, \frac{(1-x) x^{\nfeu+3} + x (1-x)^{\nfeu+3}}{(\nfeu+1)(\nfeu+2)(\nfeu+3)}
    \\
    & \quad -24 \, \frac{x^{\nfeu+4} + (1-x)^{\nfeu+4}}{(\nfeu+1)(\nfeu+2)(\nfeu+3)(\nfeu+4)}
  \end{align*}
  which proves Eq.~\eqref{eq.purf.A-B.general}, \eqref{eq.purf.AB.general} and \eqref{eq.purf.Akap+Bkap.general}.

  \paragraph{Upper bounds on remainder terms for every $x \in [0,1)$}
  The bound on $\PolPURFdiff{\nfeu}$ is straightforward. 
  The other bounds follow from Eq.~\eqref{eq.purf.AB.general} and~\eqref{eq.purf.Akap+Bkap.general} and the remark that $P_{\nfeu}(x) \geq 0$, $\E\croch{\alpha_{\nfeu}\beta_{\nfeu}} \geq 0$ and $\E\croch{\alpha_{\nfeu}^{\kappa} + \beta_{\nfeu}^{\kappa}}\geq 0$ for every $\kappa \geq 0$. 

  \paragraph{Upper bounds on remainder terms for every $x \in I_{\nfeu}$}
  First note that $ \nfeu \geq 27$ implies $\epspurf \leq 1/2$, since $x \mapsto \log(x)/x$ is a decreasing function on $[e,+\infty)$, and 
  \begin{gather*}
    \forall j \geq 0 \, , \quad \max\set{ \absj{x}^j \, , \, \absj{1-x}^j } 
    \leq \absj{1 - (\epspurf)^j } = \exp\croch{ j \log \paren{ 1 - \frac{4 \log(\nfeu)}{\nfeu}}} 
    \leq \nfeu^{- 4 j / \nfeu} 
    \\
    \absj{x (1-x)} \leq \frac{1}{4} 
    \qquad 
    \nfeu^{4/\nfeu} = \exp\paren{\frac{4  \log(\nfeu)}{\nfeu}} \leq e^{1/2} 
    \enspace .
  \end{gather*}
  Then, standard computations lead to Eq.~\eqref{eq.purf.A-B.maj.milieu}--\eqref{eq.purf.A4+B4.maj.milieu}. 
\end{proof}

\subsection{Proof of Corollary~\ref{cor.bias.purf.H2}} \label{sec.pr.purf.cor-H2}
The proof directly follows from the combination of Proposition~\ref{pro.bias.multidim.H3} and Proposition~\ref{pro.purf.moments}. 
First, we use Eq.~\eqref{eq.purf.A-B.general}, \eqref{eq.purf.AB.general} and \eqref{eq.purf.Akap+Bkap.general} in Proposition~\ref{pro.purf.moments} to compute the key quantities appearing in the result of Proposition~\ref{pro.bias.multidim.H3} with $\dimX=1$, under assumptions \eqref{hyp.s-2-fois-derivable.alt} and \eqref{hyp.unif}. 
\begin{align*}
  \mathcal{M}_{1,\cU,x} 
  &= \frac{ - s^{\prime}(x) \PolPURFdiff{\nfeu}(x)}{2 (\nfeu+1)}
  \qquad \mbox{so} \quad 
  \absj{\mathcal{M}_{1,\cU,x}} \leq \frac{ \absj{s^{\prime}(x)} }{2 \nfeu}
  \\
  \mathcal{N}_{2,\cU,x} 
  &= \frac{\paren{ s^{\prime}(x) }^2}{4 (\nfeu+1) (\nfeu+2)} \paren{ 2 + \PolPURFdeu{\nfeu}(x) - 2 \PolPURFprod{\nfeu}(x)}
  \leq \frac{\paren{ s^{\prime}(x) }^2}{2 \nfeu^2} 
  \\
  \mathcal{R}_{2,\cU,x} 
  &= \frac{\CHdeuxa}{3 (\nfeu+1) (\nfeu+2)} \paren{ 3 + \PolPURFdeu{\nfeu}(x) - \PolPURFprod{\nfeu}(x)}
  \leq \frac{\CHdeuxa}{\nfeu^2} 
  \\
  \mathcal{R}_{4,\cU,x} 
  &= \frac{2 \CHdeuxa^2}{9 (\nfeu+1) (\nfeu+2) (\nfeu+3) (\nfeu+4)} \paren{48 + \PolPURFqua{\nfeu}(x)}
  \leq \frac{32 \CHdeuxa^2}{3 \nfeu^4} 
  \enspace ,
\end{align*}
where all bounds follow from Eq.~\eqref{eq.purf.maj.restes.general}. 
In particular, 
\begin{align}
  \label{eq.purf.formule-N2-M1^2.gal}
  \mathcal{N}_{2,\cU,x} -  \mathcal{M}_{1,\cU,x}^2 
  &= \frac{\paren{ s^{\prime}(x) }^2}{4 (\nfeu+1) (\nfeu+2)} \croch{  2 + \PolPURFdeu{\nfeu}(x) - 2 \PolPURFprod{\nfeu}(x) - \frac{\nfeu+2}{\nfeu+1}  \paren{ \PolPURFdiff{\nfeu}(x) }^2 }
  \enspace . 
\end{align}

\paragraph{Whatever $x\in[0,1)$,} we deduce that 
\begin{align}
  \label{eq.purf.maj-M1+R2.gal}
  \paren{ \mathcal{M}_{1,\cU,x} + \mathcal{R}_{2,\cU,x}}^2 
  \leq \paren{ \frac{ \absj{s^{\prime}(x)} }{2 \nfeu} + \frac{\CHdeuxa}{\nfeu^2}  }^2
  &\leq \frac{ \paren{s^{\prime}(x)}^2 }{2 \nfeu^2} + \frac{2\CHdeuxa^2}{\nfeu^4}
  \\
  \mbox{and} \quad 
  2 \sqrt{ \mathcal{R}_{4,\cU,x} \: \paren{ \mathcal{N}_{2,\cU,x} - \paren{ \mathcal{M}_{1,\cU,x} }^2 } } + \mathcal{R}_{4,\cU,x}
  &\leq 2 \sqrt{ \mathcal{R}_{4,\cU,x} \: \mathcal{N}_{2,\cU,x} } + \mathcal{R}_{4,\cU,x}
  \notag \\  
  &\leq \frac{8 \CHdeuxa \absj{ s^{\prime}(x) }}{\sqrt{3} \nfeu^3} + \frac{32 \CHdeuxa^2}{3 \nfeu^4} 
    \label{eq.purf.maj-R4N2+N2.gal}
  \enspace . 
\end{align}
Hence, by Proposition~\ref{pro.bias.multidim.H3}, Eq.~\eqref{eq.purf.maj-M1+R2.gal} proves Eq.~\eqref{eq.bias.purf.H2.Biasinfty.partout}.  
From Eq.~\eqref{eq.purf.formule-N2-M1^2.gal}, since $\PolPURFdeu{\nfeu}(x) - 2 \PolPURFprod{\nfeu}(x) \leq 0$, 
\[
\mathcal{N}_{2,\cU,x} -  \mathcal{M}_{1,\cU,x}^2  \leq \frac{\paren{ s^{\prime}(x) }^2}{2 (\nfeu+1) (\nfeu+2)}
\enspace ,
\]
which proves Eq.~\eqref{eq.bias.purf.H2.Vararbre.partout}, together with Eq.~\eqref{eq.purf.maj-R4N2+N2.gal}.

\paragraph{If $x \in [\epspurf,1-\epspurf]$ and $\nfeu \geq 27$,} we can use Eq.~\eqref{eq.purf.A-B.maj.milieu}--\eqref{eq.purf.A4+B4.maj.milieu} in order to make the bounds more precise: 
\begin{gather} \label{eq.purf.maj-M1+R2.milieu}
  \paren{ \mathcal{M}_{1,\cU,x} + \mathcal{R}_{2,\cU,x}}^2 
  \leq 2 \mathcal{M}_{1,\cU,x}^2 + 2 \mathcal{R}_{2,\cU,x}^2
  \leq \frac{ \paren{ s^{\prime}(x) }^2 }{2 \nfeu^6} 
  + \frac{2 \CHdeuxa^2}{\nfeu^4} 
\end{gather}
Hence, by Proposition~\ref{pro.bias.multidim.H3}, Eq.~\eqref{eq.purf.maj-M1+R2.milieu} proves Eq.~\eqref{eq.bias.purf.H2.Biasinfty.milieu}.  

Furthermore, Eq.~\eqref{eq.purf.formule-N2-M1^2.gal} and Eq.~\eqref{eq.purf.A-B.maj.milieu}--\eqref{eq.purf.A2+B2.maj.milieu} imply that 
\begin{align*}
  &\qquad 
  \absj{ \mathcal{N}_{2,\cU,x} -  \mathcal{M}_{1,\cU,x}^2 - \frac{\paren{ s^{\prime}(x) }^2}{2 (\nfeu+1) (\nfeu+2)} }
  \\
  &\leq 
  \frac{\paren{ s^{\prime}(x) }^2}{4 (\nfeu+1) (\nfeu+2)} 
  \croch{ \absj{ \PolPURFdeu{\nfeu}(x) - 2 \PolPURFprod{\nfeu}(x)} + \frac{\nfeu+2}{\nfeu+1} \paren{ \PolPURFdiff{\nfeu}(x) }^2 } 
  \leq 
  \frac{\paren{ s^{\prime}(x) }^2}{3\nfeu^5} 
\end{align*}
which leads to Eq.~\eqref{eq.bias.purf.H2.Vararbre.milieu} by Proposition~\ref{pro.bias.multidim.H3} and Eq.~\eqref{eq.purf.maj-R4N2+N2.gal}.

\paragraph{Integrated results}
Eq.~\eqref{eq.bias.purf.H2.Biasinfty.integrated} follows from integrating Eq.~\eqref{eq.bias.purf.H2.Biasinfty.partout} over $x \in (0,\epspurf) \cup (1-\epspurf,1)$ and~\eqref{eq.bias.purf.H2.Biasinfty.milieu} over $x \in [\epspurf , 1-\epspurf]$. 
Eq.~\eqref{eq.bias.purf.H2.Biasinfty.integrated-noborder} follows from integrating Eq.~\eqref{eq.bias.purf.H2.Biasinfty.milieu} over $x \in [\epspurf , 1-\epspurf]$. 
Eq.~\eqref{eq.bias.purf.H2.Vararbre.integrated} follows from integrating Eq.~\eqref{eq.bias.purf.H2.Vararbre.partout} over $x \in (0,\epspurf) \cup (1-\epspurf,1)$ and~\eqref{eq.bias.purf.H2.Vararbre.milieu} over $x \in [\epspurf , 1-\epspurf]$. 
\qed

\subsection{Proof of Corollary~\ref{cor.bias.purf.H3}} \label{sec.pr.purf.cor-H3}
The proof directly follows from the combination of Proposition~\ref{pro.bias.multidim.H3} and Proposition~\ref{pro.purf.moments}. 

We again use Proposition~\ref{pro.purf.moments} to compute the key quantities appearing in the result of Proposition~\ref{pro.bias.multidim.H3} with $\dimX=1$, under assumptions \eqref{hyp.s-3-fois-derivable.alt} and \eqref{hyp.unif}. 
\begin{align*}
  \mathcal{M}_{1,\cU,x} 
  &= \frac{ - s^{\prime}(x) \PolPURFdiff{\nfeu}(x)}{2 (\nfeu+1)}
  \qquad \mbox{so} \quad 
  \absj{\mathcal{M}_{1,\cU,x}} \leq \frac{ \absj{s^{\prime}(x)} }{2 \nfeu}
  \\
  \mathcal{M}_{2,\cU,x} 
  &= \frac{ s^{\prime\prime}(x) }{6 (\nfeu+1) (\nfeu+2)} \paren{ 3 + \PolPURFdeu{\nfeu}(x) -  \PolPURFprod{\nfeu}(x)}
  \qquad \mbox{so} \quad 
  \absj{\mathcal{M}_{2,\cU,x}} \leq \frac{\absj{s^{\prime}(x) }}{2 \nfeu^2} 
  \\
  \mathcal{R}_{3,\cU,x} 
  &= \frac{\CHtroisa}{4 (\nfeu+1) (\nfeu+2) (\nfeu+3)} \paren{ 12 + \PolPURFtro{\nfeu}(x) }
  \qquad \mbox{so} \quad
  \absj{\mathcal{R}_{3,\cU,x} } \leq \frac{3 \CHtroisa}{\nfeu^3} 
  \enspace ,
\end{align*}
where all bounds follow from Eq.~\eqref{eq.purf.maj.restes.general}. 
In particular, 
\begin{align}
  \label{eq.purf.formule-M1+M2.gal}
  \paren{\mathcal{M}_{1,\cU,x} +  \mathcal{M}_{2,\cU,x}}^2 - \frac{\paren{ s^{\prime}(x) }^2 \PolPURFdiff{\nfeu}(x)}{4 (\nfeu+1)^2}
  &=  \frac{\paren{s^{\prime\prime}(x)}^2 \paren{ 3+\PolPURFdeu{\nfeu}(x) -  \PolPURFprod{\nfeu}(x)}^2 }{36 (\nfeu+1)^2 (\nfeu+2)^2} \notag \\
  &+ \frac{s^{\prime}(x) s^{\prime\prime}(x) \PolPURFdiff{\nfeu}(x) \paren{ 3+\PolPURFdeu{\nfeu}(x) -  \PolPURFprod{\nfeu}(x)} }{6(\nfeu+1)^2 (\nfeu+2)}
  \enspace . 
\end{align}

\paragraph{If $ x \in (0,1) \backslash \croch{ \epspurf , 1 - \epspurf
  }$ and $\nfeu \geq 27$,} we deduce that 
\begin{align}
  \label{eq.purf.maj-M1+M2.gal}
  \absj{ \paren{\mathcal{M}_{1,\cU,x} +  \mathcal{M}_{2,\cU,x}}^2 - \frac{\paren{ s^{\prime}(x) }^2 \PolPURFdiff{\nfeu}(x)}{4 \nfeu^2} } \leq \frac{\paren{s^{\prime}(x)}^2}{2 \nfeu^3} + \frac{\absj{s^{\prime}(x)}\absj{s^{\prime\prime}(x)}}{2 \nfeu^3} + \frac{\paren{s^{\prime\prime}(x)}^2}{4 \nfeu^4}
\end{align}

\begin{align}
  \label{eq.purf.maj-R3M1+M2+R3.gal}
  \mbox{and} \quad 
  2 \absj{ \mathcal{R}_{3,\cU,x} \paren{ \mathcal{M}_{1,\cU,x} +  \mathcal{M}_{2,\cU,x} } } +  \paren{ \mathcal{R}_{3,\cU,x} }^2
  & \leq \frac{3 \CHtroisa}{\nfeu^4} \paren{ \absj{s^{\prime}(x)} + \frac{\absj{s^{\prime\prime}(x)}}{\nfeu} } + \frac{9 \CHtroisa^2}{\nfeu^6}
  \enspace . 
\end{align}
Hence, by Proposition~\ref{pro.bias.multidim.H3}, Eq.~\eqref{eq.purf.maj-M1+M2.gal} and \eqref{eq.purf.maj-R3M1+M2+R3.gal} prove Eq.~\eqref{eq.bias.purf.H3.Biasinfty.border}.  

\paragraph{If $x \in [\epspurf,1-\epspurf]$ and $\nfeu \geq 27$,} we can use Eq.~\eqref{eq.purf.A-B.maj.milieu}--\eqref{eq.purf.A4+B4.maj.milieu} to get:
\begin{gather} \label{eq.purf.maj-M1+M2.milieu}
  \absj{ \paren{\mathcal{M}_{1,\cU,x} +  \mathcal{M}_{2,\cU,x}}^2 - \frac{\paren{ s^{\prime\prime}(x) }^2 }{4 \nfeu^4} } \leq \frac{\absj{s^{\prime}(x)}\absj{s^{\prime\prime}(x)}}{2 \nfeu^7} + \frac{\paren{s^{\prime\prime}(x)}^2}{3 \nfeu^7} + \frac{\paren{s^{\prime}(x)}^2}{4 \nfeu^{10}}
\end{gather}
\begin{align}
  \label{eq.purf.maj-R3M1+M2+R3.milieu}
  \mbox{and} \quad 
  2 \absj{ \mathcal{R}_{3,\cU,x} \paren{ \mathcal{M}_{1,\cU,x} +  \mathcal{M}_{2,\cU,x} } } +  \paren{ \mathcal{R}_{3,\cU,x} }^2
  & \leq \frac{3 \CHtroisa}{\nfeu^5} \paren{ \absj{s^{\prime\prime}(x)} + \frac{\absj{s^{\prime}(x)}}{\nfeu^3} } + \frac{9 \CHtroisa^2}{\nfeu^6}
  \enspace . 
\end{align}

Hence, by Proposition~\ref{pro.bias.multidim.H3}, Eq.~\eqref{eq.purf.maj-M1+M2.milieu} and \eqref{eq.purf.maj-R3M1+M2+R3.milieu} prove Eq.~\eqref{eq.bias.purf.H3.Biasinfty}.  

\paragraph{Integrated results}
Eq.~\eqref{eq.bias.purf.H3.Biasinfty.integrated-noborder} follows from integrating Eq.~\eqref{eq.bias.purf.H3.Biasinfty} over $x \in [\epspurf , 1-\epspurf]$. 

Eq.~\eqref{eq.bias.purf.H3.Biasinfty.integrated} follows from integrating Eq.~\eqref{eq.bias.purf.H3.Biasinfty.border} over $x \in (0,\epspurf) \cup (1-\epspurf,1)$ and~\eqref{eq.bias.purf.H3.Biasinfty} over $x \in [\epspurf , 1-\epspurf]$ and from the following statement:
\begin{align}
  &\absj{ \int_{(0,\epspurf) \cup (1-\epspurf,1)} \paren{s^{\prime}(x)}^2 \paren{ x^{\nfeu+1} - (1-x)^{\nfeu+1}}^2 \, dx - \frac{\paren{s^{\prime}(0)}^2 + \paren{s^{\prime}(1)}^2}{ 2 \nfeu } } \notag \\
  &\quad \leq \frac{1}{2 \nfeu^2} \croch{ 4 \norm{s^{\prime}}_{\infty}^2 + 2 \norm{s^{\prime}}_{\infty}\norm{s^{\prime\prime}}_{\infty} + \norm{s^{\prime\prime}}_{\infty}^2 }
\end{align}
obtained by a Taylor expansion of $s^{\prime}$ around $0$ and $1$ and direct calculations of integrals.

\qed

\section{Proofs: the (d-dimensional) balanced purely random forest model}  \label{sec.pr.multidim.BPRF}
The main result of Section~\ref{sec.multidim.BPRF},
Corollary~\ref{cor.bias.BPRF.H2}, is implied by
Proposition~\ref{pro.bias.multidim.H3}, where key quantities have been
replaced by their exact values (or upper bounds on them).  As shown in
Section~\ref{sec.multidim}, for every fixed $x \in [0,1)^{\dimX}$, the
key quantities are expectations of functions of the non-negative
random variables $(x_i - A_{i,\bU}(x))_{1 \leq i \leq \dimX}$ and
$(B_{i,\bU}(x) - x_i)_{1 \leq i \leq \dimX}\,$.  So, keeping $x$
fixed, we can focus on these random variables.  From a convenient
formulation of their distribution
(Section~\ref{sec.pr.multidim.BPRF.reformulation}), we will be able
to compute all quantities needed
(Sections~\ref{sec.pr.multidim.BPRF.one-dim}
and~\ref{sec.pr.multidim.BPRF.multi-dim}). Then, we will prove
Corollary~\ref{cor.bias.BPRF.H2} in
Section~\ref{sec.pr.multidim.BPRF.cor-H2}.

\subsection{Equivalent formulation of the model} \label{sec.pr.multidim.BPRF.reformulation}
\begin{proposition} \label{pro.multidim.BPRF.reformulation}
  Let $\dimX \geq 1$, $x \in [0,1)^{\dimX}$, and $\paren{\bU_\prof}_{\prof \in \N}$ be some random sequence distributed according to the BPRF model detailed in Section~\ref{sec.multidim.BPRF}. 
  For every $\prof \in \N\,$, let $I_{\bU_\prof}(x) = \prod_{i=1}^{\dimX} [ A_{i,\bU_\prof}(x) \, , \, B_{i,\bU_\prof}(x) )$ denote the unique element of $\bU_\prof$ to which $x$ belongs, and define 
  \[ \forall i \in \set{1, \ldots, \dimX} \, , \quad \alpha_i^{(\prof)} = x_i - A_{i,\bU_\prof}(x) \quad \mbox{and} \quad \beta_i^{(\prof)} = B_{i,\bU_\prof}(x) - x_i \enspace . \]
  Then, the sequence $\paren{ \paren{\alpha_i^{(\prof)}, \beta_i^{(\prof)}}_{1 \leq i \leq \dimX} }_{ \prof \in \N}$ is distributed as follows: 
  \[ \paren{\alpha_i^{(0)}, \beta_i^{(0)}}_{1 \leq i \leq \dimX} = (x_i , 1 - x_i)_{1 \leq i \leq \dimX} \quad \mbox{a.s.} \]
  and for every $\prof \in \N$, given $\paren{\alpha_i^{(\prof)}, \beta_i^{(\prof)}}_{1 \leq i \leq \dimX}\,$, for every $i \in \set{1, \ldots, \dimX} \,$, 
  \begin{equation*}
    \paren{ \alpha_i^{(\prof+1)}, \beta_i^{(\prof+1)} } = 
    \begin{cases}
      \paren{ \alpha_i^{(\prof)}, \beta_i^{(\prof)} } \quad \mbox{if } i \neq J_{\prof+1} \\
      \paren{ U_{\prof+1} \alpha_i^{(\prof)}, \beta_i^{(\prof)} } \quad \mbox{with probability } \frac{\alpha_i^{(\prof)}}{\alpha_i^{(\prof)} + \beta_i^{(\prof)}} \mbox{ if } i = J_{\prof+1} \\
      \paren{ \alpha_i^{(\prof)}, U_{\prof+1} \beta_i^{(\prof)} } \quad \mbox{with probability } \frac{\beta_i^{(\prof)}}{\alpha_i^{(\prof)} + \beta_i^{(\prof)}} \mbox{ if } i = J_{\prof+1}
    \end{cases}
  \end{equation*}
  where $(J_\prof)_{\prof \geq 1}$ and $(U_\prof)_{\prof \geq 1}$ are two independent sequences of i.i.d. random variables, with $J_\prof \sim \mathcal{U}(\set{1, \ldots, \dimX})$ and $U_\prof \sim \mathcal{U}([0,1])$. 
\end{proposition}
\begin{proof}[Proof of Proposition~\ref{pro.multidim.BPRF.reformulation}]
  By the definition of $\bU_0 = [0,1)^{\dimX}$, we get that for all $i$, $A_{i,\bU_0}(x) = 0$ and $B_{i,\bU_\prof}(x)=1$, hence 
  $\paren{\alpha_i^{(0)}, \beta_i^{(0)}}_{1 \leq i \leq \dimX} = (x_i , 1 - x_i)_{1 \leq i \leq \dimX}$ almost surely. 

  Then, let $\prof \in \N$, and denote by $\lambda_{j(\prof,x),\prof} = I_{\bU_\prof}(x)$ the piece of $\bU_\prof$ to which $x$ belongs; $\lambda_{j(\prof,x),\prof}$  is split into two pieces in $\bU_{\prof+1}$ ---one of them being $I_{\bU_{\prof+1}}(x)\,$--- along the direction $J_{\prof+1} = L_{j(\prof,x),\prof}\,$, at some random position $S_\prof(x) = (1 - Z_{j(\prof,x),\prof} ) A_{L_{j(\prof,x),\prof}} +  Z_{j(\prof,x),\prof} B_{L_{j(\prof,x),\prof}} \,$.  
  So, given $\bU_\prof$ (in particular, given $\paren{\alpha_i^{(\prof)}, \beta_i^{(\prof)}}_{1 \leq i \leq \dimX}$), for all $i \in \set{1, \ldots, \dimX}\,$, several cases can occur. 
  \begin{itemize}
  \item[1.] If $i \neq J_{\prof+1}$, then $(A_{i,\bU_{\prof+1}}(x) , B_{i,\bU_{\prof+1}}(x)) = (A_{i,\bU_\prof}(x) , B_{i,\bU_\prof}(x))$ so that $\paren{ \alpha_i^{(\prof+1)}, \beta_i^{(\prof+1)} } =  \paren{ \alpha_i^{(\prof)}, \beta_i^{(\prof)} }$.
  \item[2.] If $i = J_{\prof+1}$, two sub-cases are possible, depending on the relative position of $x$ and the point where $[ A_{i,\bU_\prof}(x) , B_{i,\bU_\prof}(x))$ is split. 
    \begin{itemize}
    \item[2a.] If the split is on the left side of $x$, i.e., if $S_\prof(x) <x$, then 
      \[ (A_{i,\bU_{\prof+1}}(x) , B_{i,\bU_{\prof+1}}(x)) = (S_\prof(x) , B_{i,\bU_\prof}(x))  \]
      so that 
      \[ \paren{ \alpha_i^{(\prof+1)}, \beta_i^{(\prof+1)} } =  \paren{ x - S_\prof(x) , \beta_i^{(\prof)} } \enspace . \]
    \item[2b.]  If the split is on the right side of $x$, i.e., if $S_\prof(x) \geq x$, then 
      \[ (A_{i,\bU_{\prof+1}}(x) , B_{i,\bU_{\prof+1}}(x)) = (A_{i,\bU_\prof}(x) , S_\prof(x))  \]
      so that 
      \[ \paren{ \alpha_i^{(\prof+1)}, \beta_i^{(\prof+1)} } =  \paren{ \alpha_i^{(\prof)}, S_\prof(x) - x } \enspace . \]
    \end{itemize}
  \end{itemize}

  To finish the proof, we remark that given that $i = J_{\prof+1}$, the sub-case 2a has probability $\frac{\alpha_i^{(\prof)}} {\alpha_i^{(\prof)} + \beta_i^{(\prof)}}$  and the sub-case 2b has probability $\frac{\beta_i^{(\prof)}}{\alpha_i^{(\prof)} + \beta_i^{(\prof)}}\,$. 
  Furthermore, given that sub-case 2a holds, the split is chosen uniformly in $[ A_{i,\bU_\prof}(x) , x)$, so that $\alpha_i^{(\prof+1)} = x - S_\prof(x) $ is equal to $\alpha_i^{(\prof)}$ multiplied by a uniform random variable, which defines $U_{\prof+1}$. 
  Similarly, given that sub-case 2b holds, the split is chosen uniformly in $[ x, B_{i,\bU_\prof}(x) )$, so that $\beta_i^{(\prof+1)} = S_\prof(x) - x$ is equal to $\beta_i^{(\prof)}$ multiplied by a uniform random variable, which defines $U_{\prof+1}$. 

  Since the random variables $( L_{j,\prof} , Z_{j,\prof} )_{1 \leq j \leq 2^{\prof} \, , \, \prof \in \N}$ are all independent, so are $(J_\prof , U_\prof)_{\prof \geq 1}$, and $J_\prof \sim \mathcal{U}(\set{1,\ldots, \dimX})$ as all the $L_{j,\prof}$. 
\end{proof}

\subsection{One-dimensional quantities} \label{sec.pr.multidim.BPRF.one-dim}
We start by computing the quantities depending only on one direction $i \in \set{1, \ldots, \dimX}$, i.e., of the form 
\[ \E\croch{ \paren{ \alpha_i^{(\prof)} }^{\delta} \paren{ \beta_i^{(\prof)} }^{\eta} } \]
for some $(\delta, \eta) \in \N^2$. 

\begin{proposition} \label{pro.multidim.BPRF.one-dim}
  With the notation of Proposition~\ref{pro.multidim.BPRF.reformulation}, for every $\prof \in \N$ and $i \in \set{1, \ldots, \dimX}\,$, 
  \begin{align} 
    \label{eq.BPRF.mom-one-dim.del-del}
    \forall \delta \in \N \, , \quad 
    \E\croch{ \paren{ \alpha_i^{(\prof)} \beta_i^{(\prof)} }^{\delta}  } 
    &= \paren{1 - \frac{\delta}{\dimX (\delta + 1)} }^{\prof} \paren{x_i (1-x_i)}^{\delta}
    \\ \label{eq.BPRF.mom-one-dim.diff-1}
    \E\croch{ \alpha_i^{(\prof)}  - \beta_i^{(\prof)} } 
    &= \paren{1 - \frac{1}{2 \dimX} }^{\prof} \paren{ 2 x_i - 1}
    \\ \label{eq.BPRF.mom-one-dim.sum-2}
    \E\croch{ \paren{\alpha_i^{(\prof)}}^2 + \paren{ \beta_i^{(\prof)} }^2 } 
    &= 4 x_i (1-x_i) \paren{1 - \frac{1}{2 \dimX} }^{\prof} 
    \\ \notag &\hspace{-1cm}+ \paren{x_i^2 + (1-x_i)^2 - 4 x_i (1-x_i)} \paren{1 - \frac{2}{3 \dimX} }^{\prof}
    \\ \label{eq.BPRF.mom-one-dim.sum-4}
    \E\croch{ \paren{\alpha_i^{(\prof)}}^4 + \paren{ \beta_i^{(\prof)} }^4 } 
    &= 12 x_i^2 (1-x_i)^2 \paren{1 - \frac{2}{3 \dimX} }^{\prof}  
    \\ \notag &\hspace{-1cm}+ 16 \croch{ x_i (1-x_i) \paren{x_i^2 + (1-x_i)^2} - 3 x_i^2 (1-x_i)^2} \paren{1 - \frac{3}{4 \dimX} }^{\prof}
    \\ \notag &\hspace{-4cm}+ \croch{ x_i^4 + (1-x_i)^4 + 36 x_i^2 (1-x_i)^2  - 16 x_i (1-x_i) \paren{x_i^2 + (1-x_i)^2}  } \paren{1 - \frac{4}{5 \dimX} }^{\prof}
    \\ \label{eq.BPRF.mom-one-dim.sum-4.maj}
    &\leq 8 \paren{1 - \frac{2}{3 \dimX} }^{\prof}
    \\ \label{eq.BPRF.mom-one-dim.sum-3}
    \E\croch{ \paren{\alpha_i^{(\prof)}}^3 + \paren{ \beta_i^{(\prof)} }^3 } 
    &\leq 2^{5/2} \paren{1 - \frac{2}{3 \dimX} }^{3\prof/4} \ll  \paren{1 - \frac{1}{2 \dimX} }^{\prof}
    \quad \mbox{as } \prof \to +\infty \enspace .
  \end{align}
\end{proposition}
The main results of Proposition~\ref{pro.multidim.BPRF.one-dim} are summarized in Table~\ref{tab.pro.multidim.BPRF.one-dim}. 
\begin{table}
  \begin{center}
    \begin{tabular}{|l|c|c|}
      \hline
      Quantity & Order of magnitude & Eq. number \\
      \hline 
      $\E\croch{ \alpha_i^{(\prof)} \beta_i^{(\prof)}   } \phantom{\Bigg|}$  & $\paren{1 - \frac{1}{2 \dimX} }^{\prof} $ & \eqref{eq.BPRF.mom-one-dim.del-del} \\
      $\E\croch{ \alpha_i^{(\prof)}  - \beta_i^{(\prof)} } \phantom{\Bigg|}$ & $\paren{1 - \frac{1}{2 \dimX} }^{\prof} $ & \eqref{eq.BPRF.mom-one-dim.diff-1} \\
      $\E\croch{ \paren{\alpha_i^{(\prof)}}^2 + \paren{ \beta_i^{(\prof)} }^2 } \phantom{\Bigg|}$ & $\paren{1 - \frac{1}{2 \dimX} }^{\prof} $ & \eqref{eq.BPRF.mom-one-dim.sum-2} \\
      \hline
      $\E\croch{ \paren{\alpha_i^{(\prof)}}^4 + \paren{ \beta_i^{(\prof)} }^4  } \phantom{\Bigg|}$ & $\paren{1 - \frac{2}{3 \dimX} }^{\prof} $ & \eqref{eq.BPRF.mom-one-dim.sum-4} \\
      $\E\croch{ \paren{\alpha_i^{(\prof)}}^3 + \paren{ \beta_i^{(\prof)} }^3   } \phantom{\Bigg|}$ & $\ll \paren{1 - \frac{1}{2 \dimX} }^{\prof} $ & \eqref{eq.BPRF.mom-one-dim.sum-3} \\
      \hline
      $\E\croch{ \paren{ \alpha_i^{(\prof)}  - \beta_i^{(\prof)} } \paren{ \alpha_j^{(\prof)} - \beta_j^{(\prof)} } } \phantom{\Bigg|}$ & $\paren{1 - \frac{1}{\dimX} }^{\prof} $ & \eqref{eq.BPRF.mom-multi-dim.prod-diff-1} \\
      \hline
    \end{tabular}
  \end{center}
  \caption{Summary of the results proved by
    Propositions~\ref{pro.multidim.BPRF.one-dim}
    and~\ref{pro.multidim.BPRF.multi-dim} for the
    $\dimX$-dimensional BPRF model. \label{tab.pro.multidim.BPRF.one-dim}}
\end{table}

\begin{proof}[Proof of Proposition~\ref{pro.multidim.BPRF.one-dim}] 
  We start the proof by a general formula that will be used repeatedly. By Proposition~\ref{pro.multidim.BPRF.reformulation}, for every $\prof, \delta, \eta \in \N$, 
  \begin{align} \notag 
    &\qquad \E\croch{ \paren{ \alpha_i^{(\prof+1)} }^{\delta} \paren{ \beta_i^{(\prof+1)} }^{\eta} \sachant \paren{\alpha_{\ell}^{(\prof)},\beta_{\ell}^{(\prof)}}_{1 \leq \ell \leq \dimX}  } 
    \\ \notag 
    &= \paren{ \alpha_i^{(\prof)} }^{\delta} \paren{ \beta_i^{(\prof)} }^{\eta} \croch{  1 - \frac{1}{\dimX} + \frac{1}{\dimX} \frac{\alpha_i^{(\prof)} \E\croch{ U_{\prof+1}^{\delta} } + \beta_i^{(\prof)} \E\croch{ U_{\prof+1}^{\eta} } }{\alpha_i^{(\prof)} + \beta_i^{(\prof)}} } 
    \\
    &= \paren{ \alpha_i^{(\prof)} }^{\delta} \paren{ \beta_i^{(\prof)} }^{\eta} \croch{  1 - \frac{1}{\dimX} + \frac{1}{\dimX} \frac{\alpha_i^{(\prof)} (\delta+1)^{-1} + \beta_i^{(\prof)} (\eta+1)^{-1} } {\alpha_i^{(\prof)} + \beta_i^{(\prof)}} }  
    \label{eq.pr.pro.multidim.BPRF.one-dim.base}
  \end{align}
  since for every $ t \geq 0$, $E\scroch{U_{\prof+1}^t} = (t+1)^{-1}\,$. 
  \paragraph{Proof of Eq.~\eqref{eq.BPRF.mom-one-dim.del-del}} 
  By Eq.~\eqref{eq.pr.pro.multidim.BPRF.one-dim.base} with $\delta = \eta$, for every $\prof \in \N$, 
  \begin{align*}
    \E\croch{ \paren{ \alpha_i^{(\prof+1)} \beta_i^{(\prof+1)} }^{\delta} \sachant \paren{\alpha_{\ell}^{(\prof)},\beta_{\ell}^{(\prof)}}_{1 \leq \ell \leq \dimX}  } 
    &= 
    \paren{ \alpha_i^{(\prof+1)} \beta_i^{(\prof+1)} }^{\delta} \croch{  1 - \frac{\delta}{\dimX (\delta + 1)}  } 
    \enspace ,
  \end{align*}
  so that 
  \[ \E\croch{ \paren{ \alpha_i^{(\prof+1)} \beta_i^{(\prof+1)} }^{\delta} } = \E\croch{ \paren{ \alpha_i^{(\prof+1)} \beta_i^{(\prof+1)} }^{\delta} } \croch{  1 - \frac{\delta}{\dimX (\delta + 1)}  }   \]
  which implies Eq.~\eqref{eq.pr.pro.multidim.BPRF.one-dim.base} since 
  $\paren{ \alpha_i^{(0)} \beta_i^{(0)} }^{\delta} = (x_i (1-x_i))^{\delta}\,$. 

  \paragraph{Proof of Eq.~\eqref{eq.BPRF.mom-one-dim.diff-1}}
  Let $\varepsilon \geq 0$. By Eq.~\eqref{eq.pr.pro.multidim.BPRF.one-dim.base} with $(\delta , \eta) \in \set{(\varepsilon,0) , (0,\varepsilon)}\,$, for every $\prof \in \N$, 
  \begin{equation*} 
    \begin{split}
      & \qquad \E\croch{ \paren{ \alpha_i^{(\prof+1)} }^{\varepsilon}  - \paren{ \beta_i^{(\prof+1)} }^{\varepsilon} \sachant \paren{\alpha_{\ell}^{(\prof)},\beta_{\ell}^{(\prof)}}_{1 \leq \ell \leq \dimX}  } 
      \\
      &= \paren{ \alpha_i^{(\prof)} }^{\varepsilon}  - \paren{ \beta_i^{(\prof)} }^{\varepsilon}
      - \frac{\varepsilon}{\dimX (\varepsilon + 1)} \frac{ \paren{ \alpha_i^{(\prof)} }^{\varepsilon+1}  - \paren{ \beta_i^{(\prof)} }^{\varepsilon+1} }{\alpha_i^{(\prof)} + \beta_i^{(\prof)} } 
      \enspace ,
    \end{split}
  \end{equation*}
  Hence, if $\varepsilon$ is an odd integer, 
  \begin{equation} \label{eq.pr.pro.multidim.BPRF.one-dim.diff-puiss}
    \begin{split}
      \E\croch{ \paren{ \alpha_i^{(\prof+1)} }^{\varepsilon}  - \paren{ \beta_i^{(\prof+1)} }^{\varepsilon} } 
      = \paren{ 1 - \frac{\varepsilon}{\dimX (\varepsilon + 1)}  } \E\croch{ \paren{ \alpha_i^{(\prof)} }^{\varepsilon}  - \paren{ \beta_i^{(\prof)} }^{\varepsilon} } 
      \\  
      + \frac{\varepsilon}{\dimX (\varepsilon + 1)} 
      \sum_{1 \leq j \leq \varepsilon-1} \paren{ (-1)^{j} \E \croch{ \paren{\alpha_i^{(\prof)}}^j \paren{\beta_i^{(\prof)}}^{\varepsilon-j}  } }
    \end{split}
  \end{equation}

  In particular, taking $\varepsilon = 1$ in Eq.~\eqref{eq.pr.pro.multidim.BPRF.one-dim.diff-puiss} yields 
  \begin{equation*} 
    \E\croch{ \alpha_i^{(\prof+1)}   - \beta_i^{(\prof+1)}  } 
    = \paren{ 1 - \frac{1}{2 \dimX}  } \E\croch{ \alpha_i^{(\prof)}   -  \beta_i^{(\prof)} } 
  \end{equation*}
  and we get Eq.~\eqref{eq.BPRF.mom-one-dim.diff-1}
  since $ \alpha_i^{(0)}   - \beta_i^{(0)} =  2 x_i - 1$. 

  \paragraph{Proof of Eq.~\eqref{eq.BPRF.mom-one-dim.sum-2}}
  Let $\varepsilon \geq 0$. By Eq.~\eqref{eq.pr.pro.multidim.BPRF.one-dim.base} with $(\delta , \eta) \in \set{(\varepsilon,0) , (0,\varepsilon)}\,$, for every $\prof \in \N$, 
  \begin{equation*} 
    \begin{split}
      & \qquad \E\croch{ \paren{ \alpha_i^{(\prof+1)} }^{\varepsilon}  + \paren{ \beta_i^{(\prof+1)} }^{\varepsilon} \sachant \paren{\alpha_{\ell}^{(\prof)},\beta_{\ell}^{(\prof)}}_{1 \leq \ell \leq \dimX}  } 
      \\
      &= \paren{ \alpha_i^{(\prof)} }^{\varepsilon}  + \paren{ \beta_i^{(\prof)} }^{\varepsilon}
      - \frac{\varepsilon}{\dimX (\varepsilon + 1)} \frac{ \paren{ \alpha_i^{(\prof)} }^{\varepsilon+1}  + \paren{ \beta_i^{(\prof)} }^{\varepsilon+1} }{\alpha_i^{(\prof)} + \beta_i^{(\prof)} } 
      \enspace ,
    \end{split}
  \end{equation*}
  Hence, if $\varepsilon$ is an even integer, 
  \begin{equation} \label{eq.pr.pro.multidim.BPRF.one-dim.sum-puiss}
    \begin{split}
      \E\croch{ \paren{ \alpha_i^{(\prof+1)} }^{\varepsilon}  + \paren{ \beta_i^{(\prof+1)} }^{\varepsilon} } 
      = \paren{ 1 - \frac{\varepsilon}{\dimX (\varepsilon + 1)}  } \E\croch{ \paren{ \alpha_i^{(\prof)} }^{\varepsilon}  + \paren{ \beta_i^{(\prof)} }^{\varepsilon} } 
      \\  
      + \frac{\varepsilon}{\dimX (\varepsilon + 1)} 
      \sum_{1 \leq j \leq \varepsilon-1} \paren{ (-1)^{j-1} \E \croch{ \paren{\alpha_i^{(\prof)}}^j \paren{\beta_i^{(\prof)}}^{\varepsilon-j}  } }
    \end{split}
  \end{equation}

  In particular, taking $\varepsilon = 2$ in Eq.~\eqref{eq.pr.pro.multidim.BPRF.one-dim.sum-puiss} yields 
  \begin{align*} 
    &\qquad \E\croch{ \paren{ \alpha_i^{(\prof+1)} }^2   +  \paren{ \beta_i^{(\prof+1)} }^2  } 
    \\
    &= \paren{ 1 - \frac{2}{3 \dimX}  } \E\croch{ \paren{ \alpha_i^{(\prof)}  }^2   +  \paren{ \beta_i^{(\prof)} }^2 } 
    + \frac{2}{3 \dimX} \E \croch{ \alpha_i^{(\prof)} \beta_i^{(\prof)}  }
    \\
    &= \paren{ 1 - \frac{2}{3 \dimX}  } \E\croch{ \paren{ \alpha_i^{(\prof)}  }^2   +  \paren{ \beta_i^{(\prof)} }^2 } 
    + \frac{2}{3 \dimX} \paren{1 - \frac{1}{2 \dimX}}^{\prof} x_i (1-x_i)
  \end{align*}
  by Eq~\eqref{eq.BPRF.mom-one-dim.del-del} with $\delta = 1$. 
  Therefore, applying Lemma~\ref{le.suites-geom} to the sequence $\sparen{  \E\scroch{ \sparen{ \alpha_i^{(n)}  }^2   +  \sparen{ \beta_i^{(n)} }^2 }  }_{n \in \N}$ proves Eq.~\eqref{eq.BPRF.mom-one-dim.sum-2} since $\sparen{ \alpha_i^{(0)}  }^2   +  \sparen{ \beta_i^{(0)} }^2 = x_i^2 + (1-x_i)^2\,$.

  \paragraph{Proof of an additional formula}
  Before proving Eq.~\eqref{eq.BPRF.mom-one-dim.sum-4} and Eq.~\eqref{eq.BPRF.mom-one-dim.sum-4.maj}, we prove the following additional formula: 
  \begin{align}
    \label{eq.BPRF.mom-one-dim.prod-1-fois-sum-carre}
    \E \croch{ \alpha_i^{(\prof)} \beta_i^{(\prof)} \paren{ \paren{\alpha_i^{(\prof)}}^2 +  \paren{\beta_i^{(\prof)}}^2  } } 
    &= 3 x_i^2 (1-x_i)^2 \paren{1 - \frac{2}{3 \dimX} }^{\prof} 
    \\ \notag &\hspace{-1cm}+ \croch{ x_i (1-x_i) \paren{x_i^2 + (1-x_i)^2} - 3 x_i^2 (1-x_i)^2} \paren{1 - \frac{3}{4 \dimX} }^{\prof}
    \enspace . 
  \end{align}
  By Eq.~\eqref{eq.pr.pro.multidim.BPRF.one-dim.base} with $(\delta , \eta) \in \set{(3,1) , (1,3)}\,$, for every $\prof \in \N$, 
  \begin{align*} 
    & \qquad \E\croch{ \alpha_i^{(\prof+1)} \beta_i^{(\prof+1)} \paren{ \paren{\alpha_i^{(\prof+1)}}^2 + \paren{\beta_i^{(\prof+1)}}^2 }  \sachant \paren{\alpha_{\ell}^{(\prof)},\beta_{\ell}^{(\prof)}}_{1 \leq \ell \leq \dimX}  } 
    \\
    &= \paren{ 1 - \frac{1}{\dimX} }  \alpha_i^{(\prof)} \beta_i^{(\prof)} \paren{ \paren{\alpha_i^{(\prof)}}^2 + \paren{\beta_i^{(\prof)}}^2 } 
    +   \frac{1}{4 \dimX} \alpha_i^{(\prof)} \beta_i^{(\prof)} \frac{\paren{\alpha_i^{(\prof)}}^3 + \paren{ \beta_i^{(\prof)}}^3 }{\alpha_i^{(\prof)} + \beta_i^{(\prof)}}  
    +   \frac{1}{2 \dimX} \paren{ \alpha_i^{(\prof)} \beta_i^{(\prof)}}^2 
    \\
    &= \paren{ 1 - \frac{3}{4 \dimX} }  \alpha_i^{(\prof)} \beta_i^{(\prof)} \paren{ \paren{\alpha_i^{(\prof)}}^2 + \paren{\beta_i^{(\prof)}}^2 } 
    +   \frac{1}{4 \dimX} \paren{ \alpha_i^{(\prof)} \beta_i^{(\prof)}}^2 
    \enspace ,
  \end{align*}
  so that 
  \begin{align*} 
    & \qquad \E\croch{ \alpha_i^{(\prof+1)} \beta_i^{(\prof+1)} \paren{ \paren{\alpha_i^{(\prof+1)}}^2 + \paren{\beta_i^{(\prof+1)}}^2 }  } 
    \\
    &= \paren{ 1 - \frac{3}{4 \dimX} }  \E\croch{ \alpha_i^{(\prof)} \beta_i^{(\prof)} \paren{ \paren{\alpha_i^{(\prof)}}^2 + \paren{\beta_i^{(\prof)}}^2 } }
    +   \frac{1}{4 \dimX} x_i^2 (1-x_i)^2 \paren{1 - \frac{2}{3 \dimX}}^{\prof} 
  \end{align*}
  by Eq.~\eqref{eq.BPRF.mom-one-dim.del-del} with $\delta = 2$.  
  Therefore, since $\alpha_i^{(0)} \beta_i^{(0)} \sparen{ \sparen{\alpha_i^{(0)}}^2 + \sparen{\beta_i^{(0)}}^2 } = x_i (1-x_i) \sparen{x_i^2 + (1-x_i)^2 }\,$, applying Lemma~\ref{le.suites-geom} to the sequence $\sparen{  \E\scroch{ \alpha_i^{(n)} \beta_i^{(n)} \sparen{ \sparen{\alpha_i^{(n)}}^2 + \sparen{\beta_i^{(n)}}^2 } }  }_{n \in \N}$ proves Eq.~\eqref{eq.BPRF.mom-one-dim.prod-1-fois-sum-carre}.

  \paragraph{Proof of Eq.~\eqref{eq.BPRF.mom-one-dim.sum-4} and Eq.~\eqref{eq.BPRF.mom-one-dim.sum-4.maj}}
  Taking $\varepsilon = 4$ in Eq.~\eqref{eq.pr.pro.multidim.BPRF.one-dim.sum-puiss} yields 
  \begin{align*} 
    &\qquad \E\croch{ \paren{ \alpha_i^{(\prof+1)} }^4   +  \paren{ \beta_i^{(\prof+1)} }^4  } 
    \\
    &= \paren{ 1 - \frac{4}{5 \dimX}  } \E\croch{ \paren{ \alpha_i^{(\prof)} }^{4}  + \paren{ \beta_i^{(\prof)} }^{4} } 
    \\  
    &+ \frac{4}{5 \dimX } \E \croch{ \alpha_i^{(\prof)} \paren{\beta_i^{(\prof)}}^{3} - \paren{\alpha_i^{(\prof)}}^2 \paren{\beta_i^{(\prof)}}^2 + \paren{\alpha_i^{(\prof)}}^3 \beta_i^{(\prof)} } 
    \\
    &= \paren{ 1 - \frac{4}{5 \dimX}  } \E\croch{ \paren{ \alpha_i^{(\prof)} }^{4}  + \paren{ \beta_i^{(\prof)} }^{4} } 
    + \frac{8}{5 \dimX } x_i^2 (1-x_i)^2 \paren{1 - \frac{2}{3 \dimX}}^{\prof} 
    \\
    &+ \frac{4}{5 d }  \croch{ x_i (1-x_i) \paren{x_i^2 + (1-x_i)^2} - 3 x_i^2 (1-x_i)^2} \paren{1 - \frac{3}{4 \dimX} }^{\prof}
  \end{align*} 
  by Eq.~\eqref{eq.BPRF.mom-one-dim.prod-1-fois-sum-carre} and Eq.~\eqref{eq.BPRF.mom-one-dim.del-del} with $\delta = 2$. 
  Since $\sparen{ \alpha_i^{(0)} }^4   +  \sparen{ \beta_i^{(0)} }^4 = x_i^4 + (1-x_i)^4\,$, applying Lemma~\ref{le.suites-geom} to the sequence $\sparen{ \sparen{ \alpha_i^{(n)} }^4   +  \sparen{ \beta_i^{(n)} }^4 }_{n \in \N}$ proves Eq.~\eqref{eq.BPRF.mom-one-dim.sum-4}. 

  Finally, using that $\forall x_i \in [0,1)\,$, $0 \leq x_i (1-x_i) \leq 1/4\,$, Eq.~\eqref{eq.BPRF.mom-one-dim.sum-4} yields, for every $x \in [0,1)^{\dimX}\,$, 
  \begin{align*}
    \E\croch{ \paren{ \alpha_i^{(\prof)} }^{4}  + \paren{ \beta_i^{(\prof)} }^{4} } 
    &\leq  \frac{3}{4} \paren{1 - \frac{2}{3 \dimX}}^{\prof} + 4 \paren{1 - \frac{3}{4 \dimX}}^{\prof} + \frac{13}{4} \paren{1 - \frac{4}{5 \dimX}}^{\prof}
    \\
    &\leq 8 \paren{1 - \frac{2}{3 \dimX}}^{\prof}
    \enspace ,
  \end{align*}
  which proves Eq.~\eqref{eq.BPRF.mom-one-dim.sum-4.maj}.

  \paragraph{Proof of Eq.~\eqref{eq.BPRF.mom-one-dim.sum-3}}
  By Jensen's inequality, $\forall a,b \geq 0\,$, $a + b \leq 2^{1/4} \sparen{a^{4/3} + b^{4/3}}^{3/4}\,$. 
  Since $\alpha_i^{(\prof)}, \beta_i^{(\prof)} \geq 0$ a.s., we get 
  \begin{align} \notag 
    \E\croch{ \paren{ \alpha_i^{(\prof)} }^{3}  + \paren{ \beta_i^{(\prof)} }^{3} } 
    &\leq 
    2^{1/4} \E\croch{ \paren{\paren{ \alpha_i^{(\prof)} }^{4}  + \paren{ \beta_i^{(\prof)} }^{4}  }^{3/4} }
    \\ \label{pr.eq.BPRF.mom-one-dim.sum-3.Jensen}
    &\leq 
    2^{1/4} \paren{\E\croch{ \paren{ \alpha_i^{(\prof)} }^{4}  + \paren{ \beta_i^{(\prof)} }^{4} } }^{3/4} \enspace . 
  \end{align}
  Then, combining Eq.~\eqref{eq.BPRF.mom-one-dim.sum-4.maj} and Eq.~\eqref{pr.eq.BPRF.mom-one-dim.sum-3.Jensen} leads to the first part of Eq.~\eqref{eq.BPRF.mom-one-dim.sum-3}.
  For the second part, we only have to prove that for every $\dimX \geq 1\,$, $\paren{1 - \frac{2}{3 \dimX}}^{3/4} < 1 - \frac{1}{2 \dimX}\,$, which is equivalent to 
  \[  1 - \frac{2}{\dimX} + \frac{4}{3 \dimX^2} - \frac{8}{27 \dimX^3} = \paren{1 - \frac{2}{3 \dimX}}^3 < \paren{ 1 - \frac{1}{2 \dimX} }^4 = 1 - \frac{2}{\dimX} + \frac{3}{2 \dimX^2} - \frac{1}{2 \dimX^3} + \frac{1}{16 \dimX^4} \enspace , \]
  that is 
  \[ \frac{\dimX^2}{6} + \paren{ \frac{8}{27} - \frac{1}{2} } \dimX + \frac{1}{16} > 0 \enspace , \] 
  which holds true since the polynomial on the left-hand side is increasing on $[33/54, +\infty)$ and positive for $\dimX=1\,$. 
\end{proof}

\subsection{Bi-dimensional quantities} \label{sec.pr.multidim.BPRF.multi-dim}
\begin{proposition} \label{pro.multidim.BPRF.multi-dim}
  With the notation of Proposition~\ref{pro.multidim.BPRF.reformulation}, for every $\prof \in \N$ and $i \neq j \in \set{1, \ldots, \dimX}\,$, 
  \begin{align} 
    \label{eq.BPRF.mom-multi-dim.prod-diff-1}
    \E\croch{ \paren{ \alpha_i^{(\prof)}  - \beta_i^{(\prof)} } \paren{ \alpha_j^{(\prof)}  - \beta_j^{(\prof)} } } &= \paren{1 - \frac{1}{\dimX} }^{\prof} \paren{2 x_i - 1} \paren{2 x_j -1}
  \end{align} 
\end{proposition}
\begin{proof}[Proof of Proposition~\ref{pro.multidim.BPRF.multi-dim}]
By Proposition~\ref{pro.multidim.BPRF.reformulation}, 
  \begin{align} \notag 
    &\qquad \E\croch{ \paren{ \alpha_i^{(\prof+1)}  - \beta_i^{(\prof+1)} } \paren{ \alpha_j^{(\prof+1)}  - \beta_j^{(\prof+1)} }  \sachant \paren{\alpha_{\ell}^{(\prof)},\beta_{\ell}^{(\prof)}}_{1 \leq \ell \leq \dimX}  } 
    \\ \notag 
    &= \paren{1 - \frac{2}{\dimX} } \paren{ \alpha_i^{(\prof)}  - \beta_i^{(\prof)} } \paren{ \alpha_j^{(\prof)}  - \beta_j^{(\prof)} }
    \\ \notag 
    &\hspace{-1cm}+ \frac{1}{\dimX} \paren{ \alpha_j^{(\prof)}  - \beta_j^{(\prof)} } \frac{1}{ \alpha_i^{(\prof)}  + \beta_i^{(\prof)} } \croch{ \frac{ \paren{\alpha_i^{(\prof)}}^2 - \paren{\beta_i^{(\prof)}}^2 }{2} + \beta_i^{(\prof)} \alpha_i^{(\prof)} - \beta_i^{(\prof)} \alpha_i^{(\prof)} }
    \\ \notag 
    &\hspace{-1cm}+ \frac{1}{\dimX} \paren{ \alpha_i^{(\prof)}  - \beta_i^{(\prof)} } \frac{1}{ \alpha_j^{(\prof)}  + \beta_j^{(\prof)} } \croch{ \frac{ \paren{\alpha_j^{(\prof)}}^2 - \paren{\beta_j^{(\prof)}}^2 }{2} + \beta_j^{(\prof)} \alpha_j^{(\prof)} - \beta_j^{(\prof)} \alpha_j^{(\prof)} }
    \\ \notag 
    &= \paren{1 - \frac{1}{\dimX} } \paren{ \alpha_i^{(\prof)}  - \beta_i^{(\prof)} } \paren{ \alpha_j^{(\prof)}  - \beta_j^{(\prof)} } 
  \end{align} 
  so that 
  \[ 
  \E\croch{  \paren{ \alpha_i^{(\prof+1)}  - \beta_i^{(\prof+1)} } \paren{ \alpha_j^{(\prof+1)}  - \beta_j^{(\prof+1)} } } = 
  \paren{1 - \frac{1}{\dimX} } \E\croch{ \paren{ \alpha_i^{(\prof)}  - \beta_i^{(\prof)} } \paren{ \alpha_j^{(\prof)}  - \beta_j^{(\prof)} } }
  \]
  which implies Eq.~\eqref{eq.BPRF.mom-multi-dim.prod-diff-1} since 
  $ \paren{ \alpha_i^{(0)}  - \beta_i^{(0)} } \paren{ \alpha_j^{(0)}  - \beta_j^{(0)} }  = (2 x_i - 1) (2 x_j - 1)\,$.

\end{proof}

\subsection{Proof of Corollary~\ref{cor.bias.BPRF.H2}} \label{sec.pr.multidim.BPRF.cor-H2}
The proof directly follows from the combination of Propositions~\ref{pro.bias.multidim.H3}, \ref{pro.multidim.BPRF.one-dim} and \ref{pro.multidim.BPRF.multi-dim}. 

First, using Propositions~\ref{pro.multidim.BPRF.one-dim} and \ref{pro.multidim.BPRF.multi-dim}, we compute the key quantities appearing in the result of Proposition~\ref{pro.bias.multidim.H3} with $\cU = \cUurt{\prof}$, under assumptions \eqref{hyp.s-2-fois-derivable.alt} and \eqref{hyp.unif}. 
\begin{align*}
  \mathcal{M}_{1,\cU,x} &= \frac{1}{2} \paren{1 - \frac{1}{2 \dimX}}^{\prof} \nabla s(x) \cdot (1-2x)
  \\
  &\qquad \absj{ \mathcal{N}_{2,\cU,x} - \frac{1}{2} \paren{1 - \frac{1}{2 \dimX}}^{\prof} \sum_{i=1}^{\dimX} \croch{ \paren{ \frac{\partial s}{\partial x_i} (x) }^2 x_i (1-x_i) } }
  \\
  &\leq \frac{1}{4} \paren{1 - \frac{2}{3 \dimX}}^{\prof} \sum_{i=1}^{\dimX} \absj{ \paren{ \frac{\partial s}{\partial x_i} (x) }^2 \paren{ x_i^2 + (1-x_i)^2 - 4 x_i (1-x_i) }  } 
  \\
  &\qquad + \frac{1}{4} \paren{1 - \frac{1}{\dimX}}^{\prof} \sum_{1 \leq i \neq j \leq \dimX} \absj{ \frac{\partial^2 s}{\partial x_i \partial x_j} (x) (2 x_i - 1) (2 x_j - 1) }
  \\
  &\leq \frac{\dimX}{4} \paren{1 - \frac{2}{3 \dimX}}^{\prof} \max_i \paren{ \frac{\partial s}{\partial x_i} (x) }^2  
  + \frac{\dimX^2 - \dimX}{4} \paren{1 - \frac{1}{\dimX}}^{\prof} \max_{1 \leq i \neq j \leq \dimX} \absj{ \frac{\partial^2 s}{\partial x_i \partial x_j} (x) }
  \\
  \mathcal{R}_{2,\cU,x} 
  &= \CHdeuxa \paren{1 - \frac{1}{2 \dimX}}^{\prof}  \sum_{i=1}^{\dimX} \croch{ x_i (1-x_i) }
  + \frac{\CHdeuxa}{3} \paren{1 - \frac{2}{3 \dimX}}^{\prof} \sum_{i=1}^{\dimX} \croch{ x_i^2 + (1-x_i)^2 - 4 x_i (1-x_i) } 
  \\
  &\leq \frac{\dimX \CHdeuxa}{4} \paren{1 - \frac{1}{2 \dimX}}^{\prof}  + \frac{\dimX \CHdeuxa}{3} \paren{1 - \frac{2}{3 \dimX}}^{\prof} 
  \\
  &\leq \dimX \CHdeuxa \paren{1 - \frac{1}{2 \dimX}}^{\prof}
  \\
  \mathcal{R}_{4,\cU,x} &\leq \frac{16 \dimX \CHdeuxa^2}{9} \paren{1 - \frac{2}{3 \dimX}}^{\prof} 
  \enspace .
\end{align*}
So, Eq.~\eqref{eq.pro.bias.multidim.Bias-H2} yields, for every $x \in [0,1)$, 
\begin{align*}
  \Biasinfty{\cU}(x) &\leq \paren{ \mathcal{M}_{1,\cU,x} + \mathcal{R}_{2,\cU,x} }^2 \leq 2 \mathcal{M}_{1,\cU,x}^2 + 2 \mathcal{R}_{2,\cU,x}^2 
  \\
  &\leq \paren{1 - \frac{1}{2 \dimX}}^{2 \prof} \croch{ \frac{1}{2}  \paren{ \nabla s(x) \cdot (1-2x) }^2  + 2 \dimX^2 \CHdeuxa^2 }
  \\
  &\leq \paren{1 - \frac{1}{2 \dimX}}^{2 \prof} \croch{ \frac{\dimX}{2}  \sup_{x \in [0,1)^{\dimX}} \norm{ \nabla s(x) }_2^2   + 2 \dimX^2 \CHdeuxa^2 }
\end{align*}
which proves Eq.~\eqref{eq.bias.BPRF.H2.Biasinfty}. 
Then, integrating Eq.~\eqref{eq.bias.BPRF.H2.Biasinfty} over $x \in [0,1)^{\dimX}$ yields Eq.~\eqref{eq.bias.BPRF.H2.Biasinfty.integrated}.

Second, Eq.~\eqref{eq.pro.bias.multidim.Var-H2} yields, for every $x \in [0,1)$, 
\begin{align*}
  & \qquad \absj{ \Vararbre{\cU}(x) - \frac{1}{2} \paren{1 - \frac{1}{2 \dimX}}^{\prof} \sum_{i=1}^{\dimX} \croch{ \paren{ \frac{\partial s}{\partial x_i} (x) }^2 x_i (1-x_i) } } 
  \\
  &\leq 
  \frac{1}{4} \paren{1 - \frac{1}{2 \dimX}}^{2\prof} \paren{ \nabla s(x) \cdot (1-2x) }^2 
  + \frac{\dimX}{4} \paren{1 - \frac{2}{3 \dimX}}^{\prof} \max_i \paren{ \frac{\partial s}{\partial x_i} (x) }^2  
  \\ &\qquad 
  + \frac{\dimX^2 - \dimX}{4} \paren{1 - \frac{1}{\dimX}}^{\prof} \max_{1 \leq i \neq j \leq \dimX} \absj{ \frac{\partial^2 s}{\partial x_i \partial x_j} (x) }
  + 2 \sqrt{ \frac{16 \dimX \CHdeuxa^2}{9} \paren{1 - \frac{2}{3 \dimX}}^{\prof} \mathcal{N}_{2,\cU,x} }
  \\ &\qquad 
  + \frac{16 \dimX \CHdeuxa^2}{9} \paren{1 - \frac{2}{3 \dimX}}^{\prof} 
  \\
  &\leq 
  \frac{\dimX}{4} \sup_{x \in [0,1)^{\dimX}} \norm{ \nabla s(x) }_2^2 \paren{1 - \frac{1}{2 \dimX}}^{2\prof}  
  \\
  &\qquad + \left[ 
    \frac{\dimX}{4} \max_i \paren{ \frac{\partial s}{\partial x_i} (x) }^2  
    + \frac{\dimX^2}{4} \max_{1 \leq i \neq j \leq \dimX} \absj{ \frac{\partial^2 s}{\partial x_i \partial x_j} (x) }
  \right. \\  & \qquad \left.
    + \frac{16 \dimX \CHdeuxa^2}{9} 
    + \frac{8 \dimX \CHdeuxa }{3 \sqrt{2}} \croch{   \max_i \absj{ \frac{\partial s}{\partial x_i} (x) } + \sqrt{ \dimX \max_{1 \leq i \neq j \leq \dimX} \absj{ \frac{\partial^2 s}{\partial x_i \partial x_j} (x) }  } }
  \right] \\
  &\qquad \times \sqrt{ \paren{1 - \frac{1}{2 \dimX}}^{\prof} \paren{1 - \frac{2}{3 \dimX}}^{\prof} }
  \\
  &\leq 
  \frac{\dimX}{4} \sup_{x \in [0,1)^{\dimX}} \norm{ \nabla s(x) }_2^2 \paren{1 - \frac{1}{2 \dimX}}^{2\prof}  
  \\
  &\qquad + \croch{  
    \dimX \max_i \paren{ \frac{\partial s}{\partial x_i} (x) }^2  
    + \dimX^2 \max_{1 \leq i \neq j \leq \dimX} \absj{ \frac{\partial^2 s}{\partial x_i \partial x_j} (x) }
    + 5 \dimX \CHdeuxa^2 } \\
  &\qquad \times \sqrt{ \paren{1 - \frac{1}{2 \dimX}}^{\prof} \paren{1 - \frac{2}{3 \dimX}}^{\prof} }
\end{align*}
since 
\begin{align*}
  \mathcal{N}_{2,\cU,x} 
  &\leq \frac{1}{2} \paren{1 - \frac{1}{2 \dimX}}^{\prof} \sum_{i=1}^{\dimX} \croch{ \paren{ \frac{\partial s}{\partial x_i} (x) }^2 x_i (1-x_i) } 
  \\
  &\qquad + \frac{\dimX}{4} \paren{1 - \frac{2}{3 \dimX}}^{\prof} \max_i \paren{ \frac{\partial s}{\partial x_i} (x) }^2  
  + \frac{\dimX^2 - \dimX}{4} \paren{1 - \frac{1}{\dimX}}^{\prof} \max_{1 \leq i \neq j \leq \dimX} \absj{ \frac{\partial^2 s}{\partial x_i \partial x_j} (x) }
  \\
  &\leq \frac{\dimX}{2}  \paren{1 - \frac{1}{2 \dimX}}^{\prof} \croch{ \max_i \paren{ \frac{\partial s}{\partial x_i} (x) }^2   
    + \dimX \max_{1 \leq i \neq j \leq \dimX} \absj{ \frac{\partial^2 s}{\partial x_i \partial x_j} (x) } }
  \enspace , 
\end{align*}
which proves Eq.~\eqref{eq.bias.BPRF.H2.Vararbre}. 
Then, integrating Eq.~\eqref{eq.bias.BPRF.H2.Vararbre} over $x \in [0,1)^{\dimX}$ yields Eq.~\eqref{eq.bias.BPRF.H2.Vararbre.integrated}.
\qed

\subsection{Proof of Corollary~\ref{cor.bias.BPRF.H3}} \label{sec.pr.multidim.BPRF.cor-H3}
Using Propositions~\ref{pro.multidim.BPRF.one-dim} and \ref{pro.multidim.BPRF.multi-dim}, we compute the quantities appearing in the result of Proposition~\ref{pro.bias.multidim.H3} with $\cU = \cUurt{\prof}$, under assumptions \eqref{hyp.s-3-fois-derivable.alt} and \eqref{hyp.unif}: 
\begin{align*}
  &\qquad \absj{ \mathcal{M}_{2,\cU,x} - \frac{1}{2} \paren{1 - \frac{1}{2 \dimX}}^{\prof} \sum_{i=1}^{\dimX} \croch{ \frac{\partial^2 s}{\partial x_i^2} (x) x_i (1-x_i) } }
  \\
  &\leq \frac{1}{6} \paren{1 - \frac{2}{3 \dimX}}^{\prof} \sum_{i=1}^{\dimX} \absj{ \frac{\partial^2 s}{\partial x_i^2} (x) \paren{ x_i^2 + (1-x_i)^2 - 4 x_i (1-x_i) }  } 
  \\
  &\qquad + \frac{1}{8} \paren{1 - \frac{1}{\dimX}}^{\prof} \sum_{1 \leq i \neq j \leq \dimX} \absj{ \frac{\partial^2 s}{\partial x_i \partial x_j} (x) (2 x_i - 1) (2 x_j - 1) }
  \\
  &\leq \frac{\dimX}{6} \paren{1 - \frac{2}{3 \dimX}}^{\prof} \max_i \absj{\frac{\partial^2 s}{\partial x_i^2} (x)} 
  + 
  \frac{\dimX^2 - \dimX}{8} \paren{1 - \frac{1}{\dimX}}^{\prof} \max_{i \neq j} \absj{\frac{\partial^2 s}{\partial x_i \partial x_j} (x) } 
  \\
  &\leq \frac{\dimX^2}{6} \paren{1 - \frac{2}{3 \dimX}}^{\prof} \max_{i,j} \absj{\frac{\partial^2 s}{\partial x_i \partial x_j} (x)} 
  \\
  \mathcal{R}_{3,\cU,x} &\leq  \CHtroisa \dimX  \sqrt{2} \paren{1 - \frac{2}{3 \dimX}}^{3 \prof / 4}
  \enspace .
\end{align*}
Then, Eq.~\eqref{eq.pro.bias.multidim.Bias-H3} yields
\begin{align*}
  \Biasinfty{\cU}(x) 
  &\leq 
  \frac{1}{4} \paren{1 - \frac{1}{2 \dimX}}^{2 \prof} \paren{  \nabla s(x) \cdot (1-2x) + \sum_{i=1}^{\dimX} \croch{ \frac{\partial^2 s}{\partial x_i^2} (x) x_i (1-x_i) } }^2 \\
  &\qquad + 
  \paren{1 - \frac{1}{2 \dimX}}^{\prof} \paren{1 - \frac{2}{3 \dimX}}^{\prof}   \frac{\dimX^2}{6}
  \\ &\qquad \qquad 
  \times \paren{  \norm{\nabla s(x)}_2 \sqrt{\dimX} + \frac{\dimX}{4} \max_i \absj{ \frac{\partial^2 s}{\partial x_i^2} (x) } }  \max_{i,j} \absj{\frac{\partial^2 s}{\partial x_i \partial x_j} (x)} 
  \\
  &\qquad + \frac{\dimX^4}{36}   \paren{1 - \frac{2}{3 \dimX}}^{2 \prof} \max_{i,j} \paren{\frac{\partial^2 s}{\partial x_i \partial x_j} (x)}^2  
  \\
  &\qquad + \CHtroisa \dimX  \sqrt{2} \paren{1 - \frac{2}{3 \dimX}}^{3 \prof / 4}  \paren{1 - \frac{1}{2 \dimX}}^{\prof}
  \paren{  \norm{\nabla s(x)}_2 \sqrt{\dimX} + \frac{\dimX}{4} \max_i \absj{ \frac{\partial^2 s}{\partial x_i^2} (x)  } }
  \\
  &\qquad + 2 \CHtroisa^2 \dimX^2   \paren{1 - \frac{2}{3 \dimX}}^{3 \prof / 2} 
  \\
  &\leq 
  \frac{1}{4} \paren{1 - \frac{1}{2 \dimX}}^{2 \prof} \paren{  \nabla s(x) \cdot (1-2x) + \sum_{i=1}^{\dimX} \croch{ \frac{\partial^2 s}{\partial x_i^2} (x) x_i (1-x_i) } }^2 \\
  &\qquad + 
  6 \dimX^4 \paren{1 - \frac{2}{3 \dimX}}^{3 \prof / 4}  \paren{1 - \frac{1}{2 \dimX}}^{\prof}  \paren{  \norm{\nabla s(x)}_2^2  +   \max_{i,j} \paren{\frac{\partial^2 s}{\partial x_i \partial x_j} (x)}^2  
    +\CHtroisa^2  }
\end{align*}
and a similar proof gives the corresponding lower bound on $\Biasinfty{\cU}(x)$, which proves Eq.~\eqref{eq.bias.BPRF.H3.Biasinfty}. 
Then, integrating it over $x \in [0,1)^{\dimX}$ proves Eq.~\eqref{eq.bias.BPRF.H3.Biasinfty.integrated}. 
\qed

\subsection{Proof of Lemma~\ref{le.BPRF}} \label{sec.pr.le.BPRF}
\subsubsection*{Proof of Eq.~\eqref{eq.le.BPRF.somme-diametres}}
Let $(\bU_{\prof})_{\prof \in \N}$ be a random sequence of partitions of $\X$ as in Section~\ref{sec.BPRF.description}. 
We prove Eq.~\eqref{eq.le.BPRF.somme-diametres}  by induction on $\prof$. 
It clearly holds for $\prof=0$. 
Then, assuming Eq.~\eqref{eq.le.BPRF.somme-diametres} holds true for some $\prof \in \N$, 
  \begin{align*}
    \E \croch{ \sum_{\lambda \in \bU_{\prof+1}} \paren{ \diam_{L^2}(\lambda) }^2 \sachant \bU_{\prof}}
    &= 
    \sum_{\lambda \in \bU_{\prof}} \E \croch{ \sum_{\mu \in \bU_{\prof+1} \, , \, \mu \subset \lambda } \paren{ \diam_{L^2}(\mu) }^2 \sachant \bU_{\prof}}
    \\
    &= \sum_{\lambda \in \bU_{\prof}} \E \croch{ \paren{ \diam_{L^2}(\lambda^-) }^2  + \paren{ \diam_{L^2}(\lambda^+) }^2 \sachant \lambda }
  \end{align*}
  where for any $\lambda \in \bU_{\prof}$, we denote by $\lambda^-$ and $\lambda^+$ the two elements of $\bU_{\prof+1}$ contained in $\lambda$, and we used that at step $\prof$, the way each $\lambda \in \bU_{\prof}$ is split only depends on $\lambda$. 
  Now, for any $\lambda = \lambda_1 \times \cdots \times \lambda_{\dimX} \in \bU_{\prof}$, $(\lambda^- , \lambda^+)$ are obtained by choosing a random direction $J \sim \mathcal{U}(\set{1, \ldots, \dimX}$ and by splitting $\lambda_J$ into $(\lambda_J^- , \lambda_J^+)$, while the $(\lambda_\prof)_{\prof \neq J}$ are kept unchanged. 
  Since 
  \begin{equation} \label{eq.pr.le.BPRF.somme-diametres} \paren{ \diam_{L^2}(\lambda) }^2 = \sum_{j=1}^{\dimX} \paren{ \diam_{L^2}(\lambda_j) }^2 \enspace , \end{equation}
  changing $\paren{ \diam_{L^2}(\lambda) }^2$ into $\paren{
    \diam_{L^2}(\lambda^-) }^2 + \paren{ \diam_{L^2}(\lambda^+) }^2$
  amounts to multiply $(\dimX-1)$ terms of the sum in
  Eq.~\eqref{eq.pr.le.BPRF.somme-diametres} by 2, while the last one
  is multiplied by $U^2 + (1-U)^2$ for some uniform random variable
  $U$.
  Since $\E\scroch{ U^2 + (1-U)^2 } = 2/3$, we get 
  \begin{align*}
    \E \croch{ \paren{ \diam_{L^2}(\lambda^-) }^2  + \paren{ \diam_{L^2}(\lambda^+) }^2 \sachant \lambda } 
    &= 
    \paren{ 2 \paren{ 1 - \frac{1}{\dimX}} + \frac{1}{\dimX} \frac{2}{3} } \paren{ \diam_{L^2}(\lambda) }^2
    \\
    &= 2 \paren{ 1 - \frac{2}{3 \dimX}} \paren{ \diam_{L^2}(\lambda) }^2
  \end{align*}
  hence 
  \[ 
  \E \croch{ \sum_{\lambda \in \bU_{\prof+1}} \paren{ \diam_{L^2}(\lambda) }^2 } 
  = 2 \paren{ 1 - \frac{2}{3 \dimX}} \E\croch{ \sum_{\lambda \in \bU_{\prof}} \paren{ \diam_{L^2}(\lambda) }^2 }
  \]
  and Eq.~\eqref{eq.le.BPRF.somme-diametres} holds for $\prof+1$, which ends the proof. 
  \qed 

\subsubsection*{Proof of Eq.~\eqref{eq.BPRF.somme-exp-pl.1}--\eqref{eq.BPRF.somme-exp-pl.2}} 
Let $\bU = \set{\lambda_{1,1}, \ldots, \lambda_{2^\prof,\prof}} \sim
  \cUurt{\prof}$ be as in the definition of the BPRF model.  Then, for
  every $j \in \set{1, \ldots, 2^\prof}$, the volume of
  $\lambda_{j,p}$ can be written as the product $Z_1 \times \cdots
  \times Z_{\prof}$, where for every $i \in \set{1, \ldots, {\prof}}$,
  $Z_i \in \set{ Z_{1,\prof} , 1-Z_{1,\prof}, \ldots, Z_{2^{\prof},\prof}, 1 - Z_{2^{\prof},\prof} }$. 
Thus, $Z_1, \ldots, Z_{\prof}$ are independent with uniform distribution on $[0,1]$ and 
\[ 
| \lambda_{j,{\prof}} | \egalloi Z_1 \times \cdots \times Z_{\prof} 
\enspace .\]
Now, let us write $V_{\prof} = Z_1 \times \cdots \times Z_{\prof}  $. 
For every $\alpha \geq 0$, since $V_{\prof} \geq 0$ a.s., 
\begin{align}
\notag 
\E\croch{ \exp(-\nobs V_{\prof})} 
&= \E\croch{ \exp(-\nobs V_{\prof}) \un_{V_{\prof} \leq \alpha} + \exp(-\nobs V_{\prof}) \un_{V_{\prof} > \alpha}}
\\
\notag 
&\leq \P(V_{\prof} \leq \alpha) + e^{-\nobs \alpha} \P(V_{\prof}>\alpha)  
\\
\label{le.BPRF.pour-minor-err-estim.eq1}
&= \P(V_{\prof} \leq \alpha) \croch{ 1 - e^{-\nobs \alpha}} + e^{-n \alpha}
\enspace . 
\end{align}
In particular, for any $\alpha,\beta >0$ such that $\P(V_{\prof} \leq \alpha) \leq \beta$, 
\begin{equation}
\label{le.BPRF.pour-minor-err-estim.eq2}
\E\croch{ \exp(-\nobs V_{\prof})}  
\leq \beta \croch{ 1 - e^{-\nobs \alpha}} + e^{-\nobs \alpha} 
= \beta + \paren{ 1 - \beta} e^{-\nobs \alpha} 
\enspace . 
\end{equation}
What remains is to upper bound $\P(V_{\prof} \leq \alpha)$. 

Remark that $-\log(V_{\prof}) = -\log(Z_1) - \cdots - \log(Z_{\prof})$ is the sum of $\prof$ independent random variables with an exponential distribution of parameter~1. 
In particular, $-\log(Z_i)$ has an expectation~1 and a variance~1, so that 
\[ 
\E\croch{ -\log(V_{\prof}) } = \prof
\quad \mbox{and} \quad 
\var\paren{ -\log(V_{\prof}) } = \prof 
\enspace . 
\]
Then, by Bienaym\'e-Chebyshev's inequality, for every $t>0$, 
\[ 
\P\paren{ -\log(V_{\prof}) \geq \prof + t } 
\leq \P \paren{ \absj{ -\log(V_{\prof}) - \prof} \geq t} 
\leq \frac{\prof}{t^2}
\enspace , 
\]
hence for every $u >0$, 
\begin{equation}
\label{le.BPRF.pour-minor-err-estim.eq3}
\P\paren{ V_{\prof} \leq \exp \paren{ -\prof - \sqrt{u \prof}} } \leq \frac{1}{u}
\enspace . 
\end{equation}
Combining Eq. \eqref{le.BPRF.pour-minor-err-estim.eq3} and~\eqref{le.BPRF.pour-minor-err-estim.eq2}, we get, for every $u>0$, 
\begin{align}
\notag 
\E\croch{ \exp(-\nobs V_{\prof})}  
&\leq \frac{1}{u} + \paren{ 1 - \frac{1}{u}} \exp\paren{-\nobs \exp \paren{ -\prof - \sqrt{u \prof}}} 
\enspace . 
\end{align}
Eq.~\eqref{eq.BPRF.somme-exp-pl.1} follows since 
\begin{align*}
\E_{\bU \sim \bU^\prof} \croch{ \sum_{\lambda \in \bU} \exp\paren{- \nobs |\lambda|} }
&= 2^\prof \E \croch{ \exp(-\nobs V_{\prof}) }
\end{align*}
and Eq.~\eqref{eq.BPRF.somme-exp-pl.2} follows from Eq.~\eqref{eq.BPRF.somme-exp-pl.1} by taking $u=5$. 
\qed

\section{Technical lemmas} \label{sec.pr.technical}

\begin{lemma}\label{le.opt-risk-bis}
  Let $a,b>0$, $c \geq 0$, $\varepsilon \in ]0,1/2[$ and $n\in\N^*$ such that $n \geq
  \max \set{ \frac{b}{a\varepsilon^5} , \paren{2^5 a/b}^{1/4} }$.
  Then, 
  \[ \inf_{1/\varepsilon \leq k \leq n, k\in \N} \set{ \frac{a}{k^4} + \frac{bk}{n} +
    c \frac{k}{n} e^{-n/k} }
  \leq 3 a^{1/5} \paren{ \frac{b}{n} }^{4/5} \paren{ 1 +
    \frac{4 c a^{1/5}}{3 n^{4/5} b^{6/5}} }
\]
\end{lemma}
\begin{proof}[Proof of Lemma~\ref{le.opt-risk-bis}]
Let $x^* = (an/b)^{1/5}$. 
We have $2 \leq 1/\varepsilon \leq x^* \leq n/2 $ and let $k^* \in [x^*, 2 x^*]$ some integer. 
Then, since $e^{-x}/x \leq 1/x^{2}$ for all $x\geq 1$, 
  \begin{align*}
    \inf_{1/\varepsilon \leq k \leq n, k\in \N} \set{ \frac{a}{k^4} + \frac{bk}{n} +
      c \frac{k}{n} e^{-n/k} } 
    & \leq \frac{a}{{k^*}^4} + \frac{bk^*}{n} +
    c \paren{\frac{k^*}{n}}^2 \\
    & \leq \frac{a}{{x^*}^4} + \frac{b 2 x^*}{n} +
    4 c \paren{\frac{x^*}{n}}^2 \\
    & \leq 3 a^{1/5} \paren{ \frac{b}{n} }^{4/5} + 4 c \paren{\frac{a}{b}}^{2/5} \frac{1}{n^{8/5}}
    \\
    & \leq 3 a^{1/5} \paren{ \frac{b}{n} }^{4/5} \croch{ 1 + \frac{4 c a^{1/5}}{3 b^{6/5} n^{4/5}} }
    \enspace . 
\end{align*}
\end{proof}

\begin{lemma} \label{le.opt-risk}
  Let $a,b, \alpha >0$ and $n\in\N^*$ such that $n \geq \max \set{ \frac{b}{a\alpha}, \paren{
      \frac{a\alpha 2^{\alpha+1}}{b} }^{1/\alpha}} $. 
  Then, 
  \[ \inf_{x \in (0,+\infty)} \set{ a x^{-\alpha} + \frac{bx}{n}} = L_1(\alpha)
  a^{1/(\alpha + 1)} \paren{ \frac{b}{n} }^{\alpha/(\alpha+1)} \]
  where $L_1(\alpha) = \paren{\alpha^{-\alpha/(\alpha+1)} +
    \alpha^{1/(\alpha+1)}}$ ,
  and
  \begin{eqnarray*}
    L_1(\alpha) a^{1/(\alpha + 1)} \paren{ \frac{b}{n} }^{\alpha/(\alpha+1)}
    & \leq & \inf_{1 \leq k \leq n \, , \, k \in \N} \set{ a
      k^{-\alpha} + \frac{bk}{n} } \\
    & \leq & \inf_{1 \leq k=2^{\ell} \leq n \, , \, \ell \in \N} \set{ a k^{-\alpha} + \frac{bk}{n}}
    \leq 2 L_1(\alpha) a^{1/(\alpha + 1)} \paren{ \frac{b}{n} }^{\alpha/(\alpha+1)}
    \enspace .
\end{eqnarray*}
\end{lemma}
\begin{proof}[Proof of Lemma~\ref{le.opt-risk}]
Let $f:(0,+\infty) \to \R$ be defined by $f(x) = a x^{-\alpha} +
  \frac{bx}{n}$ for every $x>0$. 
The function $f$ is convex, differentiable on $(0,+\infty)$ and
$f'(x) = - a \alpha x^{-(\alpha +1)} + \frac{b}{n}$ for every $x>0$. 
So, the infimum of
  $f$ on $(0,+\infty)$ is reached for $x^* = \paren{ \frac{a \alpha
      n}{b} }^{1/(\alpha + 1)}$ and 
the value of $f(x^*)$ follows from straightforward computations. 
The condition on $n$ ensures that $1 \leq x^* \leq n/2$, so, 
\[ 
\inf_{x \in (0,+\infty)} \set{ a x^{-\alpha} + \frac{bx}{n}}
\leq \inf_{1 \leq k \leq n \, , \, k \in \N} \set{ a k^{-\alpha} +
    \frac{bk}{n}}
  \leq \inf_{1 \leq k=2^{\ell} \leq n \, , \, \ell \in \N} \set{ a k^{-\alpha} + \frac{bk}{n}} 
  \enspace . \]
Finally,
  \[ \inf_{1 \leq k=2^{\ell} \leq n \, , \, \ell \in \N} \set{ a k^{-\alpha} + \frac{bk}{n}}
  \leq f \paren{ 2^{\ell^*} } \]
  where $\ell^* \in \N$ is such that $x^* \leq 2^{\ell^*} \leq 2 x^*$; such an $\ell^*$ exists
  since $x^* \geq 1$ and $2x^* \leq n$. 
Hence,
  \[ \inf_{1 \leq k=2^{\ell} \leq n \, , \, \ell \in \N} \set{ a k^{-\alpha} + \frac{bk}{n}}  
  \leq \sup_{\lambda \in [1, 2]} f(\lambda x^*)
  \leq 2 f(x^*) \]
  and the last upper-bound follows.

\end{proof}

\begin{lem}
  \label{le.suites-geom}
  Let $k \geq 1$ be an integer and $\alpha,\beta_1, \ldots, \beta_k, \gamma_1, \ldots, \gamma_k \in \R$ be such that 
  \[ \forall i \in \set{1, \ldots, k} \, , \quad \gamma_i \neq \alpha \enspace . \]
  Then, the sequence $(u_n)_{n \geq 0}$ defined by $u_0 \in \R$ and 
  \[ \forall n \in \N \, , \quad u_{n+1} = \alpha u_n + \sum_{i=1}^k \beta_i \gamma_i^n \]
  satisfies 
  \[ \forall n \in \N \, , \quad u_n = \alpha^n \paren{ u_0 - \sum_{i=1}^k \frac{\beta_i}{\gamma_i - \alpha} } + \sum_{i=1}^k \croch{ \frac{\beta_i}{\gamma_i - \alpha} \gamma_i^n } \enspace . \]
\end{lem}
\begin{proof}[Proof of Lemma~\ref{le.suites-geom}]
  Let us consider the sequence defined by 
  \[ \forall n \in \N \, , \quad v_n = u_n - \sum_{i=1}^k \frac{\beta_i}{\gamma_i - \alpha} \gamma_i^n \enspace . \]
  Then, by definition of $u_n$, we have 
  \begin{align*}
    \forall n \in \N \, , \quad v_{n+1} &= \alpha u_n + \sum_{i=1}^k \beta_i \gamma_i^n - \sum_{i=1}^k \frac{\beta_i}{\gamma_i - \alpha} \gamma_i^{n+1}
    \\ & = \alpha v_n + \alpha \sum_{i=1}^k \frac{\beta_i}{\gamma_i - \alpha} \gamma_i^n + \sum_{i=1}^k \beta_i \gamma_i^n - \sum_{i=1}^k \frac{\beta_i}{\gamma_i - \alpha} \gamma_i^{n+1}
    \\ &= \alpha v_n
  \end{align*}
  so that $v_n = \alpha^n v_0$ and the result follows. 
\end{proof}

\end{document}